\documentclass[11pt]{amsart}

\usepackage{amssymb,amsmath}
\usepackage{enumitem}
\usepackage{graphicx}
\usepackage{verbatim}
\usepackage{multirow}
\usepackage{framed}
\usepackage{color}
\usepackage[normalem]{ulem}
\usepackage{hyperref}

\makeatletter
\def\blfootnote{\gdef\@thefnmark{*}\@footnotetext}
\makeatother

\usepackage[left=1.26in,right=1.26in,top=1.65in,bottom=1.3in]{geometry}
\newtheorem{theorem}{Theorem}[section]
\newtheorem{lemma}[theorem]{Lemma}
\newtheorem{definition}[theorem]{Definition}
\newtheorem{proposition}[theorem]{Proposition}
\newtheorem{corollary}[theorem]{Corollary}

\newtheorem{remark}[theorem]{Remark}
\newtheorem{example}[theorem]{Example}

\DeclareMathOperator{\diam}{diam}
\DeclareMathOperator{\vol}{vol}

\def \L{\Lambda}

\newcommand{\ttt}{q}
\newcommand{\bpsi}{\overline{\psi}}
\newcommand{\cs}{{cs}}
\newcommand{\Ecs}{E^{cs}}
\newcommand{\C}{\mathcal{C}}
\newcommand{\CB}{C_{\mathrm{B}}}

\newcommand{\RR}{\mathbb{R}}
\newcommand{\NN}{\mathbb{N}}
\newcommand{\ph}{\varphi}
\newcommand{\eps}{\epsilon}

\newcommand{\htop}{h_\mathrm{top}}
\newcommand{\Vl}{V_\mathrm{loc}}

\newcommand{\ulim}{\varlimsup}
\newcommand{\llim}{\varliminf}
\newcommand{\Zspan}{Z^{\mathrm{span}}}
\newcommand{\Zsep}{Z^{\mathrm{sep}}}

\newcommand{\vv}{\mathbf{v}}
\newcommand{\ww}{\mathbf{w}}
\newcommand{\Lcs}{L}
\newcommand{\Kcs}{K^{cs}}
\newcommand{\Ku}{K^u}

\DeclareMathOperator{\Int}{int}

\DeclareMathOperator{\Jac}{Jac}

\numberwithin{figure}{section}
\numberwithin{equation}{section}

\title{Equilibrium measures for some partially hyperbolic systems}
\author{Vaughn Climenhaga }
\address{Department of Mathematics \\ University of Houston \\ Houston, TX 77204, USA}
\email{climenha@math.uh.edu}
\urladdr{https://www.math.uh.edu/~climenha/}
\author{Yakov Pesin}
\address{Department of Mathematics \\ Pennsylvania State University \\ University Park, PA 16802, USA}
\email{pesin@math.psu.edu}
\urladdr{http://www.math.psu.edu/pesin/}
\author{Agnieszka Zelerowicz}
\address{Department of Mathematics \\ University of Maryland \\ College Park, MD 20742, USA}
\email{azelerow@umd.edu}

\begin{document}

\date{\today}

\begin{abstract}
We study thermodynamic formalism for topologically transitive partially hyperbolic systems in which the center-stable bundle satisfies a bounded expansion property, and show that every potential function satisfying the Bowen property has a unique equilibrium measure.  Our method is to use tools from geometric measure theory to construct a suitable family of reference measures on unstable leaves as a dynamical analogue of Hausdorff measure, and then show that the averaged pushforwards of these measures converge to a measure that has the Gibbs property and is the unique equilibrium measure.
\end{abstract}

\thanks{V.~C.\ is partially supported by NSF grant DMS-1554794.  Ya. P. and A. Z. were partially supported by NSF grant DMS-1400027.}

\maketitle

\setcounter{tocdepth}{1}
\tableofcontents



\part{Results and applications}

\section{Introduction}

Consider a dynamical system $f\colon X\to X$, where $X$ is a compact metric space and $f$ is continuous.  Given a continuous \emph{potential function} $\ph\colon X\to \RR$, the \emph{topological pressure} of $\ph$ is the supremum of $h_\mu(f) + \int\ph\,d\mu$ taken over all $f$-invariant Borel probability measures on $X$, where $h_\mu$ denotes the measure-theoretic entropy.  A measure achieving this supremum is called an \emph{equilibrium measure}, and one of the central questions of thermodynamic formalism is to determine when $(X,f,\ph)$ has a unique equilibrium measure.

When $f\colon M\to M$ is a diffeomorphism and $X\subset M$ is a topologically transitive locally maximal hyperbolic set, so that the tangent bundle splits as $E^u\oplus E^s$, with $E^u$ uniformly expanded and $E^s$ uniformly contracted by $Df$, it is well-known that every H\"older continuous potential function has a unique equilibrium measure \cite{Bow08}.  In the specific case when $X$ is a hyperbolic attractor and 
$\ph=-\log|\det(Df|E^u)|$ is the \emph{geometric potential}, this equilibrium measure is the unique \emph{Sinai--Ruelle--Bowen (SRB) measure}, which is also characterized by the fact that its conditional measures along unstable leaves are absolutely continuous with respect to leaf volume.

In this paper we study the case when uniform hyperbolicity is replaced by partial hyperbolicity.  Here the analogue of SRB measures are the \emph{$u$-measures}\footnote{In \cite{PS82} these are called ``$u$-Gibbs measures''; here we use the terminology from \cite{PC10}.} constructed by the second author and Sinai in \cite{PS82} using a geometric construction based on pushing forward leaf volume and averaging.  We follow this approach, replacing leaf volume with a family of reference measures defined using a Carath\'eodory dimension structure.  We give conditions on $f$ and 
$\ph$ under which the averaged pushforwards of these reference measures converge to the unique equilibrium measure, and describe various examples that satisfy these conditions. Roughly speaking, our conditions are
\begin{itemize}
\item the map $f$ is topologically transitive and partially hyperbolic with an integrable center-stable bundle along which expansion under $f^n$ is uniformly bounded independently of $n$ (``Lyapunov stability''), and 
\item the potential $\ph$ satisfies a leafwise \emph{Bowen property} that uniformly bounds the difference between Birkhoff sums along nearby trajectories, independently of the trajectory length.
\end{itemize}
See \S\ref{def} and \S\ref{assumptions} for precise statements of the conditions. The reference measures are defined in \S\ref{car-struc}, and our main results appear in \S\ref{sec:main-results}. We highlight some examples to which our results apply (see \S\ref{sec:applications} for more):
\begin{itemize}
\item time-$1$ maps of Anosov flows, with H\"older potentials given by integrating some function along the flow through time $1$;
\item frame flows with similar `averaged' potentials that are constant on fibers $SO(n-1)$.
\end{itemize}
We also refer to the survey paper \cite{CPZ} for a general overview of how our techniques work in the uniformly hyperbolic setting.

Before giving precise definitions we stop to recall some known results from the literature.  In uniform hyperbolicity, the general existence and uniqueness result can be established by various methods; most relevant for our purposes are the approaches that proceed by building reference measures on stable and/or unstable leaves, which take the place of leaf volume when the potential is not geometric.  Such measures were first constructed by Sinai (in discrete time) \cite{yS68} and Margulis (in continuous time) \cite{gM70} when $\ph=0$.  In the setting of uniform hyperbolicity, closest to our approach is the work of Hamenst\"adt \cite{uH89} for geodesic flows; see Hasselblatt \cite{bH89} for the Anosov flow case, and \cite{uH97} for an extension to nonzero potential functions.  In these papers the leaf measures are constructed as Hausdorff measure for an appropriate metric, which is similar to our construction in \S\ref{car-struc} below.  A different construction based on Markov partitions can be found in work of Haydn \cite{nH94} and Leplaideur \cite{Lep}.

Now we briefly survey some known results on thermodynamic formalism in partial hyperbolicity. The first remark is that whenever $f$ is $C^\infty$, the entropy map $\mu\mapsto h_\mu(f)$ is upper semi-continuous \cite{sN89} and thus existence of an equilibrium measure is guaranteed by weak*-compactness of the space of $f$-invariant Borel probability measures; however, this nonconstructive approach does not address uniqueness or describe how to produce an equilibrium measure.

Even without the $C^\infty$ assumption, the expansivity property would be enough to guarantee that the entropy map is upper semi-continuous, and thus gives existence; moreover, for expansive systems the construction in the proof of the variational principle \cite[Theorem 9.10]{pW82} actually produces an equilibrium measure. Partially hyperbolic systems are not expansive in general, but when the center direction is one-dimensional, they are \emph{entropy-expansive} \cite[\S5.3]{CY05}; this property, introduced by Bowen in \cite{rB72}, also suffices to guarantee that the standard construction produces an equilibrium measure, and continues to hold when the center direction admits a dominated splitting into one-dimensional sub-bundles \cite{EnExp}. On the other hand, when the center direction is multi-dimensional and admits no such splitting, there are (many) examples with positive tail entropy, for which the system is not even asymptotically entropy-expansive; such examples were constructed 
in \cite{DF11,BF13} following ideas from \cite{DN05}.



For results on SRB measures in partial hyperbolicity, we refer to \cite{BV00,ABV00,BDP02,BDPP08,CY05}. For the broader class of equilibrium measures, existence and uniqueness questions have been studied for certain classes of partially hyperbolic systems. In general one should not expect uniqueness to hold without further conditions; see \cite{HHTU12} for an open set of topologically mixing partially hyperbolic diffeomorphisms in three dimensions with more than one \emph{measure of maximal entropy (MME)} -- that is, multiple equilibrium measures for the potential 
$\ph=0$. Some results on existence and uniqueness of an MME are available when the partially hyperbolic system is semi-conjugated to the uniformly hyperbolic one; see \cite{BFSV12,rU12}. For some partially hyperbolic systems obtained by starting with an Anosov system and making a perturbation that is $C^1$-small except in a small neighborhood where it may be larger and is given by a certain bifurcation, uniqueness results can be extended to a class of nonzero potential functions \cite{CFT,CFT2}.

The largest set of results is available for the examples known as ``partially hyperbolic horseshoes'': existence of equilibrium measures was proved by Leplaideur, Oliveira, and Rios \cite{LOR11}; examples of rich phase transitions were given by D\'iaz, Gelfert, and Rams \cite{DGR11,DG12,DGR14}; and uniqueness for certain classes of H\"older continuous potentials was proved by Arbieto and Prudente \cite{AP12} and Ramos and Siqueira \cite{RS17}. A related class of partially hyperbolic skew-products with non-uniformly expanding base and uniformly contracting fiber was studied by Ramos and Viana \cite{RV17}.  We point out that our results study a class of systems, rather than specific examples, and that we establish uniqueness results, rather than the phase transition results that have been the focus of much prior work.

Another class of partially hyperbolic examples is obtained by considering the time-$1$ map of an Anosov flow; see \S\ref{sec:app}, where we describe two arguments, one based on our main result, and one using a general argument communicated to us by F.\ Rodriguez Hertz for deducing uniqueness for the map from uniqueness for the flow.  We also study the time-$1$ map for frame flows in negative curvature, which are partially hyperbolic. In this latter setting, equilibrium measures for the flow were recently studied by Spatzier and Visscher \cite{SV}, but the general argument for deducing uniqueness for the map from uniqueness for the flow (see \S\ref{sec:time1}) may not apply. In both settings, the class of potential functions to which our results apply includes all scalar multiples of the geometric potential, whose equilibrium measures are precisely the $u$-measures from \cite{PS82}.

In \S\ref{def} we give background definitions and describe the classes of systems we will study. In \S\ref{car-struc} we recall the general notion of a Carath\'eodory dimension structure and use it to define a family of reference measures on unstable leaves.  In \S\ref{sec:main-results} we describe the class of potential functions that we will consider, and formulate our main results. In \S\ref{sec:particular-potentials} we apply our results to two particular potentials: $\ph=0$, which gives existence and uniqueness of the MME, and the \emph{geometric potential} $\ph=-\log|\det(Df|E^u)|$, which gives existence and uniqueness of the SRB measure (assuming the case of a partially hyperbolic attractor).
In fact, our results apply to the one-parameter family of \emph{geometric $q$-potentials} $\ph_q=-q\log|\det(Df|E^u)|$ for all $q\in\mathbb{R}$. In \S\ref{sec:app} we apply our results to some particular dynamical systems: the time-$1$ map of an Anosov flow, the time-$1$ map of the frame flow, and partially hyperbolic diffeomorphisms whose central foliation is absolutely continuous and has compact leaves The proofs are given in \S\S\ref{sec:measures}--\ref{sec:proof-of-main}. In Appendix \ref{appendix} we gather some proofs of background results that we do not claim are new, but which seemed worth proving here for completeness.


\subsection*{Acknowledgement} 
This work had its genesis in workshops at ICERM (Brown University) and ESI (Vienna) in March and April 2016, respectively.  We are grateful to both institutions for their hospitality and for creating a productive scientific environment. We are also grateful to the anonymous referee for a careful reading and for multiple comments that improved the exposition.

\section{Preliminaries}\label{def}
  
\subsection{Partially hyperbolic sets}\label{sec:def-ph}
Let $M$ be a compact smooth connected Riemannian manifold, $U\subset M$ an open set and $f\colon U\to M$ a diffeomorphism onto its image. 
Let $\L\subset U$ be a compact $f$-invariant subset on which $f$ is  \emph{partially hyperbolic in the broad sense}:
 that is
\begin{itemize}
\item the tangent bundle over $\L$ splits into two invariant and continuous subbundles $T_{\L}M = \Ecs \oplus E^u$;
\item there is a Riemannian metric $\|\cdot\|$ on $M$ and numbers $0<\nu<\chi$ with $\chi>1$ such that for every $x\in\L$
\begin{equation}\label{dominated-splitting}
\begin{aligned}
\|Df_x v\|\le\nu\|v\| & \text{ for } v\in \Ecs(x),\\
\|Df_x v\|\ge\chi\|v\| & \text{ for } v\in E^u(x).
\end{aligned}
\end{equation}
\end{itemize}

\begin{remark}
One could replace the requirement that $\chi>1$ with the condition that $\nu<1$ (and make the corresponding edit to \ref{C1} below); to apply our results in this setting it suffices to replace $f$ with $f^{-1}$, and so we limit our discussion to the case when $E^u$ is uniformly expanding.
\end{remark}

The above splitting of the tangent bundle is only defined over $\Lambda$ and may not extend to $U$ as a continuous and invariant splitting.  However, there are continuous cone families $\Kcs$ and $\Ku$ defined on all of $U$ such that $\Ku$ is $Df$-invariant and $\Kcs$ is $Df^{-1}$-invariant.  A curve $\gamma$ in $U$ will be called a \emph{$cs$-curve} if all its tangent vectors lie in $\Kcs$; define a \emph{$u$-curve} similarly.  Our main results in \S\ref{sec:main-results} will require the following two conditions (among others):
\begin{enumerate}[label=\upshape{(C\arabic{*})}]
\item\label{C1} for every $\theta>0$ there is $\delta>0$ such that if $\gamma$ is a curve in $U$ with length $\leq \delta$, and $n\geq 0$ is such that $f^n\gamma$ is a $cs$-curve in $U$, then the length of $f^n\gamma$ is $\leq \theta$;
\item\label{C2} $f|\L$ is topologically transitive.
\end{enumerate}

\begin{remark}\label{rmk:Lyap-stab}
Condition \ref{C1} is called \emph{Lyapunov stability} in \cite{HHU} and is related to the condition of \emph{topologically neutral center} in \cite{BZ}. It is automatically satisfied if $\nu\leq 1$, where $\nu$ is the number in \eqref{dominated-splitting}.
More generally, \ref{C1} holds if there is $\Lcs>0$ such that for every $x\in \L$, $v\in \Ecs(x)$, and $n\in \NN$, we have $\|Df_x^n v\| \leq \Lcs \|v\|$.  However, \ref{C1} can be true even if this condition fails; see \cite[Proposition 2.3]{BZ} for an example, using ideas from \cite{BCW09}. 
See Example \ref{eg:no-C1} for an illustration of how our results can fail if Conditions \ref{C2} and \ref{C3} (below) are satisfied but \ref{C1} does not hold.
\end{remark}


Since the distributions $\Ecs$ and $E^u$ are continuous on $\L$, there is $\kappa>0$ such that $\angle (E^u(x),\Ecs(x))>\kappa$ for all $x\in\L$; we also have $p\geq 1$ such that $\dim E^u(x)=p$ and 
$\dim \Ecs(x)=\dim M-p$ for all $x\in\L$.

\subsection{Local product structure and rectangles}\label{sec:lps}

The following well known result describes existence and some properties of local unstable manifolds;
for a proof, see \cite[\S4]{P} or \cite[Theorem 6.2.8 and \S6.4]{Kat}.

\begin{proposition}\label{prop:local-mfds}
There are numbers $\tau>0$, $\lambda \in (\chi^{-1}, 1)$, $C_1>0$, $C_2>0$, and for every $x\in\L$ a $C^{1+\alpha}$ local manifold $\Vl^{u}(x)\subset M$ such that
\begin{enumerate}[label=\textup{(\arabic{*})}]
\item $T_y \Vl^u(x) = E^u(y)$ for every $y\in \Vl^u(x) \cap \L$;
\item $\Vl^u(x)=\exp_x\{v+\psi^u_x(v): v\in B_x^u(0,\tau)\},$
where $B_x^u(0,\tau)$ is the ball in $E^u(x)$ centered at zero of radius $\tau>0$ and 
$\psi^{u}_x\colon B_x^{u}(0,\tau)\to \Ecs(x)$ is a $C^{1+\alpha}$ function;
\item\label{leaves-contract} 
For every $n\ge 0$ and $y\in\Vl^u(x) \cap \L$,
$d(f^{-n}(x),f^{-n}(y))\le C_1\lambda^{n}d(x,y)$;  
\item\label{unif-holder} $|D\psi^{u}_x|_\alpha\le C_2$.
\end{enumerate}
\end{proposition}
The number $\tau$ is the \emph{size} of the local manifolds and will be fixed at a sufficiently small value to guarantee various estimates.  Using \ref{C1}, the arguments in \cite[Theorem 6.2.8 and \S6.4]{Kat} also give the following result for the center-stable direction.

\begin{theorem}\label{thm:Wcs}
Let $f\colon U\to M$ be a diffeomorphism onto its image and $\Lambda\subset U$ a compact $f$-invariant subset admitting a splitting $\Ecs\oplus E^u$ satisfying condition \ref{C1} as above.  Then $\Ecs$ 
can be uniquely integrated to a continuous lamination with $C^1$ leaves, which can be described as follows:
There is $\tau>0$ such that for every $x\in\L$ there is a local manifold $\Vl^\cs(x)\subset M$ satisfying
\begin{enumerate}
\item $T_y \Vl^\cs(x) = \Ecs(y)$ for every $y\in \Vl^\cs(x) \cap \L$;
\item $\Vl^\cs(x)=\exp_x\{v+\psi^\cs_x(v): v\in B_x^\cs(0,\tau)\},$
where $B_x^\cs(0,\tau)$ is the ball in $\Ecs(x)$ centered at zero of radius $\tau>0$ and $\psi^\cs_x\colon B_x^\cs(0,\tau)\to E^u(x)$ is a $C^{1}$ function.
\end{enumerate}
Moreover, there is $r_0>0$ such that for every $y\in B(x,r_0)\cap\L$ we have that $\Vl^u(x)$ and $\Vl^\cs(y)$ intersect at exactly one point which we denote by $[x,y]$.
\end{theorem}

\begin{remark}
The argument in \cite{Kat} constructs the local leaves by first writing a sequence of maps $f_m$ on Euclidean space that (in a neighborhood of the origin) correspond to local coordinates around the orbit of $x$, and then producing the corresponding manifolds for this sequence.  As remarked in the paragraphs following \cite[Theorem 6.2.8]{Kat}, in order to go from the local result to the manifold itself, one needs to know that there is some neighborhood $D_r$ of the origin such that the leaf is ``determined by the action of $f_m$ on $D_r$ only''; see also the ``note of caution'' following \cite[Corollary 6.2.22]{Kat}.  In \cite[Theorem 6.4.9]{Kat} 
this condition is guaranteed by assuming that $\Ecs$ is uniformly contracting, so that points on the local leaf in Euclidean coordinates have orbits staying inside $D_r$, and thus represent true dynamical behavior on $M$.  In our setting, this same fact is guaranteed by \ref{C1}.
\end{remark}

Finally, we require the set $\L$ to have the following product structure, so that the point $[x,y]$  lies in $\L$:
\begin{enumerate}[label=\upshape{(C\arabic{*})}]\setcounter{enumi}{2}
\item\label{C3} $\L=\left(\bigcup_{x\in\L}\,\Vl^u(x)\right)\cap\left(\bigcup_{x'\in\L}\,\Vl^\cs(x')\right)$.
\end{enumerate}

\begin{remark}\label{rmk:unif-hyp}
If the number $\nu$ from \eqref{dominated-splitting} satisfies $\nu<1$, so that the set $\L$ is uniformly hyperbolic and $\Ecs$ is uniformly contracted under $f$, then conditions \ref{C1}--\ref{C3} hold whenever $\L$ is locally maximal for $f$; in particular, our results apply to every topologically transitive locally maximal hyperbolic set.
\end{remark}

Given $x\in \L$ and $r\in (0,\tau)$, we write $B_\L(x,r) = B(x,r) \cap \L$ for convenience.  We also write
\[
B^u(x,r) = B(x,r) \cap \Vl^u(x), \qquad
B_\L^u(x,r) = B^u(x,r) \cap \L,
\]
and similarly with $u$ replaced by $\cs$.

\begin{definition}\label{def:rectangle}
A closed set $R\subset \L$ is called a \emph{rectangle} if $[x,y] = \Vl^u(x) \cap \Vl^\cs(y)$ exists and is contained in $R$ for every $x,y\in R$.  
\end{definition}

One can easily produce rectangles by fixing $x\in \L$, $\delta>0$ sufficiently small, and putting
\begin{equation}\label{rectangle}
R=R(x,\delta):=\bigg(\bigcup_{y\in \overline{B_\L^\cs(x,\delta)}} \Vl^u(y)\bigg)\cap\bigg(\bigcup_{z\in \overline{B_\L^u(x,\delta)}} \Vl^\cs(z)\bigg).
\end{equation}
Note that the intersection of two rectangles is either empty or is itself a rectangle.  The following result is standard: for completeness, we give a proof in Appendix \ref{appendix}.

\begin{lemma}\label{lem:rectangle-partition}
For every $\epsilon>0$ and every Borel measure $\mu$ on $\L$ (whether invariant or not), there is a finite set of rectangles $R_1,\dots, R_N \subset \L$ satisfying the following properties:
\begin{enumerate}[label=\textup{(\arabic{*})}]
\item each $R_i$ is the closure of its (relative) interior;
\item $\L = \bigcup_{i=1}^N R_i$, and the (relative) interiors of the $R_i$ are disjoint;
\item $\mu(\partial R_i) = 0$ for all $i$;
\item $\diam R_i < \epsilon$ for all $i$.
\end{enumerate}
\end{lemma}

We refer to $\mathcal{R} = \{R_1,\dots, R_N\}$ as a \emph{partition by rectangles}.  Note that even though the rectangles may overlap, the fact that the boundaries are $\mu$-null implies that there is a full $\mu$-measure subset of $\L$ on which $\mathcal{R}$ is a genuine partition.  In particular, when $\mu$ is $f$-invariant, we can use a partition by rectangles for the computation of measure-theoretic entropy.

We end this section with one more definition.

\begin{definition}\label{def:holonomy}
Given a rectangle $R$ and points $y,z\in R$, the \emph{holonomy map} $\pi_{yz} \colon V_R^u(y) \to V_R^u(z)$ is defined by 
\begin{equation}\label{eqn:holonomy}
\pi_{yz}(x) = \Vl^u(z) \cap \Vl^\cs(x) = [z,x].
\end{equation}
\end{definition}

The holonomy map is a homeomorphism between $V_R^u(y)$ and $V_R^u(z)$.
One can define a holonomy map between $V_R^\cs(y)$ and $V_R^\cs(z)$ in the analogous way, sliding along unstable leaves; if need be, we will denote this map by $\pi_{yz}^u$ and the holonomy map from Definition \ref{def:holonomy} by $\pi_{yz}^\cs$.  The notation $\pi_{yz}$ without a superscript will always refer to $\pi_{yz}^\cs$.

\subsection{Measures with local product structure}

We recall some facts about measurable partitions and conditional measures; see \cite{Roh52} or \cite[\S5.3]{EW11} for proofs and further details.
Given a measure space $(X,\mu)$, a partition $\xi$ of $X$, and $x\in X$, write $\xi(x)$ for the partition element containing $x$.  The partition is said to be \emph{measurable} if it can be written as the limit of a refining sequence of finite partitions.  In this case there exists a system of conditional measures $\{\mu^\xi_x\}_{x\in X}$ such that:
\begin{enumerate}
\item each $\mu_x^\xi$ is a probability measure on $\xi(x)$;
\item if $\xi(x) = \xi(y)$, then $\mu_x^\xi = \mu_y^\xi$;
\item for every $\psi\in L^1(X,\mu)$, we have
\begin{equation}\label{eqn:conditionals}
\int_X \psi\,d\mu = \int_X \int_{\xi(x)} \psi\,d\mu_x^\xi \,d\mu(x).
\end{equation}
\end{enumerate}
Moreover, the system of conditional measures is unique mod zero: if $\hat\mu_x^\xi$ is any other system of measures satisfying the conditions above, then $\mu_x^\xi = \nu_x^\xi$ for $\mu$-a.e.\ $x$.

We will be most interested in the following example: Given a rectangle $R\subset \L$ and a point $x\in R$, we consider the measurable partition $\xi$ of $R$ by unstable sets of the form
\[
V_R^u(x) = \Vl^u(x) \cap R.
\]
Let $\{\mu_x^u\}_{x\in R}$ denote the corresponding system of conditional measures; given a partition element $V = V_R^u(x)$, we may also write $\mu_V = \mu_x^u$.  
Let $\tilde{\mu}$ denote the factor-measure on $R/\xi$ defined by $\tilde{\mu}(E) = \mu(\bigcup_{V\in E} V)$.  Then for every $\psi\in L^1(R,\mu)$, we have
\begin{equation}\label{gibbs-split}
\int_R\psi\,d\mu=\int_{R/\xi}\int_{V}\psi(z)\,d\mu_V(z)\,d\tilde{\mu}(V). 
\end{equation}
Note that the factor space $R/\xi$ can be identified with the set 
$V_R^\cs(x) = \Vl^\cs(x) \cap R$ for any $x\in R$, so that the measure $\tilde{\mu}$ can be viewed as a measure on this set, in which case \eqref{gibbs-split} becomes
\begin{equation}\label{eqn:R-split}
\int_R \psi\,d\mu = \int_{V_R^\cs(x)} \int_{V_R^u(y)} \psi(z) \,d\mu_y^u(z) \,d\tilde\mu(y).
\end{equation}
Note that the conditional measures $\mu_x^u$ depend on the choice of rectangle $R$, although this is not reflected in the notation.  In fact the ambiguity only consists of a normalizing constant, as the following lemma shows.
\begin{lemma}\label{lem:compare-cond}
Let $\mu$ be a Borel measure on $\L$, and let $R_1,R_2 \subset \L$ be rectangles.  Let $\{\mu_x^1\}_{x\in R_1}$ and $\{\mu_x^2\}_{x\in R_2}$ be the corresponding systems of conditional measures on the unstable sets $V_{R_j}^u(x)$.  Then for $\mu$-a.e.\ $x\in R_1\cap R_2$, the measures $\mu_x^1$ and $\mu_x^2$ are scalar multiples of each other when restricted to $V_{R_1}^u(x) \cap V_{R_2}^u(x)$.
\end{lemma}

See Appendix \ref{appendix} for a proof of Lemma \ref{lem:compare-cond} and the next lemma, which relies on it.

\begin{lemma}\label{lem:cond-inv}
If $\mu$ is an $f$-invariant Borel measure on $\L$, then for $\mu$-a.e.\ $x\in \L$ and any choice of two rectangles containing $x$ and $f(x)$, the corresponding systems of conditional measures are such that $f_* \mu_x^u$ is a scalar multiple of $\mu_{f(x)}^u$ on the intersection of the corresponding unstable sets.
\end{lemma}


For the next definition, we recall that two measures $\nu,\mu$ are said to be \emph{equivalent} if $\nu\ll\mu$ and $\mu\ll\nu$; in this case we write $\nu\sim\mu$.  Also, given a rectangle $R$, a point $p\in R$, and measures $\nu^u_p,\nu^\cs_p$ on $V_R^u(p),V_R^\cs(p)$ respectively, we can define a measure $\nu = \nu^u_p\otimes \nu^\cs_p$ on $R$ by $\nu([A,B]) = \nu^u_p(A) \nu^\cs_p(B)$ for  $A\subset V_R^u(p)$ and $B\subset V_R^\cs(p)$.  The following lemma is proved in Appendix \ref{appendix}.


\begin{lemma}\label{lem:lps}
Let $R$ be a rectangle and $\mu$ a measure with $\mu(R)>0$.  Then the following are equivalent.
\begin{enumerate}[label=\textup{(\arabic{*})}]
\item\label{cs-ac} $(\pi_{yz}^\cs)_* \mu_y^u \ll \mu_z^u$ for $\mu$-a.e.\ $y,z\in R$.
\item\label{cs-sim} $(\pi_{yz}^\cs)_* \mu_y^u \sim \mu_z^u$ for $\mu$-a.e.\ $y,z\in R$.
\item\label{u-ac} $(\pi_{yz}^u)_* \mu_y^\cs \ll \mu_z^\cs$ for $\mu$-a.e.\ $y,z\in R$.
\item\label{u-sim} $(\pi_{yz}^u)_* \mu_y^\cs \sim \mu_z^\cs$ for $\mu$-a.e.\ $y,z\in R$.
\item\label{prod-sim-0} there exist $p\in R$ and measures $\tilde\mu_p^u,\tilde\mu_p^\cs$ on $V_R^u(p),V_R^\cs(p)$ such that $\mu|_R \sim \tilde\mu_p^u \otimes \tilde\mu_p^\cs$.
\item\label{prod-sim} $\mu|_R \sim \mu_y^u \otimes \mu_y^\cs$ for $\mu$-a.e.\ $y\in R$.
\end{enumerate}
\end{lemma}

\begin{definition}\label{def:lps}
A measure $\mu$ on $\L$ has \emph{local product structure} if there is $\sigma>0$
such that for any rectangle $R \subset \L$ with $\mu(R)>0$ and $\diam(R) < \sigma$, one (and hence all) of the conditions in Lemma \ref{lem:lps} holds.
\end{definition}


\subsection{Equilibrium measures}\label{def.2}

Let $\varphi\colon \L \to \RR$ be a continuous function, which we call a \emph{potential function}. Given an integer $n\ge 0$, the \emph{dynamical metric of order $n$} is
\begin{equation}\label{eqn:dn}
d_n(x,y) = \max\{d(f^kx,f^ky) : 0\leq k< n\}
\end{equation}
and for each $r>0$, the associated \emph{Bowen balls} are given by
\begin{equation}\label{eqn:Bowen-balls}
B_n(x,r)=\{y\colon d_n(x,y)< r\}.
\end{equation}
A set $E\subset\Lambda$ is said to be \emph{$(n,r)$-separated} if $d_n(x,y)\geq r$ for all $x\ne y\in E$.
Given $X\subset M$, a set $E\subset X$ is said to be \emph{$(n,r)$-spanning} for $X$ if
$X\subset \bigcup_{x\in E} B_n(x,r)$.

Let $S_n\ph(x)=\sum_{k=0}^{n-1} \ph(f^k x)$ denote the $n$th \emph{Birkhoff sum} along the orbit of $x$.  The \emph{partition sums} of $\ph$ on a set $X\subset M$ are the following quantities:
\begin{equation}\label{eqn:part-sum}
\begin{aligned}
\Zspan_n(X,\ph,r) &:= \inf \Big\{ \sum_{x\in E} e^{S_n\ph(x)} : E\subset X
\text{ is $(n,r)$-spanning for }X\Big\}, \\
\Zsep_n(X,\ph,r) &:= \sup \Big\{\sum_{x\in E} e^{S_n\ph(x)} : E\subset X \text{
is $(n,r)$-separated}\Big\}.
\end{aligned}
\end{equation}
The \emph{topological pressure} of $\ph$ on $X=\L$ is given by \begin{equation}\label{eqn:pressure}
P(\ph) = \lim_{r\to 0} \ulim_{n\to\infty} \frac 1n \log \Zspan_n(\Lambda,\ph,r)
= \lim_{r\to 0} \ulim_{n\to\infty} \frac 1n \log\Zsep_n(\Lambda,\ph,r);
\end{equation}
see \cite[Theorem 9.4]{pW82} for a proof that the limits are equal, and that one gets the same value if $\ulim$ is replaced by $\llim$.

Denote by $\mathcal{M}(f)$ the set of $f$-invariant Borel probability measures on $\L$. 
The \emph{variational principle} \cite[Theorem 9.10]{pW82} establishes that 
\begin{equation}\label{eqn:var}
P(\varphi)=  \sup_{\mu\in\mathcal{M}(f)}  \left\{  h_{\mu}(f) + \int \varphi\, d\mu   \right\}.
\end{equation}
We call a measure $\mu\in\mathcal{M}(f)$ an \emph{equilibrium measure} for $\varphi$ if it achieves the supremum in \eqref{eqn:var}. (Such a measure is also often referred to as an \emph{equilibrium state}.)

We say that a measure $\mu\in\mathcal{M}(f)$ is a \emph{Gibbs measure} (or that $\mu$ has the \emph{Gibbs property}) with respect to $\ph$ if for every small $r>0$ there is $Q=Q(r)>0$ such that for every $x\in \L$ and $n\in \NN$, we have
\begin{equation}\label{gibbs2}
Q^{-1}\leq \frac{\mu(B_n(x,r))}{\exp(-P(\varphi)n+S_n\ph(x))}\leq Q.
\end{equation}
A straightforward computation with partition sums shows that every Gibbs measure for $\ph$ is an equilibrium measure for $\ph$; however, the converse is not true in general, and there are examples of systems and potentials with equilibrium measures that do not satisfy the Gibbs property.

\section{Carath\'eodory dimension structure}\label{car-struc} 

We recall the \emph{Carath\'eodory dimension construction} described in \cite[\S10]{pes97}, which generalizes the definition of Hausdorff dimension and measure.

\subsection{Carath\'eodory dimension and measure} A \emph{Carath\'eodory dimension structure}, or \emph{$C$-structure}, on a set $X$ is given by the following data.
\begin{enumerate}[label=(\arabic{*})]
\item An indexed collection of subsets of $X$, denoted 
$\mathcal{F}=\{U_s\colon s\in\mathcal{S}\}$.
\item Functions $\xi,\eta,\psi\colon \mathcal{S}\to[0,\infty)$ satisfying the following conditions:
\begin{enumerate}[label=\upshape{(H\arabic{*})}]
\item\label{H1} if $U_s=\emptyset$, then $\eta(s)=\psi(s)=0$; if $U_s\ne\emptyset$, then 
$\eta(s)>0$ and $\psi(s)>0$;\footnote{In \cite{pes97}, Condition \ref{H1} includes the requirement that there is $s_0\in\mathcal{S}$ such that $U_{s_0} = \emptyset$, but this can safely be omitted as long as we define $m_C(\emptyset,\alpha)=0$, since we can always formally enlarge our collection by adding the empty set, without changing any of the definitions below.}
\item\label{H2} for any $\delta>0$ one can find $\varepsilon>0$ such that $\eta(s)\leq\delta$ for any $s\in\mathcal{S}$ with $\psi(s)\le\varepsilon$;
\item\label{H3} for any $\eps>0$ there exists a finite or countable subcollection 
$\mathcal{G}\subset\mathcal{S}$ that covers $X$ (meaning that 
$\bigcup_{s\in\mathcal{G}}U_s\supset X$) and has 
$\psi(\mathcal{G}):=\sup\{\psi(s)\colon\in\mathcal{S}\}\le\eps$.
\end{enumerate}
\end{enumerate}
Note that no conditions are placed on $\xi$.

The $C$-structure $(\mathcal{S},\mathcal{F},\xi,\eta,\psi)$ determines a one-parameter family of outer measures on $X$ as follows. Fix a nonempty set $Z\subset X$ and consider some 
$\mathcal{G}\subset\mathcal{S}$ that covers $Z$ as in \ref{H3}. 
Interpreting $\psi(\mathcal{G})$ as the largest \emph{size} of sets in the cover, we can define for each $\alpha\in \RR$ an outer measure on $X$ by
\begin{equation}\label{eqn:mCZa}
m_C(Z,\alpha):=\lim_{\eps\to 0}\inf_{\mathcal{G}}\sum_{s\in\mathcal{G}}\xi(s)\eta(s)^\alpha,
\end{equation}
where the infimum is taken over all finite or countable $\mathcal{G}\subset\mathcal{S}$ covering $Z$ with $\psi(\mathcal{G})\le\eps$.  Defining $m_C(\emptyset,\alpha) := 0$, this gives an outer measure by \cite[Proposition 1.1]{pes97}.
The measure induced by $m_C(\cdot,\alpha)$ on the $\sigma$-algebra of measurable sets is the \emph{$\alpha$-Carath\'eodory measure}; it need not be $\sigma$-finite or non-trivial.

\begin{proposition}[{\cite[Proposition 1.2]{pes97}}]\label{prop:C-dim}
For any set $Z\subset X$ there exists a critical value $\alpha_C\in\mathbb{R}$ such that $m_C(Z,\alpha)=\infty$ for $\alpha<\alpha_C$ and $m_C(Z,\alpha)=0$ for 
$\alpha>\alpha_C$.
\end{proposition}
We call $\dim_CZ =\alpha_C$ the \emph{Carath\'eodory dimension} of the set $Z$ associated to the $C$-structure $(\mathcal{S},\mathcal{F},\xi,\eta,\psi)$. By Proposition \ref{prop:C-dim}, $\alpha=\dim_C X$ is the only value of $\alpha$ for which \eqref{eqn:mCZa} can possibly produce a non-zero finite measure on $X$, though it is still possible that $m_C(X,\dim_C X)$ is equal to $0$ or $\infty$.

\subsection{A $C$-structure on local unstable leaves}\label{car.measure} Given a potential 
$\ph$, a number $r>0$, and a point $x\in\L$, we define a $C$-structure on 
$X=\Vl^u(x)\cap\L$ in the following way.
For our index set we put $\mathcal{S}=X\times \NN$, and to each
$s = (x,n)\in X\times\NN$, we associate the $u$-Bowen ball
\begin{equation}\label{eqn:Bowen-u-ball}
U_s = B^u_n(x,r)=B_n(x,r)\cap \Vl^u(x);
\end{equation}
then $\mathcal{F}$ is the collection of all such balls.  Set 
\begin{equation}\label{eqn:PC2}
\xi(x,n)=e^{S_n\ph(x)}, \quad \eta(x,n)=e^{-n}, \quad \psi(x,n)=\tfrac 1n.
\end{equation}
It is easy to see that $(\mathcal{S},\mathcal{F},\xi,\eta,\psi)$ satisfies \ref{H1}--\ref{H3} and defines a $C$-structure, whose associated outer measure is given by
\begin{equation}\label{eqn:mCP2}
m_C(Z,\alpha)=\lim_{N\to\infty}\inf_\mathcal{G}\sum_{(x,n)\in\mathcal{G}}e^{S_n\ph(x)} 
e^{-n\alpha},
\end{equation}
where the infimum is taken over all $\mathcal{G} \subset \mathcal{S}$ such that 
$\bigcup_{(x,n)\in\mathcal{G}}B^u_n(x,r)\supset Z$ and $n\geq N$ for all 
$(x,n)\in\mathcal{G}$.

Given $x\in \Lambda$ and the corresponding $X = \Vl^u(x) \cap \L$, we are interested in computing 
\begin{enumerate}
\item the Carath\'eodory dimension of $X$, as determined by this $C$-structure; 
\item the (outer) measure on $X$ defined by \eqref{eqn:mCP2} at $\alpha=\dim_CX$.
\end{enumerate}
We settle the first problem in Theorem  \ref{thm:finite} below in which we prove (among other things) that for small $r$, under the assumptions \ref{C1}--\ref{C3} on the map $f$ and some regularity assumptions on the potential function $\varphi$ (see Section \ref{assumptions}), the $C$-structure defined on $X=\Vl^u(x)\cap\L$ as above satisfies $\dim_CX= P(\ph)$ for every $x\in\L$.
This allows us to consider the outer measure on $X$ given by 
\begin{equation}\label{car0}
m_x^\C(Z):=m_C(Z,P(\ph))=\lim_{N\to\infty}\inf\sum_ie^{-n_iP(\ph)}e^{S_{n_i}\varphi(x_i)},
\end{equation}
where the infimum is taken over all collections $\{B^u_{n_i}(x_i,r)\}$ of $u$-Bowen balls with $x_i\in \Vl^u(x)\cap\L$, $n_i\ge N$, which cover $Z$; for convenience we write $\C=(\ph,r)$ to keep track of the data on which the reference measure depends. We use the same notation $m_x^\C$ for the corresponding Carath\'eodory measure on $X$  obtained by restricting to the $\sigma$-algebra of $m_x^\C$-measurable sets.

One must do some work to show that this outer measure is finite and nonzero; we do this in \S\S\ref{sec:Borel}--\ref{sec:finite}. We also show that this measure is Borel; that is, that every Borel set is $m_x^\C$-measurable.

\section{Main results}\label{sec:main-results}

\subsection{Assumptions on the map and the potential}\label{assumptions}

As in \S\ref{def}, let $f\colon U\to M$ be a diffeomorphism onto its image, where $M$ is a compact smooth Riemannian manifold and $U\subset M$ is open, and suppose that $\L\subset U$ is a compact $f$-invariant set on which $f$ is partially hyperbolic in the broad sense, with $T_\L M = \Ecs \oplus E^u$.  Suppose moreover that \ref{C1} and \ref{C2} are satisfied, so that $\Ecs$ is Lyapunov stable and $f|\L$ is topologically transitive.  Finally, suppose that the local product structure condition \ref{C3} is satisfied. 

Let $\tau>0$ be the size of the local manifolds in Proposition \ref{prop:local-mfds} and Theorem \ref{thm:Wcs}.
A potential function $\ph\colon \L\to\RR$ is said to have the \emph{$u$-Bowen property} if there exists $Q_u>0$ such that for every $x\in \L$, $n\geq 0$, and $y\in B_n^u(x,\tau) \cap \L$, we have $|S_n \ph(x) - S_n\ph(y)| \leq Q_u$.  Similarly, we say that $\ph$ has the \emph{$cs$-Bowen property}
if there exist $Q_\cs>0$ and $r_0'>0$ such that for every $x\in \L$, $n\geq 0$, and $y\in B_\L^\cs(x,r_0')$, we have $|S_n \ph(x) - S_n\ph(y)| \leq Q_\cs$.  Let $\CB(\L)$ be the set of all functions $\ph\colon \L\to \RR$ that satisfy both the $u$- and $cs$-Bowen properties.

\begin{remark}\label{rmk:unif-hyp-2}
As mentioned in Remark \ref{rmk:unif-hyp}, all of the conditions in the first paragraph above are satisfied if $\L$ is a transitive locally maximal hyperbolic set for $f$.  Moreover, in this case it follows from \cite[Lemma 6.6]{CPZ} that $\CB(\L)$ contains every H\"older continuous potential function.
\end{remark}

\subsection{Statements of main results} 

From now on we fix $\L$, $f$, and $\ph\in \CB(\L)$ as described above.
Our first result, which we prove in \S\ref{sec:measures}, shows that the measure $m_x^\C$ defined in \eqref{car0} is finite and nonzero.

\begin{theorem}\label{thm:finite}
Fix $0<r<\tau/3$.  There is $K>0$ such that for every $x\in\Lambda$, the following are true.
\begin{enumerate}[label=\upshape{(\arabic{*})}]
\item\label{fin1} For the $C$-structure defined on $X=\Vl^u(x)\cap\L$ by $u$-Bowen balls $B^u_n(x,r)$ and \eqref{eqn:PC2}, we have $\dim_CX= P(\ph)$ for every $x\in\L$. 
\item\label{fin2} $m_x^\C$ is a Borel measure on $X:=\Vl^u(x)\cap\Lambda$.
\item\label{fin3} $m_x^\C(\Vl^u(x)\cap \Lambda)\in [K^{-1},K]$.
\item\label{fin4} If $\Vl^u(x)\cap\Vl^u(y)\cap\Lambda\ne\emptyset$, then $m_x^\C$ and $m_y^\C$ agree on the intersection.
\end{enumerate}
\end{theorem}

\begin{definition}\label{def:u-gibbs}
Consider a family of measures $\{\mu_x : x\in \Lambda\}$ such that $\mu_x$ is supported on $\Vl^u(x)$. We say that this family has the \emph{$u$-Gibbs property}\footnote{Note that this is a different notion than the idea of \emph{$u$-Gibbs state} from \cite{PS82}.} with respect to the potential function $\ph\colon \Lambda\to \RR$ if there is $Q_0=Q_0(r)>0$ such that for all $x\in \Lambda$ and $n\in \NN$, we have
\begin{equation}\label{eqn:u-gibbs}
Q_0^{-1} \leq \frac{\mu_x(B_n^u(x,r))}{e^{-nP(\ph) + S_n\ph(x)}} \leq Q_0.
\end{equation}
\end{definition}

The following two results are proved in \S\ref{sec:scaling}: the first establishes the scaling properties of the measures $m_x^\C$ under iteration by $f$, which then leads to the $u$-Gibbs property. 
\begin{theorem}\label{thm:Gibbs}
For every $x\in\Lambda$, we have $f^*m_{f(x)}^\C:= m_{f(x)}^\C\circ f\ll m_x^\C$, with Radon--Nikodym derivative $e^{P(\ph)-\ph}$.   
\end{theorem}

In the uniformly hyperbolic setting a similar result was obtained by Leplaideur in \cite{Lep}.

\begin{corollary}\label{cor:Gibbs}
The family of measures 
$\{m_x^\C\}_{x\in \Lambda}$ has the $u$-Gibbs property. In particular, for every relatively open $U\subset \Vl^u(x) \cap \L$, we have $m_x^\C(U)>0$.
\end{corollary}

Given a rectangle $R$ and points $y,z\in R$, let $\pi_{yz} \colon V_R^u(y) \to V_R^u(z)$ be the holonomy map from Definition \ref{def:holonomy}.  We say that $\pi_{yz}$ is \emph{absolutely continuous with respect to the system of measures $m^\C_{y}$} if the pullback measure $\pi_{yz}^*m^\C_z$ is equivalent to the measure $m^\C_y$ for every $y,z$.\footnote{This looks similar to the notion of local product structure in Lemma \ref{lem:lps}, but the difference here is that the system of measures $m_y^\C$ are not assumed to arise as conditional measures for some Borel measure on $\L$.} In this case the Jacobian of $\pi_{yz}$ is the function $\Jac\pi_{yz}\colon V_R^u(y) \to (0,\infty)$ defined by the following Radon--Nikodym derivative:
\[
\Jac\pi_{yz}=\frac{d\pi_{yz}^*m^\C_z}{dm^\C_y}.
\]

\begin{theorem}\label{thm:holonomy}
The holonomy map is absolutely continuous with respect to the system of measures $m^\C_y$. Moreover, there are $C,\sigma>0$
such that for every
rectangle $R$ with $\diam(R)<\sigma$, and every $y,z\in R$, the Jacobian of 
$\pi_{zy}$ satisfies $C^{-1}\le\Jac\pi_{zy}(x)\le C$ for $m_y^\C$-a.e.\ $x$.
\end{theorem}

The measure $m^\C_x$ can be extended to a measure on $\L$ by taking $m^\C_x(A):=m^\C_x(A\cap \Vl^u(x))$ for any Borel set $A\subset\L$. We consider the evolution of (the normalization of) this measure by the dynamics; that is the sequence of measures 
\begin{equation}\label{seq-meas}
\mu_n:=\frac{1}{n}\sum_{k=0}^{n-1} \frac{f^k_{*}m^\C_x}{m_x^\C(\Vl^u(x))}. 
\end{equation}
Our main result is the following, which we prove in \S\ref{sec:proof-of-main}.

\begin{theorem}\label{thm:main}
Under the conditions in \S\ref{assumptions}, the following are true.
\begin{enumerate}[label=\upshape{(\arabic{*})}]
\item\label{m1} For every $x\in \Lambda$, the sequence of measures 
from \eqref{seq-meas} is weak* convergent as $n\to\infty$ to a limiting probability measure $\mu_\ph$, which is independent of $x$.
\item\label{m2} The measure $\mu_\ph$ is ergodic, gives positive weight to every open set in $\L$, has the Gibbs property \eqref{gibbs2}, and is the unique equilibrium measure for 
$(\Lambda,f,\ph)$.  
\item\label{m3}
For every rectangle $R\subset \Lambda$ with $\mu_\ph(R)>0$, the conditional measures 
$\mu_y^u$ generated by $\mu_\ph$ on unstable sets $V_R^u(y)$ are equivalent for 
$\mu_\ph$-almost every $y\in R$ to the reference measures $m_y^\C|_{V_R^u(y)}$.  Moreover, there exists $C_0>0$, independent of $R$ and $y$, such that for 
$\mu_\ph$-almost every $y\in R$ we have
\begin{equation}\label{eqn:mu-m}
C_0^{-1}\leq \frac{d\mu_y^u}{dm_y^\C}(z) m_y^\C(R) \leq C_0 \text{ for $\mu_y^u$-a.e.}\ z\in V_R^u(y).
\end{equation}
\item\label{m4} 
The measure $\mu_\ph$ has local product structure as in Definition \ref{def:lps}.
\end{enumerate}
\end{theorem}

In the uniformly hyperbolic setting Statement \ref{m4} was proved by Leplaideur in \cite{Lep}.

\section{Applications}\label{sec:applications}
\subsection{Particular potentials}\label{sec:particular-potentials}

In this section we use our main results to establish existence and uniqueness of equilibrium measures for some particular potentials.

\subsubsection{Measures of maximal entropy  (MME)}\label{sec:mme} 

As in \S\ref{def}, let $f\colon U\to M$ be a diffeomorphism onto its image, where $M$ is a compact smooth Riemannian manifold and $U\subset M$ is open, and suppose that $\L\subset U$ is a compact $f$-invariant set on which $f$ is partially hyperbolic in the broad sense with $T_\L M = \Ecs \oplus E^u$.  Suppose moreover that Conditions \ref{C1} and \ref{C2} are satisfied, so that 
$\Ecs$ is Lyapunov stable and $f|\L$ is topologically transitive.  Finally, suppose that the local product structure  condition \ref{C3} is satisfied.

We observe that a constant potential function always satisfies the $u$- and $cs$-Bowen properties, and thus our construction produces a unique MME which we denote by 
$\mu_0$. To describe this measure given $x\in \L$, define an outer measure on 
$X=\Vl^u(x)\cap\L$ by 
\begin{equation}\label{carMME}
m_x^h(Z):=\lim_{N\to\infty}\inf\sum_ie^{-n_ih_{\text{top}}},
\end{equation}
where $h_{\text{top}}$ denotes the topological entropy of $f$ on $\L$ and the infimum is taken over all collections $\{B^u_{n_i}(x_i,r)\}$ of $u$-Bowen balls with 
$x_i\in \Vl^u(x)\cap\L$, $n_i\ge N$, which cover $Z$. We use the same notation $m_x^h$ for the corresponding Carath\'eodory measure on $X$.
Then the following is an immediate consequence of Theorem \ref{thm:main}.

\begin{theorem}\label{thm:MME}
The following statements hold:
\begin{enumerate}[label=\upshape{(\arabic{*})}]
\item\label{m1h} For every $x\in \Lambda$, the sequence of measures 
$\mu_n:=\frac1n\sum_{k=0}^{n-1} f_*^k m_x^h / m_x^h(\Vl^u(x))$ converges in the  weak$^*$ topology as $n\to\infty$ to $\mu_0$ (independently of $x$).
\item\label{m2h} $\mu_0$ is ergodic and is fully supported on $\L$. 
\item $\mu_0$ has the Gibbs property: for every small $r>0$ there is $Q=Q(r)>0$ such that for every $x\in\L$ and $n\in\NN$, 
\begin{equation}\label{gibbs2h}
Q^{-1}\leq \frac{\mu(B_n(x,r))}{\exp(-n\htop(f))}\le Q.
\end{equation}
\item\label{m3h}
For every rectangle $R\subset\L$ with $\mu_0(R)>0$, the conditional measures 
$\mu_y^u$ generated by $\mu_0$ on unstable sets $V_R^u(y)$ are equivalent for 
$\mu_0$-almost every $y\in R$ to the reference measures $m_y^h|_{V_R^u(y)}$;  moreover, there exists $C_0>0$, independent of $R$ and $y$, such that for 
$\mu_0$-almost every $y\in R$ we have
\begin{equation}\label{eqn:mu-mh}
C_0^{-1}\leq \frac{d\mu_y^u}{dm_y^h}(z) m_y^h(R) \leq C_0 \text{ for $\mu_y^u$-a.e.}\ z\in V_R^u(y).
\end{equation}
\item\label{m4h} 
$\mu_0$ has local product structure as in Definition \ref{def:lps}.
\end{enumerate}
\end{theorem}

\subsubsection{The geometric $q$-potential}\label{sec:geometric-potential}
We consider the family of geometric $q$-potentials,
\[
\ph_q(x) = -q\log \det Df|E^u(x),\quad q\in\mathbb{R}.     
\]
In order to apply our results to this family, we need to verify the $u$- and $cs$-Bowen properties. In general, they may not be satisfied in our setting and therefore we shall impose the following additional requirements:

\begin{enumerate}[label=\upshape{(A\arabic{*})}]
\item\label{A1} The partially hyperbolic set $\L\subset U$ is an \emph{attractor} for $f$; that is, $\overline{f(U)}\subset U$ and $\L:=\bigcap_{n\geq 0}f^n(U)$.
\item\label{A2} There is $\sigma>0$ such that for every rectangle $R\subset\L$ with $\diam(R)<\sigma$, the holonomy maps between local unstable leaves are uniformly absolutely continuous with respect to leaf volume $\vol_x$; that is, there exists $C>0$ such that for every $y,z\in R$, the Jacobian of $\pi_{zy}$ with respect to leaf volumes $\vol_y$ and $\vol_z$ satisfies $C^{-1}\le\Jac\pi_{zy}(x)\le C$ for $\vol_y$-a.e.\ $x$.
\end{enumerate}
For $x\in\L$ and $q\in\mathbb{R}$ define an outer measure on $X=\Vl^u(x)\cap\L$ by 
\begin{equation}\label{cart}
m_x^q(Z):=m_C(Z,P(\ph_q))=\lim_{N\to\infty}\inf\sum_ie^{-n_iP(\ph_q)}e^{S_{n_i}\varphi_q(x_i)},
\end{equation}
where the infimum is taken over all collections $\{B^u_{n_i}(x_i,r)\}$ of $u$-Bowen balls with $x_i\in \Vl^u(x)\cap\L$, $n_i\ge N$, which cover $Z$.  We use the same notation $m_x^q$ for the corresponding Carath\'eodory measure on $X$.

\begin{theorem}\label{thm:tpotential}
If $\L$ is a partially hyperbolic attractor for a diffeomorphism $f$ satisfying Conditions \ref{C1}--\ref{C3} and \ref{A1}--\ref{A2}, then for every $q\in\mathbb{R}$ the following statements hold:
\begin{enumerate}[label=\upshape{(\arabic{*})}]
\item\label{m1t} For every $x\in\Lambda$, the sequence of measures 
$\mu_n:=\frac1n\sum_{k=0}^{n-1} f_*^k m_x^q / m_x^q(\Vl^u(x))$ converges in the  weak$^*$ topology as $n\to\infty$ to a probability measure $\mu_{\ph_q}$ (independently of $x$).
\item\label{m2t} The measure $\mu_{\ph_q}$ is ergodic, gives positive weight to every open set in $\L$, has the Gibbs property \eqref{gibbs2}, and is the unique equilibrium measure for $(\Lambda,f,\ph_q)$.
\item\label{m3t} For every rectangle $R\subset \Lambda$ with  $\mu_{\ph_q}(R)>0$, the conditional measures $\mu_y^u$ generated by $\mu_{\ph_q}$ on unstable sets $V_R^u(y)$ are equivalent for $\mu_{\ph_q}$-almost every $y\in R$ to the reference measures $m_y^q|_{V_R^u(y)}$.  Moreover, there exists $C_0>0$, independent of $R$ and $y$, such that for $\mu_{\ph_q}$-almost every $y\in R$ we have
\begin{equation}\label{eqn:mu-mt}
C_0^{-1}\le\frac{d\mu_y^u}{dm_y^q}(z)m_y^q(R)\le C_0\text{ for $\mu_y^u$-a.e.}\ z\in V_R^u(y).
\end{equation}
\item\label{m4t} 
The measure $\mu_{\ph_q}$ has local product structure as in Definition \ref{def:lps}.
\end{enumerate}
\end{theorem}

\begin{proof}
We need to verify that $\varphi_q\in C_B(\L)$; it suffices to show that $\ph_1\in C_B(\L)$ since $\ph_q = q\ph_1$.   First observe that $\varphi_1$ is H\"older continuous. By the argument presented in \cite[Lemma 6.6]{CPZ}, this implies that $\varphi_1$ has the $u$-Bowen property with some constant $Q_u>0$. Then observing that
\begin{equation}\label{eqn:Sn-det}
S_n\ph_1(x) = -\log\prod_{k=0}^{n-1} \det Df|E^u(f^kx) = -\log \det Df^n|E^u(x),
\end{equation}
we see that for every $x\in\L$, $n\in\mathbb{N}$, and $z\in B_n^u(x,\tau) \cap \L$, we have
\[
|\log \det Df^n|E^u(z) - \log \det Df^n|e^u(x)| = |S_n\ph_1(z) - S_n\ph_1(x)| \leq Q_u,
\]
and exponentiating gives
\begin{equation}\label{eqn:detdet}
\det Df^n|E^u(z) = e^{\pm Q_u} \det Df^n|E^u(x).
\end{equation}
Here and below we use the following notation: given $A,B,C,a\ge 0$, we write $A=C^{\pm a}B$ as shorthand to mean $C^{-a}B\le A\le C^a B$.

To show that $\ph_1$ has the $cs$-Bowen property we start by choosing $\sigma'\in (0,\sigma/4)$ sufficiently small (here $\sigma$ is as in \ref{A2}) that if $d(x,y) < \sigma'$ and $a\in B_\L^u(x,\sigma/4)$, then 
$\pi_{xy}(a)\in B_\L^u(y,\sigma/2)$. Then we let 
$r_0'$ be the value of $\delta$ given by \ref{C1} with $\theta = \sigma'$.

For every $x\in \L$, $y\in B_\L^\cs(x,r_0')$, and $n\in\NN$, Condition \ref{C1} gives $f^n(y)\in B_\L^\cs(f^n(x),\sigma')$. Writing $A := B_n^u(x,\sigma/4)$ and 
$B :=\pi_{xy}(A)$, we see that $B\subset B_n^u(y,\sigma/2)$ by our choice of $\sigma'$.
For any $a\in A$ and $b\in B$ we have $d(a,b)\le d(a,x)+d(x,y)+d(y,b)\le\sigma$ and similarly $d(f^n(a),f^n(b))\le\sigma$, so Condition \ref{A2} guarantees that
\[
\vol_x(A) = C^{\pm1} \vol_y(B)
\quad\text{and}\quad
\vol_{f^n(x)} (f^n(A)) = C^{\pm1} \vol_{f^n(y)}(f^n(B)),
\]
and thus
\begin{equation}\label{eqn:volvol}
\frac{\vol_{f^n(y)}(f^n(B))}{\vol_y(B)}
= C^{\pm 2} \frac{\vol_{f^n(x)}(f^n(A))}{\vol_x(A)}.
\end{equation}
On the other hand, \eqref{eqn:detdet} gives 
\[
\vol_{f^n(x)}(f^n(A)) = \int_A \det Df^n|E^u(z) \,d\vol_x(z)
= e^{\pm Q_u} (\det Df^n|E^u(x)) \vol_x(A),
\]
and similarly,
\[
\vol_{f^n(y)}(f^n(B)) = e^{\pm Q_u} (\det Df^n|E^u(y)) \vol_y(B),
\]
which yield
\begin{equation}\label{eqn:volvol2}
\frac{\vol_{f^n(y)}(f^n(B))}{\vol_y(B)}
= e^{\pm 2Q_u} \frac{\det Df^n|E^u(y)}{\det Df^n|E^u(x)} \frac{\vol_{f^n(x)}(f^n(A))}{\vol_x(A)}.
\end{equation}
Combining this with \eqref{eqn:volvol} and using \eqref{eqn:Sn-det} gives
\[
C^{\pm 2}
=
\frac{\vol_{f^n(y)}(f^n(B)) \vol_x(A)}{\vol_y(B) \vol_{f^n(x)}(f^n(A))}
= e^{\pm 2Q_u} e^{S_n\ph_1(x) - S_n\ph_1(y)},
\]
and upon taking logs we conclude that
$S_n\varphi_1(x)-S_n\varphi_1(y)=\pm 2\log C\pm 2Q_u$, and thus $\ph_1\in C_B(\L)$.  With this complete, the theorem follows immediately from Theorem \ref{thm:main}.
\end{proof}

\subsection{Particular classes of dynamical systems}\label{sec:app}
\subsubsection{Time-$1$ map of an Anosov flow}\label{sec:time1}

\begin{definition}
A $C^1$ flow $f^t\colon M \to M$ on a smooth compact manifold $M$ is called an \emph{Anosov flow} if there exists a Riemannian metric and a number $0<\lambda <1$ such that the tangent bundle splits into three subbundles $TM=E^s\oplus E^0\oplus E^u$, each invariant under the flow such that
\begin{enumerate}
\item $\frac{d}{dt}f^t(x)|_{t=0}  \in E^0(x)\setminus \{0 \}$ and $\dim E^0(x)=1$;
\item $\|Df^t|E^s\|\le\lambda^t$ and $\|Df^{-t}|E^u\|\le\lambda^t$ for all $t\geq 0$.
\end{enumerate} 
\end{definition}
It is well known that if an Anosov flow $f^t$ is of class $C^r$, $r\ge 1$, then for each $x\in M$ there are a pair of embedded $C^r$-discs $W^s(x)$ and $W^u(x)$ called \emph{local strong stable} and \emph{unstable manifolds}, and a number $C>0$ such that 
\begin{enumerate}
\item $T_x W^s(x)=E^s(x)$ and $T_x W^u(x)=E^u(x)$;
\item if $y\in W^u(x)$, then $d(f^{-t}(x),f^{-t}(y))\le C\lambda^t d(x,y)$ for all $t\geq 0$;
\item if $y\in W^s(x)$, then $d(f^t(x),f^t(y))\le C \lambda^t d(x,y)$ for all $t\geq 0$.
\end{enumerate}
We define \emph{weak-unstable} and \emph{weak-stable manifolds} through $x$ by
\[
W^{0u} = \bigcup_{t\in (-r,r)} W^u(f^t(x)), \qquad  W^{0s} = \bigcup_{t\in (-r,r)} W^s(f^t(x)).
\]
Given an Anosov flow $f^t\colon M\to M$ one can define a diffeomorphism 
$f\colon M\to M$ to be the time-$1$ map of the flow. That is, $f(x):=f^1(x)$. Observe that such an $f$ is partially hyperbolic in the broad sense with $\L=M$ and 
$\Ecs =E^0\oplus E^s$ and satisfies Assumptions \ref{C1} and \ref{C3} from \S\ref{def}.

We stress that even when the flow is known to have a unique equilibrium measure for a certain potential function, this does not automatically imply uniqueness for the time-$1$ map; the simplest example is a constant-time suspension flow over an Anosov diffeomorphism. In this case the flow is topologically transitive but the time-$1$ map  
need not be.\footnote{The time-$1$ map is transitive if and only if the constant value of the roof function is irrational.}
In fact for Anosov flows this is the only obstruction to transitivity: if the flow is not topologically conjugate to a constant-time suspension, then it is topologically mixing and in particular the time-$1$ map is transitive \cite{jP72,FH}.  Even in this case, there may be measures that are invariant for the map but not for the flow \cite[Corollary 4]{QS12}, so uniqueness for the map is in general a more subtle question.

Given a H\"older continuous function $\psi\colon M \to \mathbb{R}$
(thought of as a potential for the flow), consider $\varphi(x):= \int_0^1 \psi(f^{\ttt}(x))\, d\ttt$
(thought of as a potential for the map). 
When the time-$1$ map is transitive, we obtain the following result.

\begin{theorem}\label{thm:maina1}
Let $f^t\colon M \to M$ be an Anosov flow on a smooth compact manifold $M$. 
Let $f=f^1$ be the time-$1$ map of $f^t$, and let $m_x^\C$ be the reference measures on $\Vl^u(x)$ associated to the potential function $\ph=\int_0^1\psi\circ f^\ttt\,d\ttt$. If $f$ is topologically transitive, then the following statements hold:
\begin{enumerate}[label=\upshape{(\arabic{*})}]
\item\label{1} For every $x\in M$, the sequence of measures 
$\mu_n:=\frac1n\sum_{k=0}^{n-1} f_*^k m_x^\C / m_x^\C(\Vl^u(x))$ is weak* convergent as $n\to\infty$ to a measure $\mu_\ph$, which is independent of $x$.
\item\label{2} The measure $\mu_\ph$ is ergodic, gives positive weight to every open set in $M$, has the Gibbs property \eqref{gibbs2}, and is the unique equilibrium measure for 
$(M,f,\ph)$.
\item\label{3}
For every rectangle $R\subset M$ with $\mu_\ph(R)>0$, the conditional measures 
$\mu_y^u$ generated by $\mu_\ph$ on unstable sets $V_R^u(y)$ are equivalent for 
$\mu_\ph$-a.e.\ $y\in R$ to the reference measures $m_y^\C|_{V_R^u(y)}$.  Moreover, there exists $C_0>0$, independent of $R$ and $y$, such that for $\mu_\ph$-a.e.\ 
$y\in R$ we have
\begin{equation}\label{eqn:mu-ma1}
C_0^{-1}\leq \frac{d\mu_y^u}{dm_y^\C}(z) m_y^\C(R) \leq C_0 \text{ for $\mu_y^u$-a.e.}\ z\in V_R^u(y).
\end{equation}
\item\label{4} The measure $\mu_{\varphi}$ has local product structure.
\end{enumerate}
\end{theorem}  
\begin{proof}
It is enough to check that $\varphi\in \CB(M)$ and then apply Theorem \ref{thm:main}. Since the function 
$\psi$ is H\"older continuous, so is $\varphi$. In particular, $\varphi$ satisfies Bowen's property along strong stable and unstable leaves.  Moreover,  $\varphi$ satisfies Bowen's property along the flow direction: indeed, consider two points $x=x(0)\in M$ and $y=x(\epsilon)$ for some $\epsilon>0$. We have 
$$
\begin{aligned}
|S_n\varphi(x)-S_n\varphi(y)|&=\Big|\int_0^n \psi(f^{\ttt}(x)) d\ttt - \int_0^n \psi(f^{\ttt}(y)) d\ttt  \Big| \\ 
&= \Big| \int_0^n \psi(f^{\ttt}(x)) d\ttt - \int_{\epsilon}^{n+\epsilon} \psi(f^{\ttt}(x)) d\ttt  \Big|   \\ 
 &=\Big|\int_0^{\epsilon}\psi(f^{\ttt}(x))\,d\ttt-\int_n^{n+\epsilon}\psi(f^{\ttt}(x))\,d\ttt\Big|   
\le 2\epsilon\|\psi\|_{\infty}.  
\end{aligned}
$$
Using the local product structure of the flow, this shows that $\ph \in \CB(M)$, and thus completes the proof of the theorem.
\end{proof}

\begin{remark}
Theorem \ref{thm:maina1} applies to the geometric potential $\ph(x)=-\log\det Df|E^u(x)$ and all its scalar multiples $q\ph$ for $q\in\RR$; indeed, taking 
$\psi(x)=\lim_{t\to 0}-\frac1t\log\det Df^t|E^u(x)$, we have 
$\ph=\int_0^1\psi\circ f^\ttt\,d\ttt$, and $\psi$ is H\"older continuous because the distribution $E^u$ is H\"older continuous. When $q=0$ the measure produced in Theorem \ref{thm:maina1} is the unique MME; when $q=1$ it is the unique $u$-measure, which is the unique SRB measure for the flow.
\end{remark}

We mention two alternate approaches to existence and/or uniqueness of equilibrium measures in the setting of Theorem \ref{thm:maina1}.  
First, the time-$1$ map of an Anosov flow has the entropy expansivity property \cite{EnExp}, which implies existence of an equilibrium measure for $f$ with respect to any continuous potential function $\varphi$. However, this approach does not say anything about uniqueness, the Gibbs property, or local product structure.

Substantially more information, including uniqueness, can be obtained by appealing to the corresponding result for the flow itself.  We are grateful to F.\ Rodriguez Hertz for the following argument, which uses a simple construction that goes back to Walters \cite[Corollary 4.12(iii)]{pW75}
and Dinaburg \cite{eD71}.

\begin{proposition}\label{prop:flow-map}
Let $X$ be a compact metric space and $\{f^t\colon X\to X\}_{t\in\RR}$ a continuous flow.  Suppose that $\psi\colon X\to \RR$ is continuous and that there is a unique equilibrium measure $\mu$ for $\psi$ with respect to the flow $f^t$.  Suppose moreover that $\mu$ is weak mixing.  Then $\mu$ is the unique equilibrium measure for the time-$1$ map $f=f^1$ and the potential function $\ph(x) = \int_0^1 \psi(f^t x)\,dt$.
\end{proposition}
\begin{proof}
First observe that the pressure of the flow (w.r.t.\ $\psi$) agrees with the pressure of the map (w.r.t.\ $\ph$) by \cite[Corollary 4.12(iii)]{pW75}, so $\mu$ is automatically an equilibrium measure for the map.  It remains only to prove uniqueness.

First note that since $(\{f^t\},\mu)$ is weak mixing, then $(f,\mu)$ is ergodic \cite[Proposition 3.4.40]{FH}.  Now let $\nu$ be any equilibrium measure for $(f=f^1,\ph)$; we claim that $\int_0^1 f_*^t \nu\,dt$ is an equilibrium measure for $(\{f^t\},\psi)$, and thus is equal to $\mu$; then ergodicity will imply that $f_*^t \nu = \mu$ for a.e.\ $t$, and thus $\nu=\mu$.

Since $\int_0^1 f_*^t \nu\,dt$ is flow-invariant (by $f$-invariance of $\nu$), to show that it is an equilibrium measure for the flow it suffices to prove that $h_{f_*^t\nu}(f) = h_\nu(f)$ and $\int\ph\,d(f_*^t\nu) = \int \ph\,d\nu$ for all $t$.  The first of these is standard.  For the second we observe that
\begin{align*}
\int \ph \,d(f_*^t\nu)
&= \int_M \int_t^{1+t} \psi(f^\ttt x)\,d\ttt\,d\nu 
= \int_M \Big( \int_t^1 \psi(f^\ttt x)\,d\ttt
+ \int_1^{1+t} \psi(f^\ttt x)\,d\ttt\Big) \,d\nu \\
&=\int_M \Big( \int_t^1 \psi(f^\ttt x)\,d\ttt
+ \int_0^{t} \psi(f^\ttt x)\,d\ttt\Big) \,d\nu 
= \int_M \int_0^1 \psi(f^\ttt x)\,d\ttt \,d\nu
= \int \ph\,d\nu,
\end{align*}
where the third equality uses $f$-invariance of $\nu$.  Then as argued above $\int_0^1 f_*^t \nu\,dt$ is an equilibrium measure for the flow, so it is equal to $\mu$, and ergodicity implies that $\nu=\mu$.
This proves that the unique equilibrium measure for the flow is also the unique equilibrium measure for the map.
\end{proof}

In the specific case when $f^t$ is a topologically mixing Anosov flow and $\psi$ is H\"older continuous, uniqueness of the equilibrium measure for the flow, together with the mixing property, was shown in \cite{BowRue}, and thus this argument establishes uniqueness for the time-$1$ map; moreover, the equilibrium measure is known to have the Gibbs property and local product structure \cite{nH94}, providing another proof of some of the statements in Theorem \ref{thm:maina1}.

\subsubsection{Time-$1$ map of the frame flow}
Let $M$ be a closed oriented $n$-dimensional manifold of negative sectional curvature. 
Consider the unit tangent bundle 
$$
SM = \{(x,v) : x\in M, ~v\in T_xM, ~ \|v\| = 1\}
$$ 
and the frame bundle
\begin{multline*}
FM  = \{ (x, v_0, v_1, \ldots, v_{n-1}) : x\in M, ~ v_i\in T_xM,~\text{and } \\
\{v_0,\dots,v_{n-1}\}\text{ is a positively oriented orthonormal frame at $x$}  \}.
\end{multline*}
We write $\vv = (x,v_0,v_1,\dots, v_{n-1})$ for an element of $FM$. 
The \emph{geodesic flow} $g^t\colon SM \to SM$ is defined by 
\[
g^t(x,v)=(\gamma_{(x,v)}(t), \dot{\gamma}_{(x,v)}(t)),
\]
where $\gamma_{(x,v)}(t)$
is the unique geodesic determined by the vector $(x,v)$, and the 
\emph{frame flow} $f^t\colon FM \to FM$ is given by
\[
f^t(x, v_0, v_1, \ldots, v_{n-1}) = (g^t(x,v_0), \Gamma^t_{\gamma}(v_1), \ldots, \Gamma^t_{\gamma}(v_{n-1})),
\]
where $\Gamma^t_{\gamma}$ is the parallel transport along the geodesic $\gamma(x,v_0)$.
The flow $f^t$ is partially hyperbolic with splitting $TFM=E^s\oplus E^0 \oplus SO(n-1)\oplus E^u,$ where $f^t$ acts isometrically on fibers of the center bundle $SO(n-1)$.\footnote{Note that for a partially hyperbolic \emph{flow}, the center bundle does not contain the flow direction.} Thus the time-1 map $f$ is partially hyperbolic with $\Ecs = E^0 \oplus SO(n-1) \oplus E^s$.

Given $x\in M$ and $v_0\in T_xM$ with $\|v_0\|=1$, denote by $N_{x,v_0}$ the compact set of positively oriented orthonormal $(n-1)$-frames in $T_xM$, which are orthogonal to $v_0$.
We will consider a class of H\"older potentials that are constant on each $N_{x,v_0}$; this class contains the geometric potentials.  

We need to restrict our attention to the case when the time-$1$ map $f$ is topologically transitive;
to this end, note that the frame flow $f^t$ preserves a smooth measure that is locally the product of the Liouville measure with normalized Haar measure on $SO(n-1)$. There are several cases in which the time-$1$ map $f$ is known to be ergodic with respect to this measure, and hence topologically transitive by \cite[Proposition 4.1.18]{Kat}.

\begin{proposition}[{\cite[Theorem 0.2]{BP03}}]\label{prop:ergodic}
Let $f^t$ be the frame flow on an $n$-dimensional compact smooth Riemannian manifold with sectional curvature between $-\Lambda^2$ and $-\lambda^2$ for some $\Lambda,\lambda>0$.
Then in each of the following cases the flow and its time-$1$ map are ergodic:
\begin{itemize}
\item if the curvature is constant,
\item for a set of metrics of negative curvature which is open and dense in the $C^3$ topology,
\item if $n$ is odd and $n\neq 7$,
\item if $n$ is even, $n\neq 8$, and $\lambda/\Lambda > 0.93$,
\item if $n=7$ or $8$ and $\lambda/\Lambda >0.99023\ldots$.
\end{itemize}
\end{proposition}

We therefore have the following result.

\begin{theorem}\label{thm:maina2}
Let $(FM,f)$ be the time-$1$ map of a frame flow from one of the cases in Proposition \ref{prop:ergodic}.
Suppose that $\ph\colon FM \to \RR$ is constant on fibers $N_{x,v_0}$ and given by $\varphi(\vv):= \int_0^1\psi(f^{\ttt}(\vv))d\ttt$, where $\psi\colon FM\to\mathbb{R}$ is H\"older continuous.
Then the following are true.
\begin{enumerate}[label=\upshape{(\arabic{*})}]
\item\label{F1} For every $\vv\in FM$, the sequence of measures 
$\mu_n := \frac 1n \sum_{k=0}^{n-1} f_*^k m_\vv^\C / m_\vv^\C(\Vl^u(\vv))$ from \eqref{seq-meas}
is weak$^*$ convergent as $n\to\infty$ to a measure $\mu_\ph$, which is independent of $\vv$.
\item\label{F2} The measure $\mu_\ph$ is ergodic, gives positive weight to every open set in $FM$, has the Gibbs property \eqref{gibbs2}, and is the unique equilibrium measure for 
$(FM,f,\ph)$.
\item\label{F3}
For every rectangle $R\subset FM$ with $\mu_\ph(R)>0$, the conditional measures $\mu_\vv^u$ generated by $\mu_\ph$ on unstable sets $V_R^u(\vv)$ are equivalent for $\mu_\ph$-a.e.\ 
$\vv\in R$ to the reference measures $m_\vv^\C|_{V_R^u(\vv)}$.  Moreover, there exists $C_0>0$, independent of $R$ and $\vv$, such that for $\mu_\ph$-a.e.\ $\vv\in R$ we have
\begin{equation}\label{eqn:mu-ma2}
C_0^{-1}\le\frac{d\mu_\vv^u}{dm_\vv^\C}(\ww)m_\vv^\C(R)\le C_0\text{ for $\mu_\vv^u$-a.e.}\ \ww\in V_R^u(\vv).
\end{equation}
\item\label{F4} The measure $\mu_{\varphi}$ has local product structure.
\end{enumerate}
\end{theorem}
\begin{proof}
The same computation as in the proof of Theorem \ref{thm:maina1} shows that $\ph$ satisfies Bowen's property along strong stable and unstable leaves and along the flow direction; since $\ph$ is constant along fibers this establishes the $u$- and $cs$-Bowen properties, and thus Theorem \ref{thm:main} applies.
\end{proof}

We point out that because the action of $SO(n-1)$ is transitive on each fiber and commutes with the flow, the geometric potential is constant on fibers, and thus Theorem \ref{thm:maina2} applies to the geometric $q$-potential for every $q\in\RR$.
\begin{remark}
Equilibrium measures for frame flows were recently studied by Spatzier and Visscher \cite{SV}; we briefly compare Theorem \ref{thm:maina2} to their results.
\begin{enumerate}
\item 
When the manifold $M$ has an odd dimension other than $7$, it is shown in \cite{SV} that for any H\"older continuous potential which is constant on fibers $N_{x,v_0}$ the frame flow possesses a unique equilibrium measure. The authors show that this measure is ergodic, fully supported, and has local product structure.  
However, whether this measure is weak mixing remains unknown, and without this the argument in the previous section cannot be used to deduce uniqueness for the time-$1$ map.
\item For the equilibrium measures constructed in \cite{SV} it is shown that the conditional measures generated by the equilibrium measure on central leaves are invariant under the action of $SO(n-1)$. This is a corollary of the fact that the equilibrium measure has the local product property. Hence, a similar argument will work to establish the same property for any equilibrium measure in Theorem \ref{thm:maina2}.\footnote{We would like to thank Ralf Spatzier for this comment.}
\end{enumerate}
\end{remark}

If $\mu$ is the unique equilibrium measure for the time-1 map of the flow (w.r.t.\ $\ph$), then each $f_*^t \mu$ is also an equilibrium measure for $(f,\ph)$ by similar arguments to those in the proof of Proposition \ref{prop:flow-map}, and by uniqueness we see that $\mu$ is flow-invariant; thus $\mu$ is an equilibrium measure for the flow (w.r.t.\ $\psi$).  Conversely, any equilibrium measure for the flow is an equilibrium measure for the map, so $\mu$ is also the unique equilibrium measure for the flow.  Thus Theorem \ref{thm:maina2} has the following consequence, which extends \cite{SV} to a  broader class of manifolds.

\begin{corollary}\label{cor:frame-flow}
For manifold $M$ satisfying one of the conditions listed in Proposition \ref{prop:ergodic}, and for any H\"older continuous potential which is constant on fibers $N_{x,v_0}$, the frame flow possesses a unique equilibrium measure which is ergodic, fully supported, has the Gibbs property and local product structure and satisfies \eqref{eqn:mu-ma2}.
\end{corollary}


\subsubsection{Partially hyperbolic diffeomorphisms with compact center leaves}
Let $M$ be a compact smooth Riemannian manifold and $U\subset M$ an open set. Let 
$f\colon U\to M$ be a diffeomorphism onto its image and $\L\subset U$ a compact invariant set on which $f$ is topologically transitive and which has a partially hyperbolic invariant splitting $T_\L M = E^u \oplus E^c \oplus E^s$, where $E^u$ and $E^s$ are uniformly expanding and contracting, respectively.  Suppose moreover that the center distribution $E^c$ is integrable to a continuous foliation with smooth leaves which are compact and that $\sup_n \|Df^n|E^c\|<\infty$.  Then Theorem \ref{thm:main} gives the following result.

\begin{theorem}\label{thm:cpt-Wc}
Let $f,\L$ be as above and let $\ph\colon\L\to\RR$ be a H\"older continuous function that is constant on each center leaf. Then there is $Q>0$ such that for every $x\in\L$, the measures $m_x^\C = m_x^\C(\cdot, P(\ph))$ on $X = \Vl^u(x) \cap \L$ satisfy
\begin{enumerate}
\item $Q^{-1}< m^{\C}_{x}(\Vl^u(x) \cap \L)< Q$;         
\item $Q^{-1}<\Jac\pi_{xy}(z)< Q$ for all rectangles $R$ and $x,y\in R$, $z\in V_R^u(x)$;
\item the sequence \eqref{seq-meas} converges to a unique equilibrium measure 
$\mu=\mu_\ph$ for $\varphi$;
\item for $\mu$-almost every $y\in M$ we have that the conditional measure 
$\mu^u_{y}$ is equivalent to $m^{\C}_{y}$, with uniform bounds as in \eqref{eqn:mu-m};
\item $\mu$ has a local product structure.        
\end{enumerate} 
\end{theorem}

In particular, Theorem \ref{thm:cpt-Wc} applies when $f$ is a  topologically transitive skew product over a uniformly hyperbolic set that acts along the fibers by isometries, and when $\ph$ is a H\"older continuous potential function that is constant along fibers.

\begin{remark}\label{rmk:base}
In this skew product case one can also give a proof of existence and uniqueness of equilibrium measures by considering the dynamics of the factor map $g$ on the original uniformly hyperbolic set with respect to the H\"older continuous potential $\Phi(x)=\ph(x,\cdot)$.  This has a unique equilibrium measure $\nu$ by classical results, and using topological transitivity one can argue that there is exactly one invariant measure on $\L$ that projects to $\nu$,\footnote{As described to us by Federico Rodriguez Hertz, the idea is to show that the conditional measures on fibers must be given by Haar measure of a compact group acting transitively on fibers.} 
which must be the unique equilibrium measure for $\ph$.  We note, though, that the other properties of the equilibrium measure $\mu_{\ph}$ stated in Theorem \ref{thm:cpt-Wc} are new.
\end{remark}

We describe a specific example of a partially hyperbolic diffeomorphism $f$ for which 
\begin{enumerate}
\item the central distribution $E^c$ integrates to a continuous foliation with smooth compact leaves;
\item $\|Df|E^c\|\le 1$;
\item $f$ is {\bf not} topologically conjugate via a H\"older continuous homeomorphism to a skew product. 
\end{enumerate}
For a fixed $\alpha\in (0,1)\setminus\mathbb{Q}$ consider a transformation 
$B\colon \mathbb{T}^3\to\mathbb{T}^3$ of the $3$-torus given by
\[
B(x,y,z):= (2x+y, x+y, z+\alpha).   
\]
One can easily see that $B$ commutes with $a(x,y,z)=(-x,-y,z+\frac{1}{2})$ and hence induces a map on $M=\mathbb{T}^3/a$. This map is topologically transitive and partially hyperbolic with center foliation being the Seifert fibration of $M$ (for definition and constructions of Seifert fibrations see for example \cite{Orlik}). Consequently, 
$B\colon M\to M$ is an example of a partially hyperbolic diffeomorphism to which Theorem \ref{thm:cpt-Wc} applies and which is not topologically conjugate to a skew product.\footnote{We would like to thank Andrey Gogolyev for this example.} 

Finally, we give an example where \ref{C1} fails and there are multiple equilibrium measures, even though the growth along $\Ecs$ is subexponential.

\begin{example}\label{eg:no-C1}
Consider the linear flow on the $2$-torus $\mathbb{T}^2$ generated by the system of differential equations: 
$$
\dot{x}=\alpha ~\text{ and }~\dot{y}=\beta
$$ 
for some positive numbers $\alpha, \beta$ whose ratio is irrational. Choose a small number $t_0>0$. One can find a function $\kappa\colon [0,1]\to [0,1]$ which is $C^{\infty}$ except at the origin and satisfies
\begin{enumerate}
\item $\kappa(0)=0$ and $\kappa(t)>0$ for $t\neq 0$;
\item $\kappa(t)=1$ for $t\geq t_0$;
\item $\int_{0}^{1}\frac{1}{\kappa(t)}dt <\infty$.
\end{enumerate}
Define a function $\psi_0\colon \mathbb{T}^2\to [0,1]$ by
\[
\psi_0(x,y)= \begin{cases}
\kappa(\sqrt{x^2+y^2}) & \text{ if } (x,y)\in B(0,t_0),\\
1 & \text{ otherwise}      \end{cases}
\]
and then choose a point $p=(x_0,y_0)\in\mathbb{T}^2$ and introduce a transformation 
$\chi\colon\mathbb{T}^2 \to \mathbb{T}^2$ given by $\chi=\chi_2\circ \chi_1$ where
\[
\chi_1(x,y)=(x-x_0,y-y_0), ~ \text{ and }~ \chi_2(x,y)= (x- y\alpha/\beta, y/\beta). 
\]
Roughly speaking, $\chi$ transforms the flow lines near $p$ into vertical lines near the origin.

Finally, consider a function $\psi=\psi_0\circ \chi$ and the vector field 
\[ 
\dot{x}=\psi \alpha ~ \text{ and } ~ \dot{y}= \psi \beta. 
\]
The corresponding flow $g_t$ has a fixed point at $p$ and therefore taking a vector 
$v\in T_p\mathbb{T}^2$ in the direction of the flow, we obtain for $g=g_1$ that 
$\|Dg^n_p v\|$ is unbounded. On the other hand, one can show that
$\lim_{n\to\infty}\frac1n\log \|Dg^n_p\|=0$. In addition, since $\psi_0(x,y)=\psi_0(x,-y)$,
one can show that for any $x\neq p$ there exists $L(x)>0$ such that 
$\|Dg^n_x u\| \leq L(x) \|u\|$ for any $u\in T_x \mathbb{T}^2$.

Property (3) of function $\kappa$ guarantees that the map $g$ preserves a probability measure $m_{\kappa}$ which is absolutely continuous with respect to area. Another invariant measure for $g$ is the delta measure at $p$, $\delta_p$.

Let now $A\colon \mathbb{T}^2\to \mathbb{T}^2$ be a hyperbolic toral automorphism and let $f\colon\mathbb{T}^4\to\mathbb{T}^4$ be given by $f(x,y) = (Ax,gy)$, where 
$x,y\in \mathbb{T}^2$. Then $f$ is partially hyperbolic on $\mathbb{T}^4$ with $E^u$ coming from the unstable eigenspace of $A$ and $\Ecs = E^s\oplus\mathbb{T}^2$ where $E^s$ is the stable eigenspace of $A$. Condition \ref{C3} is clearly satisfied and \ref{C2} holds because $A$ is topologically mixing and $g$ is topologically transitive. However, \ref{C1} fails and so does the conclusion of Theorem \ref{thm:main} for $\ph=0$: writing $m$ for Lebesgue measure on $\mathbb{T}^2$, the measures $m\times m_{\kappa}$ and 
$m\times \delta_p$ are both measures of maximal entropy for $f$.
\end{example}

Note that the above direct product construction of the map $f$ together with Theorem \ref{thm:cpt-Wc} allow us to obtain a new proof of the well known result that if 
$g\colon M\to M$ is a topologically transitive isometry, then $g$ is uniquely ergodic.

\part{Proofs}
\section{Basic properties of reference measures}\label{sec:measures}

Now we begin to prove the results from \S\ref{sec:main-results}, starting with Theorem \ref{thm:finite} in this section, and the remaining results in \S\S\ref{sec:scaling}--\ref{sec:proof-of-main}.

Statement \ref{fin4} of Theorem \ref{thm:finite} is immediate from the definitions.  Statement \ref{2} is proved in \S\ref{sec:Borel}.  Most of the rest of the section is devoted to the following result, which is proved in \S\S\ref{sec:unif-trans}--\ref{sec:uniform-control}. 

For convenience, given $x\in \Lambda$ and $\delta \in (0,\tau)$ we will write 
\[
B_\Lambda(x,\delta) = B(x,\delta) \cap \Lambda\quad\text{and}\quad 
B^u_\Lambda(x,\delta) := B^u(x,\delta) \cap \Lambda = B(x,\delta) \cap \Vl^u(x) \cap \Lambda.
\]
\begin{proposition}\label{prop:uniform}
For every $r_1\in (0,\tau)$ and $r_2 \in (0,\tau/3]$ there is $Q_1>1$ such that for every $x\in \Lambda$ and $n\in\NN$ we have
\begin{equation}\label{eqn:uniform}
Q_1^{-1}e^{nP(\ph)}\le\Zsep_n(B_\L^u(x,r_1),\ph,r_2)\le Q_1e^{nP(\ph)}.
\end{equation}
\end{proposition}

In the course of the proof of Proposition \ref{prop:uniform}, we establish Statement \ref{fin1} of Theorem \ref{thm:finite}; see \S\ref{sec:correct-growth-rate}.  Then in \S\ref{sec:finite} we prove Statement \ref{fin3}.

Throughout, we recall that $\Lambda$ is a compact $f$-invariant set that is partially hyperbolic in the broad sense, on which \ref{C1} and \ref{C2} are satisfied (so $\Ecs$ is Lyapunov stable and $f|\L$ is topologically transitive) and the local product structure condition \ref{C3} holds. We also assume that $\ph\colon \L\to \RR$ is a potential function satisfying the $u$- and $\cs$-Bowen properties with constants $Q_u$ and $Q_\cs$, as in \S\ref{assumptions}. Recall that occasionally we use the following notation: given $A,B,C,a\ge 0$, we write $A=C^{\pm a}B$ as shorthand to mean $C^{-a}B\le A\le C^a B$.

Many of the techniques used in the proof of Proposition \ref{prop:uniform} are adapted from Bowen's paper \cite{rB745}; the underlying principle is that if $Z_n$ is a `nearly multiplicative' sequence of numbers satisfying $Z_{n+k} = Q^{\pm1}Z_n Z_k$ for some $Q$ independent of $n,k$, then $P = \lim_{n\to\infty}\frac1n\log Z_n$ exists and 
$Z_n=Q^{\pm 1}e^{nP}$ (see \cite[Lemmas 6.2--6.4]{CPZ} for a proof of this elementary fact).
The proofs here are more involved than those in \cite{rB745} because the partition sums in \eqref{eqn:uniform} actually depend on $x,r_1,r_2$, so we must control how they vary when these parameters are changed.

\subsection{Reference measures are Borel}\label{sec:Borel}

An outer measure $m$ on a metric space $(X,d)$ is said to be a \emph{metric outer measure} if $m(E\cup F) = m(E) + m(F)$ whenever $d(E,F) := \inf \{d(x,y) : x\in E, y\in F\} > 0$.  By \cite[\S2.3.2(9)]{Fed}, every metric outer measure is Borel, so to prove Statement \ref{fin2} it suffices to show that $m_x^\C$ is metric.  To this end, note that given $x\in \L$ and $y\in X=\Vl^u(x)\cap \L$, we have $\diam B_n^u(y,r) \leq r\lambda^n \to 0$ as $n\to\infty$, and thus for any $E,F\subset X$ with $d(E,F)>0$, there is $N\in \NN$ such that $B_n(y,r) \cap B_k(z,r) = \emptyset$ whenever $y\in E$, $z\in F$, and $k,n\geq N$.  In particular, for this (and larger) $N$, every $\mathcal{G}$ used in \eqref{eqn:mCP2} splits into two disjoint subsets, one that covers $E$ and one that covers $F$, which implies that $m_x^\C(E\cup F) = m_x^\C(E) + m_x^\C(F)$, so $m_x^\C$ is a metric outer measure.

\subsection{Uniform transitivity of local unstable leaves}\label{sec:unif-trans}

We will need the following consequence of topological transitivity.

\begin{lemma}\label{lem:leaf-iterates}
For every $\delta>0$ there is $n\in\NN$ such that for every $x,y\in \Lambda$, there is 
$0\le k\le n$ such that $f^k(B_\Lambda^u(x,\delta))\cap B^\cs(y,\delta)\ne\emptyset$.
\end{lemma}
\begin{proof}
Let $\Delta_\eps = \{(x,y)\in \Lambda\times\Lambda : d(x,y)\leq \eps\}$, where $\eps>0$ is small enough so that the Smale bracket $[x,y] = \Vl^\cs(x) \cap \Vl^u(y)$ defines a continuous map $\Delta_\eps\to \Lambda$.  The function $G(x,y) = \max\{d([x,y],x),d([x,y],y)\}$ is continuous on $\Delta_\eps$ and vanishes on the diagonal $\Delta_0$.  Thus there is $\delta_1\in (0,\delta/2)$ such that $d(x,y)<\delta_1$ implies $G(x,y)<\delta/2$, and similarly there is $\delta_2\in (0,\delta_1/2)$ such that $d(x,y)<\delta_2$ implies $G(x,y)<\delta_1/2$.

Now fix $x\in \Lambda$ and let $U = B_\Lambda(x,\delta_2)$.  Given $n\in \NN$, let
\[
\gamma_n := \sup\Big\{\gamma>0 : \text{there exists $y\in \L$ such that } B(y,\gamma) \cap \bigcup_{k=0}^n f^k(U) = \emptyset \Big\}.
\]
If $\gamma_n\not\to 0$, then there are $y_n\in \Lambda$ and $\gamma>0$ such that $B_\Lambda(y_n,\gamma)\subset \Lambda \setminus \bigcup_{k=0}^n f^k(U)$ for all $n$, and thus any limit point $y=\lim_{j\to\infty} y_{n_j}$ has 
$B(y,\gamma) \cap f^k(U)=\emptyset$ for all $k\in\NN$, contradicting topological transitivity of $f|\Lambda$. Thus $\gamma_n\to 0$, and in particular, 
$B(y,\delta_2)\cap\bigcup_{k=0}^n f^k(U)\ne\emptyset$.  

\begin{figure}[htbp]
\includegraphics[width=.9\textwidth]{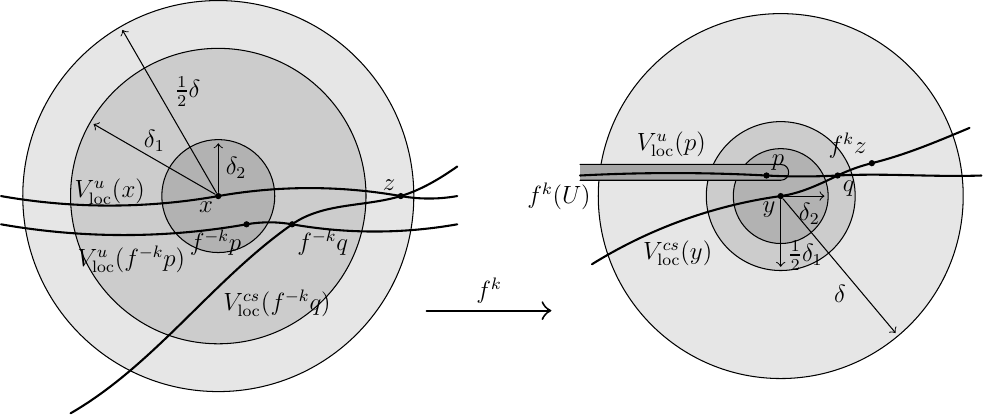}
\caption{Proving Lemma \ref{lem:leaf-iterates}.}
\label{fig:leaf-iterates}
\end{figure}

Now given $x,y\in \Lambda$, this argument gives $k\in [0,n]$ and $p\in f^k(U) \cap B(y,\delta_2)$; see Figure \ref{fig:leaf-iterates}.  Let $q=[y,p]$, so $q\in B^\cs(y,\delta_1/2) \cap B^u(p,\delta_1/2)$ by our choice of $\delta_2$.  It follows that 
\[
f^{-k}(q) \in B^u(f^{-k}(p),\delta_1/2) \subset B(x,\delta_1/2+\delta_2) \subset B(x,\delta_1).
\]
Now by our choice of $\delta_1$, we have $z := [f^{-k}(q),x]\in B^\cs(f^{-k}(q),\delta/2) \cap B^u(x,\delta/2)$, so $f^k(z) \in B^\cs(q,\delta/2) \subset B^\cs(y,\delta)$, which proves the lemma.
\end{proof}

\subsection{Preliminary partition sum estimates}
Now we need to compare the partition sums $\Zspan_n(B_\Lambda^u(x,r_1),\ph,r_2)$ and $\Zsep_n(B_\Lambda^u(x,r_1),\ph,r_2)$ from \eqref{eqn:part-sum} for various 
$x\in \Lambda$ and $0<r_1,r_2\leq \tau$. It will be useful to note that given $x\in\Lambda$ and $y\in B_n^u(x,r_2)$, we have
\[
d_n(x,y) = \max_{0\leq k\leq n} d(f^k(x), f^k(y)) = d(f^n(x), f^n(y)),
\]
so that in particular, $B_n^u(x,r_2) = f^{-n}(B^u(f^n x,r_2))$.

\subsubsection{Comparing spanning and separated sets}
\begin{lemma}\label{lem:spansep}
For every $x\in \Lambda$, $n\in \NN$, and $r_1,r_2 \in (0,\tau]$ we have
\[
\begin{aligned}
\Zspan_n(B_\Lambda^u(x,r_1),\ph,r_2) &\leq \Zsep_n(B_\Lambda^u(x,r_1),\ph,r_2) \leq e^{Q_u} \Zspan_n(B_\L^u(x,r_1),\ph,r_2/2).
\end{aligned}
\]
\end{lemma}
\begin{proof}
If $E\subset B^u(x,r_1)$ is a maximal $(n,r_2)$-separated set, then it must be an
$(n,r_2)$-spanning set as well, otherwise we could add another point to it while
remaining $(n,r_2)$-separated.  Thus
\[
\Zspan_n(B_\L^u(x,r_1),\ph,r_2) \leq \sum_{z\in E} e^{S_n\ph(z)} \leq
\Zsep_n(B_\L^u(x,r_1),\ph,r_2),
\]
which proves the first inequality.
Now let $F\subset B_\L^u(x,r_1)$ be \emph{any} $(n,r_2/2)$-spanning set.  Given any
$(n,r_2)$-separated set $E\subset B_\L^u(x,r_1)$, every $z\in E$ has a point
$y(z)\in F \cap B_n^u(z,r_2/2)$, and the map $z\mapsto y(z)$ is injective, so
\[
\Zsep_n(B_\L^u(x,r_1),\ph,r_2)
\le\sum_{z\in E} e^{S_n\ph(z)}
\le\sum_{z\in E} e^{S_n\ph(y(z))+Q_u}\le e^{Q_u}\sum_{y\in F}e^{S_n\ph(y)}.
\]
Taking an infimum over all such $F$ gives the second inequality.
\end{proof}

\subsubsection{Changing leaves}
Let $\eps>0$ be such that $[y,z]$ exists whenever $d(y,z) < \eps$.  Without loss of generality we assume that $\eps \leq \tau/3$. The following two statements allow us to compare partition sums along different leaves.

\begin{lemma}\label{lem:comparison}
Given any $r_1\in (0,\eps)$ and $r_2\in (0,\tau/3]$, there are $n_1 = n_1(r_1)\in\NN$ and $Q_2=Q_2(r_1,r_2)>0$ such that
given any $x,y\in \Lambda$ and $n\ge n_1$ we have
\begin{align}
\label{eqn:Zspan-xy}
\Zspan_{n-n_1}(B_\L^u(y,r_1),\ph,3r_2) &\leq Q_2\Zspan_n(B_\L^u(x,r_1),\ph,r_2).
\end{align}
\end{lemma}
\begin{proof}
Choose $\eps>0$ small enough that if $x\in \Lambda$, $y\in B_\Lambda^u(x,r_2)$, and $z\in B_\Lambda^\cs(x,\eps)$, then $\Vl^\cs(y) \cap B^u(z,2r_2) \neq\emptyset$; then let $\delta=\delta(\eps)>0$ be given by \ref{C1}.
By Lemma \ref{lem:leaf-iterates}, there is $n_1\in \NN$ such that for every $x,y\in \Lambda$ there is $k=k(x,y)\in [0,n_1]$ with $f^k(B_\L^u(x,r_1))\cap B^\cs(y,\delta)\neq\emptyset$. 

Now given $x,y\in \Lambda$, $n\geq n_1$, and any $(n,r_2)$-spanning set $E\subset B_\L^u(x,r_1)$, we will produce an $(n-n_1,3r_2)$-spanning set $E'\subset B_\L^u(y,r_1)$. To this end, let $U = \bigcup_{z\in B_\L^u(y,r_1+r_2)} \Vl^{cs}(z)$, and let $\pi\colon U\to B_\L^u(y,r_1+r_2)$ be projection along center-stable leaves.  We first claim that
\[
E_1 := \bigcup_{k=0}^{n_1} \pi(f^k(E) \cap U)
\subset B_\L^u(y,r_1+r_2)
\]
has the property that
\begin{equation}\label{eqn:E1}
\bigcup_{w\in E_1} B_{n-n_1}^u(w,2r_2) \supset B_\L^u(y,r_1).
\end{equation}
Indeed, given $z\in B_\L^u(y,r_1)$, by the choice of $n_1$ there are $k\in [0,n_1]$ and $p\in B_\L^u(x,r_1)$ such that $f^k(p) \in B^\cs(z,\delta)$, and since $E$ is an $(n,r_2)$-spanning set in $B_\L^u(x,r_1)$, we can choose a point $q\in E \cap B_n^u(p,r_2)$.  Then
\[
f^k(q) \in B_{n-k}^u(f^k(p),r_2) \subset B_{n-n_1}^u(f^k(p),r_2),
\]
so for all $0\leq j < n-n_1$ we have $d(f^j(f^k q), f^j(f^k p)) < r_2$. By \ref{C1} we also have $d(f^j(f^k q), f^j(\pi f^k q)) \leq \eps$, and thus our choice of $\eps$ gives 
$d(f^j(\pi f^k q),f^j(\pi f^k p)) < 2r_2$ for all such $j$.  Since $\pi(f^k p) = z$, we conclude that
$\pi(f^k(q)) \in B_{n-n_1}^u(z,2r_2)$, 
which
 proves \eqref{eqn:E1}.  To produce $E' \subset B_\L^u(y,r_1)$, consider the sets
\[
E_2 := E_1 \cap B_\L^u(y,r_1),\quad
E_3 := \{z\in E_1 \setminus E_2 : B_{n-n_1}^u(z,r_2) \cap B_\L^u(y,r_1) \neq\emptyset\}.
\]
Define a map $T\colon E_3 \to B_\L^u(y,r_1)$ by choosing for each $z\in E_3$ some 
$T(z) \in B_{n-n_1}^u(z,r_2) \cap \L$. Then $E' = E_2 \cup T(E_3)$ is an $(n-n_1,3r_2)$-spanning set in $B_\L^u(y,r_1)$.

Given $k\in [0,n_1]$ and $j\in \{2,3\}$, let $E_j^k = \{p\in E : \pi(f^k(p)) \in E_j\}$, so
\begin{equation}\label{eqn:E'-union}
E' = \bigcup_{k=0}^{n_1} \pi f^k(E_2^k) \cup T(\pi f^k(E_3^k))
\end{equation}
By the $\cs$-Bowen property, for each $p\in E$ we have
\begin{equation}\label{eqn:p-pi}
|S_{n-n_1}\ph(\pi(f^k(p))) - S_{n-n_1}\ph(f^k(p))| \leq Q_\cs.
\end{equation}
Since $T(z)\in B_{n-n_1}^u(z,r_2)$, the $u$-Bowen property gives
\begin{equation}\label{eqn:zTz}
|S_{n-n_1}\ph(T(z)) - S_{n-n_1}\ph(z)| \leq Q_u.
\end{equation}
Using the $(n-n_1,3r_2)$-spanning property of $E'$ together with \eqref{eqn:E'-union}--\eqref{eqn:zTz}, we obtain
\[
\begin{aligned}
\Zspan_{n-n_1}&(B_\L^u(y,r_1),\ph,3r_2)\leq 
\sum_{z\in E'} e^{S_{n-n_1}\ph(z)} \\
&\leq
\sum_{k=0}^{n_1}\bigg(\sum_{p\in E_2^k} e^{S_{n-n_1}\ph(\pi(f^k(p)))} +
\sum_{p\in E_3^k} e^{S_{n-n_1}\ph(T\pi(f^k(p)))}\bigg) \\
&\leq \sum_{k=0}^{n_1} \bigg(\sum_{p\in E_2^k} e^{S_{n-n_1}\ph(f^kp) + Q_\cs} + \sum_{p\in E_3^k} e^{S_{n-n_1} \ph(f^k(p0) + Q_\cs + Q_u}\bigg) \\
&\leq (n_1+1) \sum_{p\in E} e^{S_n\ph(p) + Q_\cs+Q_u + n_1 \|\ph\|}.
\end{aligned}
\]
Putting $Q_2:= (n_1 +1) e^{Q_\cs+Q_u + n_1\|\ph\|}$ and taking an infimum over all $E$ proves \eqref{eqn:Zspan-xy}.
\end{proof}

\subsubsection{Changing scales}
\begin{lemma}\label{lem:r3r2}
For every $r_2,r_3\in (0,\tau]$, there is $n_0\in \NN$ such that for every $x\in \Lambda$ and $r_1\in (0,\tau]$, we have
\[
\Zsep_n(B_\L^u(x,r_1), \ph,r_3) \leq e^{n_0\|\ph\|} \Zsep_{n+n_0}(B_\L^u(x,r_1),\ph,r_2).
\]
\end{lemma}
\begin{proof}
Choose $n_0\in \NN$ such that $r_2 \lambda^{n_0} < r_3$, where $\lambda<1$ is as in Proposition \ref{prop:local-mfds}.  Then if $x\in \Lambda$ and $y,z\in \Vl^u(x)$ are such that $d_n(y,z)\geq r_3$, we must have $d_{n+n_0}(y,z)\geq r_2$.  This shows that any $(n,r_3)$-separated subset $E\subset B_\L^u(x,r_1)$ is $(n+n_0,r_2)$-separated.  Moreover, we have
\[
\sum_{y\in E} e^{S_n\ph(y)} \leq e^{n_0\|\ph\|} \sum_{y\in E} e^{S_{n+n_0}\ph(y)},
\]
and taking a supremum over all such $E$ completes the proof.
\end{proof}

\subsubsection{Correct growth rate}\label{sec:correct-growth-rate}
At this point we have enough machinery developed to prove that the leafwise partition sums have the same growth rate as the overall partition sums so that we can use the former to compute the topological pressure in \eqref{eqn:pressure}.  This is not yet quite enough to conclude Proposition \ref{prop:uniform}, but is an important step along the way.*\blfootnote{
The published version of this paper contains an error in the proof of Lemma \ref{lem:P-on-W} (an incorrect deduction involving lim sup and lim inf using Lemma \ref{lem:spansep}). We are grateful to Xue Liu for bringing this issue to our attention. The lemma remains correct as stated in the published paper, and the proof presented here corrects the problem.}

\begin{lemma}\label{lem:P-on-W}
For every $x\in \Lambda$, $r_1 \in (0,\eps)$, and $r_2 \in (0,\tau/3]$, we have
\begin{equation}\label{eqn:Pspansep}
P(\ph) = \lim_{n\to\infty}\frac 1n \log \Zspan_n(B_\L^u(x,r_1),\ph,r_2)\\
= \lim_{n\to\infty}\frac 1n \log \Zsep_n(B_\L^u(x,r_1),\ph,r_2).
\end{equation}
\end{lemma}
\begin{proof}
Given $Y\subset M$, write
\[
\overline{P}^{\mathrm{span}}_{Y}(r_2) := \ulim_{n\to\infty} \frac 1n \log \Zspan_n(Y\cap \Lambda,r_2);
\]
define $\overline{P}^{\mathrm{sep}}_Y(r_2)$, $\underline{P}^{\mathrm{span}}_Y(r_2)$, and $\underline{P}^{\mathrm{sep}}_Y(r_2)$ similarly.  (Since $\ph$ is fixed throughout we omit it from the notation.)

Using Lemma \ref{lem:spansep} and the fact that any $(n,r_2)$-separated subset of $B_\L^u(x,r_1)$ is also an 
$(n,r_2)$-separated subset of $\Lambda$, we get
\[
\underline{P}^{\mathrm{span}}_{B_\L^u(x,r_1)}(r_2)
\leq \overline{P}^{\mathrm{span}}_{B_\L^u(x,r_1)}(r_2)
\leq \overline{P}^{\mathrm{sep}}_{B_\L^u(x,r_1)}(r_2)
\leq \overline{P}^{\mathrm{sep}}_\Lambda(r_2) \leq P(\ph).
\]
Thus to prove Lemma \ref{lem:P-on-W}, it suffices to show that $\underline{P}^{\mathrm{span}}_{B_\L^u(x,r_1)}(r_2) \geq P(\ph)$. Indeed, it will suffice to show that $\underline{P}^{\mathrm{span}}_{B_\L^u(x,r_1)}(r_2) \geq \underline{P}_\L^{\mathrm{span}}(r_3)$ for all $r_3>0$.

To this end, fix $r_3>0$. By Lemma \ref{lem:r3r2}, there is $n_0\in \NN$ such that for every $x\in \L$, we have
\[
\Zsep_n(B_\L^u(x,r_1),\ph,r_3/4) \leq e^{n_0\|\ph\|} \Zsep_{n+n_0}(B_\L^u(x,r_1),\ph,2r_2).
\]
Using this together with Lemma \ref{lem:spansep} gives
\[
\underline{P}^{\mathrm{span}}_{B_\L^u(x,r_1)}(r_2)
\geq \underline{P}^{\mathrm{sep}}_{B_\L^u(x,r_1)}(2r_2)
\geq \underline{P}^{\mathrm{sep}}_{B_\L^u(x,r_1)}(r_3/4)
\geq \underline{P}^{\mathrm{span}}_{B_\L^u(x,r_1)}(r_3/4),
\]
and so we can complete the proof of Lemma \ref{lem:P-on-W} by showing that $\underline{P}^{\mathrm{span}}_{B_\L^u(x,r_1)}(r_3/4) \geq \underline{P}_\L^{\mathrm{span}}(r_3)$ for all $r_3>0$.

For this, we need to use the Lyapunov stability of $\Ecs$ from Condition \ref{C1}
(see also Remark \ref{rmk:Lyap-stab}).
Let $\delta>0$ be given by \ref{C1} with $\eps = r_3/4$, and
consider for each $y\in \Lambda$ the (relatively) open set $U_y := \bigcup_{z\in B_\L^u(y,r_1)} B_\L^\cs(z,\delta)$.  Since $\Lambda$ is compact, we have $\Lambda \subset \bigcup_{i=1}^N U_{y_i}$ for some $\{y_1,\dots, y_N\}$.  Now for any $x\in \Lambda$, Lemma 6.4
gives $(n-n_1,3\eps)$-spanning sets $E_i$ for $B_\L^u(y_i,r_1)$ such that
\[
\sum_{z\in E_i} e^{S_n\ph(z)} \leq Q_2\Zspan_n(B_\L^u(x,r_1),\ph,\eps).
\]
We claim that $E_i$ is an $(n-n_1,4\eps)$-spanning set for $U_{y_i}$.  Indeed, for every $z\in U_{y_i}$ we have $[z,y_i] \in B_\L^u(y_i,r_1)$ and hence there is $p\in E_i$ such that $d_n(p,[z,y_i]) < 3\eps$.  Moreover, $d_n(z,[z,y_i]) < \eps$ using Condition \ref{C1} and the fact that $z \in B^\cs([z,y_i],\delta)$; then the triangle inequality proves the claim.  Now writing $E' = \bigcup_{i=1}^N E_i$, we see that $E'$ is an $(n-n_1,4\eps)$-spanning set for $\Lambda$, and hence,
\[
\Zspan_{n-n_1}(\Lambda,\ph,4\eps) \leq Q_2N \Zspan_n(B_\L^u(x,r_1),\ph,\eps).
\]
Taking logs, dividing by $n$, and sending $n\to\infty$ gives 
$\underline{P}^{\mathrm{span}}_{B_\L^u(x,r_1)}(\eps) \geq \underline{P}_\L^{\mathrm{span}}(4\eps)$, which proves Lemma \ref{lem:P-on-W}.
\end{proof}

\subsection{Uniform control of partition sums}\label{sec:uniform-control}

Now we are nearly ready to use the estimates from the preceding sections to prove Proposition \ref{prop:uniform}.  We need two more lemmas.

Given $n\in \NN$ and $r_1,r_2\in (0,\tau]$, consider the quantity
\begin{equation}\label{eqn:Znu}
Z_n^u(\ph,r_1,r_2) := \sup_{x\in \Lambda} \Zsep_n(B_\L^u(x,r_1),\ph,r_2).
\end{equation}
We have the following submultiplicativity result.

\begin{lemma}\label{lem:submult}
For every $x\in \Lambda$, $r_1,r_2 \in (0,\tau]$, and $k,\ell\in \NN$, we have
\begin{equation}\label{eqn:submult}
\Zsep_{k+\ell}(B_\L^u(x,r_1),\ph,r_2) \leq e^{Q_u} \Zsep_k(B_\L^u(x,r_1),\ph,r_2) 
Z^u_\ell(\ph,r_1,r_2).
\end{equation}
\end{lemma}
\begin{proof}
Given $x\in \Lambda$ and $k,\ell\in \NN$, let $E\subset B_\L^u(x,r_1)$ be a
$(k+\ell,r_2)$-separated set.  Let $E'\subset E$ be a maximal $(k,r_2)$-separated set,
and given $y\in E'$ let $E_y = E \cap B_k^u(y,r_1)$. Then $f^k(E_y)$ is an
$(\ell,r_2)$-separated subset of $f^k(B_k^u(y,r_1)\cap\L)=B_\L^u(f^k(y),r_1)$, and we
conclude that
\begin{align*}
\sum_{y\in E} e^{S_{k+\ell}\ph(y)} &= \sum_{y\in E'} \sum_{z\in E_y}
e^{S_{k+\ell}\ph(z)}
= \sum_{y\in E'} \sum_{z\in E_y} e^{S_k\ph(z)} e^{S_\ell\ph(f^k(z))} \\
&\le\sum_{y\in E'} e^{S_k\ph(y) + Q_u} \sum_{z\in E_y} e^{S_\ell\ph(f^k(z))} \\
&\le e^{Q_u}\Zsep_k(B_\L^u(x,r_1),\ph,r_2)\max_{y\in E'} 
\Zsep_\ell(B_\L^u(f^k(y),r_1),\ph,r_2),
\end{align*}
where the first inequality uses the $u$-Bowen property. Taking a supremum over all choices of $E$ completes the proof.
\end{proof}
Now we can assemble Lemmas \ref{lem:spansep}, \ref{lem:comparison}, \ref{lem:r3r2}, and \ref{lem:submult} into the following result that lets us change parameters $x$ and $r_2$ in partition sums more or less at will.\footnote{With a little more work we could vary $r_1$ as well, but we will not need this.}

\begin{lemma}\label{lem:movable}
For every $r_1\in (0,\eps)$ and $r_2,r_2'\in (0,\tau/3]$ there is $Q_3$ such that for every $x,y\in \Lambda$ and $n\in \NN$, we have
\[
\Zsep_n(B_\L^u(x,r_1),\ph,r_2) \leq Q_3\Zsep_n(B_\L^u(y,r_1),\ph,r_2').
\]
\end{lemma}
\begin{proof}
Let $n_1$ be as in Lemma \ref{lem:comparison} (note that it only depends on $r_1$, not on $r_2$) and let $n_0$ be as in Lemma \ref{lem:r3r2}, with $r_3=r_2'/6$.  Then for all $x,y\in\Lambda$, we have
\begin{alignat*}{3}
\Zsep_n&(B_\L^u(y,r_1),\ph,r_2') \leq e^{Q_u}\Zspan_n(B_\L^u(x,r_1),\ph,r_2'/2) 
\qquad\qquad &&\text{(Lemma \ref{lem:spansep})} \\
&\leq Q_2e^{Q_u} \Zspan_{n+n_1}(B_\L^u(x,r_1),\ph,r_2'/6) 
&&\text{(Lemma \ref{lem:comparison})} \\
&\leq Q_2e^{Q_u} \Zsep_{n+n_1}(B_\L^u(x,r_1),\ph,r_2'/6)
&&\text{(Lemma \ref{lem:spansep})} \\
&\leq Q_2e^{Q_u + n_0\|\ph\|} \Zsep_{n+n_0+n_1}(B_\L^u(x,r_1),\ph,r_2).
&&\text{(Lemma \ref{lem:r3r2})}
\end{alignat*}
By Lemma \ref{lem:submult}, this gives
\[
\Zsep_n(B^u(y,r_1),\ph,r_2')\le Q_2e^{2Q_u + n_0\|\ph\|}\Zsep_n(B^u(x,r_1),\ph,r_2)Z_{n_0+n_1}^u(\ph,r_1,r_2).
\]
Putting $Q_3=Q_2e^{2Q_u+n_0\|\ph\|}Z_{n_0+n_1}^u(\ph,r_1,r_2)$ completes the proof.
\end{proof}
\begin{proof}[Proof of Proposition \ref{prop:uniform}]
For the lower bound, we apply Lemma \ref{lem:submult} iteratively to get
\[
\Zsep_{nk}(B_\L^u(x,r_1),\ph,r_2) \leq e^{(n-1)Q_u} 
Z^u_k(\ph,r_1,r_2)^n.
\]
Taking logs, dividing by $nk$, and sending $n\to\infty$ gives
\[
\frac 1k \log Z^u_k(\ph,r_1,r_2)
\geq -\frac{Q_u}k + P(\ph)
\]
by Lemma \ref{lem:P-on-W}.  Thus for every $x\in \Lambda$ and $k\in \NN$, Lemma \ref{lem:movable} gives
\[
\Zsep_k(B^u(x,r_1),\ph,r_2) \geq Q_3^{-1} Z_k^u(\ph,r_1,r_2)
\geq Q_3^{-1} e^{-Q_u} e^{kP(\ph)},
\]
which proves the lower bound in \eqref{eqn:uniform} by taking $Q_1\ge Q_3e^{Q_u}$.

For the upper bound in \eqref{eqn:uniform}, start by letting $n_2\in \NN$ be such that $r_2 \lambda^{-n_2} \geq r_2 + 2r_1$, where once again $\lambda<1$ is as in Proposition \ref{prop:local-mfds}\ref{leaves-contract}.  Now fix $x\in \Lambda$ and $n\in \NN$, and let $E_0\subset B_\L^u(x,r_1)$ be any $(n,r_2)$-separated set. By Lemma \ref{lem:movable}, for every $z\in \Lambda$ there is a $(n,r_2)$-separated set $G(z) \subset B_\L^u(z,r_1)$ with
\begin{equation}\label{eqn:Gz}
\sum_{p\in G(z)} e^{S_n\ph(p)}\ge Q_3^{-1}\sum_{y\in E_0}e^{S_n\ph(y)}.
\end{equation}
Given $k=1,2,\dots$, construct $E_k\subset B_\L^u(x,r_1)$ iteratively by
\begin{equation}\label{eqn:Ek}
E_{k}=\bigcup_{y\in E_{k-1}}f^{-k(n+n_2)}(G(f^{k(n+n_2)}(y))).
\end{equation}
We prove by induction that $E_k$ is a $((k+1)n + kn_2,r_2)$-separated set.  The case $k=0$ is true by our assumption on $E_0$.  For $k\geq 1$, suppose that the set $E_{k-1}$ is $(kn+(k-1)n_2),r_2)$-separated.  Then given any $p_1,p_2\in E_k$ we have one of the following two cases.
\begin{enumerate}
\item There is $z\in E_{k-1}$ with $p_1,p_2 \in f^{-k(n+n_2)}(G(f^{k(n+n_2)}z))$.  By
the definition of $G(f^{k(n+n_2)}(z))$, this gives
\[
d_{(k+1)n + kn_2}(p_1,p_2) \ge d_n(f^{k(n+n_2)}(p_1),f^{k(n+n_2)}(p_2))\ge r_2.
\]
\item There are $z_1\ne z_2\in E_{k-1}$ such that $p_i\in f^{-k(n+n_2)}(G(f^{k(n+n_2))}(z_i))$ for $i=1,2$. Then $f^\ell(p_i)\in B^u(f^\ell(z_i),r_1)$ for all $0\leq \ell < k(n+n_2)$, and  there is $0\leq j < kn + (k-1)n_2$ such that $f^j(z_2) \in \Vl^u(f^j(z_1))$ and $d(f^j(z_1),f^j(z_2)) > r_2$. By our choice of $n_2$, we have
\[
d(f^{j+n_2}(z_1),f^{j+n_2}(z_2)) > r_2 + 2r_1\ge r_2 + \sum_{i=1}^2 d(f^{j+n_2}(z_i), f^{j+n_2}(p_i)),
\]
and the triangle inequality gives $d(f^{j+n_2}(p_1), f^{j+n_2}(p_2))\ge r_2$.
\end{enumerate}
This completes the induction, and gives the following estimate:
\begin{equation}\label{eqn:Zsep-kn}
\Zsep_{(k+1)n+kn_2}(B_\L^u(x,r_1),\ph,r_2)
\geq \sum_{y\in E_k} e^{S_{(k+1)n + kn_2}\ph(y)}.
\end{equation}
Write $\tilde Z_k$ for the sum on the right-hand side of \eqref{eqn:Zsep-kn}.
The definition of $E_k$ in \eqref{eqn:Ek} gives
\begin{equation}\label{eqn:Zk}
\begin{aligned}
\tilde Z_k
&= \sum_{y\in E_{k-1}}
\Bigg(\sum_{z\in f^{-k(n+n_2)}(G(f^{k(n+n_2)}(y)))}e^{S_{(k+1)n+kn_2}\ph(z)}\Bigg)
\\
&\ge\sum_{y\in E_{k-1}} e^{S_{kn+(k-1)n_2}\ph(z)-Q_u}e^{-n_2\|\ph\|}\sum_{p\in G(f^{k(n+n_2)}(y))}e^{S_n\ph(p)} \\
&\geq Q_3^{-1} e^{-Q_u - n_2\|\ph\|} \tilde Z_{k-1} \sum_{y\in E_0} e^{S_n\ph(y)},
\end{aligned}
\end{equation}
where the last inequality uses \eqref{eqn:Gz}. Writing $Q_4:= Q_3e^{Q_u+n_2\|\ph\|}$ and applying \eqref{eqn:Zk} $k$ times yields
\[
\tilde Z_k \geq Q_4^{-k} \Big(\sum_{y\in E_0} e^{S_n\ph(y)} \Big)^k.
\]
Then \eqref{eqn:Zsep-kn} gives
\[
\Zsep_{(k+1)n+kn_2}(B_\L^u(x,r_1),\ph,r_2)
\geq Q_4^{-k} \Big( \sum_{y\in E_0} e^{S_n\ph(y)} \Big)^k.
\]
Taking logs and dividing by $k$ gives
\[
\frac{n+\frac nk + n_2}{(k+1)n + kn_2} \log\Zsep_{(k+1)n+kn_2}(B_\L^u(x,r_1),\ph,r_2)
\geq -\log Q_4 + \log \sum_{y\in E_0} e^{S_n\ph(y)}.
\]
Sending $n\to\infty$ and taking a supremum over all choices of $E_0$, Lemma \ref{lem:P-on-W} yields 
\[
(n+n_2)P(\ph)\ge-\log Q_4+\log\Zsep_n(B_\L^u(x,r_1),\ph,r_2),
\]
and so $\Zsep_n(B_\L^u(x,r_1),\ph,r_2) \leq Q_4 e^{(n+n_2)P(\ph)}$.  Choosing 
$Q_1\ge Q_4 e^{n_2P(\ph)}$ completes the proof of Proposition \ref{prop:uniform}.
\end{proof}

\subsection{Proof of Theorem \ref{thm:finite}}\label{sec:finite}
Fix $x\in\Lambda$ and set $X:=\Vl^u(x) \cap \Lambda$.
We showed in \S\ref{car.measure} that $m_x^\C$ defines a metric outer measure on $X$, and hence gives a Borel measure. Note that the final claim in Theorem \ref{thm:finite} about agreement on intersections is immediate from the definition. Thus it remains to prove that $m_x^\C(X) \in [K^{-1},K]$, where $K$ is independent of $x$; this will complete the proof of Theorem \ref{thm:finite}.

We start with the following basic fact about local unstable leaves, which follows easily from Proposition \ref{prop:local-mfds}\ref{unif-holder}. 

\begin{lemma}\label{lem:bounded-cover}
For all $r_1,r_2 \in (0,\tau]$ there is $Q_5>0$ such that for all $y\in \Lambda$, there are $k\leq Q_5$ and $z_1,\dots, z_k\in B_\L^u(y,r_2)$ such that $\bigcup_{i=1}^k B_\L^u(z_i,r_1) \supset B_\L^u(y,r_2)$.
\end{lemma}
Together with Proposition \ref{prop:uniform} this leads to the following.

\begin{lemma}\label{lem:sep-in-Bn}
For every $r\in (0,\tau/3]$, there is a constant $Q_6>0$ such that for every $y\in \Lambda$, $n\in \NN$, and $N\geq n$, we have
\[
\Zsep_N(B_n^u(y,r) \cap \Lambda, \ph, r) \leq Q_6 e^{(N-n)P(\ph)} e^{S_n\ph(y)}.
\]
\end{lemma}
\begin{proof}
Let $\ell\in \NN$ be such that $\lambda^\ell < \frac 12$, where $\lambda<1$ is as in Proposition \ref{prop:local-mfds}\ref{leaves-contract}.  Given any $(N,r)$-separated set $E\subset B_n^u(y,r) \cap \Lambda$, we have $d(f^n(z_1), f^n(z_2))<2r$ for every 
$z_1,z_2\in E$, and thus $d(f^{n-\ell}(z_1), f^{n-\ell}(z_2))<r$, so $f^{n-\ell}(E)$ is an 
$(N-n+\ell,r)$-separated subset of 
$f^{n-\ell}(B_{n}^u(y,r)\cap\L)\subset B_\L^u(f^{n-\ell}(y),r)$.

Applying Proposition \ref{prop:uniform} with $r_1 = \eps/2$ and $r_2=r$, gives 
$Q_1= Q_1(r)$ such that
\[
\Zsep_{N-n+\ell}(B_\L^u(f^{n-\ell}(y),\eps/2),\ph,r) \leq Q_1e^{(N-n+\ell)P(\ph)},
\]
and so, applying Lemma \ref{lem:bounded-cover} with $r_1 = \eps/2$ and $r_2=4$, gives $Q_5=Q_5(r)$ such that 
\[
\sum_{z\in f^{n-\ell}(E)} e^{S_{N-n+\ell}\ph(z)} 
\le\Zsep_{N-n+\ell}(B_\L^u(f^{n-\ell}(y),r),\ph,r) \leq Q_5Q_1e^{(N-n+\ell)P(\ph)}.
\]
Thus we get
\begin{align*}
\sum_{z\in E}e^{S_N\ph(z)} &=\sum_{z\in E}e^{S_{n-\ell}\ph(z)} 
e^{S_{N-n+\ell}\ph(f^{n-\ell}(z))} \\
&\le e^{Q_u}e^{S_{n-\ell}\ph(y)}\sum_{z\in f^{n-\ell}(E)}e^{S_{N-n+\ell}\ph(z)} \\
&\le e^{Q_u}e^{\ell\|\ph\|}e^{S_n\ph(y)}Q_5Q_1e^{(N-n)P(\ph)}e^{\ell P(\ph)},
\end{align*}
and so putting $Q_6=e^{Q_u}e^{\ell(\|\ph\|+P(\ph))}Q_5Q_1$ proves the result.
\end{proof}
Now we can complete the proof of Theorem \ref{thm:finite}. Fix $r_1 \in (0,\tau)$, and note that Proposition \ref{prop:uniform} and Lemmas \ref{lem:spansep} and \ref{lem:bounded-cover}
apply with $r_2 = r$. We will find $Q_7>0$ such that
\begin{equation}\label{eqn:Q5}
Q_7^{-1} \leq m_x^\C(B_\L^u(x,r_1)) = m_x^\C(B^u(x,r_1)) \leq Q_7
\end{equation}
for every $x\in \Lambda$; then Lemma \ref{lem:bounded-cover} will complete the proof of the theorem by taking $K=Q_7 Q_5$.

For the upper bound in \eqref{eqn:Q5}, let $Q_1$ be given by Proposition \ref{prop:uniform} with $r_2=r$.
By Lemma \ref{lem:spansep} and Proposition \ref{prop:uniform}, for every $N\in\NN$ there is an $(N,r)$-spanning set $E_N\subset B_\L^u(x,r_1)$ with 
$\sum_{y\in E_N}e^{S_N\ph(y)}\leq Q_1e^{NP(\ph)}$. Then \eqref{car0} gives
\begin{equation}\label{eqn:upper}
m_x^\C(B_\L^u(x,r_1))\le\lim_{N\to\infty}\sum_{y\in E_N}e^{-NP(\ph)}e^{S_N\ph(y)}\le Q_1.
\end{equation}
For the lower bound in \eqref{eqn:Q5}, let $\{(y_i,n_i)\}_i \subset (\Vl^u(x) \cap \Lambda) \times \NN $ be any finite or countable set such that $B_\L^u(x,r_1) \subset \bigcup_i B_{n_i}^u(y_i,r)$. By compactness, there is $k \in\NN$ such that 
$\overline{B_\L^u(x,r_1/2)} \subset \bigcup_{i=1}^k B_{n_i}^u(y_i,r)$.  
Fix $N\ge\max\{n_1,\dots,n_k\}$ and for each 
$1\le i\leq k$, let $E_i\subset B_{n_i}^u(y_i,r)\cap\Lambda$ be a maximal $(N, r)$-separated set. Then $\bigcup_{i=1}^k E_i$ is an $(N, r)$-spanning set for $B_\L^u(x,r_1/2)$, and we conclude that
\begin{equation}\label{eqn:sumsum}
\begin{aligned}
\sum_{i=1}^k \sum_{z\in E_i} e^{S_N\ph(z)}&\ge\Zspan_N(B^u(x,r_1/2),\ph, r)  \\
&\ge e^{-Q_u} \Zsep_N(B_\L^u(x,r_1/2),\ph,2 r)\ge Q_1^{-1} e^{-Q_u} e^{NP(\ph)},
\end{aligned}
\end{equation}
where the second inequality uses Lemma \ref{lem:spansep}, and the third uses Proposition \ref{prop:uniform} with $Q_1=Q_1(r_1/2,2 r)$.
 
Now we can use Lemma \ref{lem:sep-in-Bn} to get the bound
$$
\sum_{z\in E_i} e^{S_N\ph(z)}\le\Zsep_N(B_{n_i}^u(y_i,r)\cap\Lambda,\ph, r) \\
\le Q_6 e^{(N-n_i) P(\ph)} e^{S_{n_i} \ph(y_i)}
$$
for each $1\leq i\leq k$. Summing over $i$ and using \eqref{eqn:sumsum} gives 
\begin{align*}
Q_1^{-1} e^{-Q_u} e^{NP(\ph)} \leq \sum_{i=1}^k \sum_{z\in E_i} e^{S_{N}\ph(z)}
\leq \sum_{i=1}^k Q_6 e^{(N-n_i) P(\ph)} e^{S_{n_i} \ph(y_i)},
\end{align*}
and dividing both sides by $e^{NP(\ph)}$ yields 
\[
Q_1^{-1} e^{-Q_u} \leq Q_6 \sum_{i=1}^k e^{-n_i P(\ph)} e^{S_{n_i} \ph(y_i)}.
\]
Taking an infimum over all choices of $\{(y_i,n_i)\}_i$ and sending $N\to\infty$ gives $m_x^\C(B_\L^u(x,r_1))\ge Q_1^{-1}e^{-Q_u}Q_6^{-1}$.  Thus we can prove \eqref{eqn:Q5} and complete the proof of Theorem \ref{thm:finite} by putting 
$Q_7 = \max\{Q_1, Q_1e^{Q_u}Q_6\}$.

\section{Behavior of reference measures under iteration and holonomy}\label{sec:scaling}

\subsection{Proof of Theorem \ref{thm:Gibbs}}\label{sec:Gibbs}

We will prove that for every Borel $A \subset \Vl^u(f(x)) \cap \L$, 
\begin{equation}\label{eqn:scaling}
m_{f(x)}^\C(A) = \int_{f^{-1}A} e^{P(\ph) - \ph(y)} \,dm_x^\C(y),
\end{equation}
which shows that $f_*^{-1} m_{f(x)}^\C \ll m_x^\C$ and that the Radon--Nikodym derivative is $g := e^{P(\ph)-\ph}$.   Given such an $A$, we approximate the integrand on the right-hand side of \eqref{eqn:scaling} by simple functions; for every $T\in\NN$ there are real numbers
\[
\inf_{y\in f^{-1}(A)}g(y)=a_1^T < a_2^T < \cdots < a_T^T = \sup_{y\in f^{-1}(A)}(g(y) + 1)
\]
and disjoint sets
\[
E_i^T := \{y\in f^{-1}(A) : a_i^T \leq g(y) < a_{i+1}^T \} \text{ for } 1\leq i < T
\]
such that $f^{-1}(A) = \bigcup_{i=1}^{T-1} E_i^T$;   since the union is disjoint we have
\begin{equation}\label{eqn:limit-of-simple}
\int_{f^{-1}(A)} g(y)\,dm_x^\C(y) = \lim_{T\to\infty} \sum_{i=1}^{T-1} a_i^T m_x^\C(E_i^T) = \lim_{T\to\infty} \sum_{i=1}^{T-1} a_{i+1}^T m_x^\C(E_i^T).
\end{equation}
To prove \eqref{eqn:scaling}, start by using the first equality in \eqref{eqn:limit-of-simple} and the definition of $m_x^\C$ in \eqref{car0} to write
\begin{equation}\label{eqn:scaling-1}
\int_{f^{-1}(A)} g\,dm_x^\C = \lim_{T\to\infty} \sum_{i=1}^{T-1} a_i^T \lim_{N\to\infty} \inf \sum_j e^{-n_jP(\ph)} e^{S_{n_j}\varphi(z_j)} ,
\end{equation}
where the infimum is taken over all collections $\{B^u_{n_j}(z_j,r)\}$ of 
$u$-Bowen balls with $z_j\in \Vl^u(x) \cap \Lambda$, $n_j\ge N$ that cover $E_i^T$.  Without loss of generality we can assume that 
\begin{equation}\label{eqn:intersects}
B^u_{n_j}(z_j,r) \cap E_i^T \neq\emptyset \text{ for all } j.
\end{equation}
Consider the quantity
\begin{equation}\label{eqn:R}
R_N := \sup \{ |\ph(y) - \ph(z)| : y\in \Lambda, z\in B^u_N(y,\tau/3)\},
\end{equation}
and note that $R_N\to 0$ as $N\to \infty$ using uniform continuity of $\ph$ together with the fact that $\diam B^u_N(y,\tau/3) \leq \tau\lambda^N \to 0$.  Now by \eqref{eqn:intersects} we have 
\[
a_i^T \geq g(z_j) e^{-R_N} = e^{-R_N} e^{P(\ph)} e^{-\ph(z_j)}
\] 
for all $j$, and thus \eqref{eqn:scaling-1} gives
\[
\int_{f^{-1}(A)} g \,dm_x^\C
\geq \lim_{T\to\infty} \sum_{i=1}^{T-1} \lim_{N\to\infty} \inf \sum_j e^{-R_N} e^{-(n_j -1 )P(\ph)} e^{S_{n_j-1} \ph(f(z_j))} ,
\]
where again the infimum is taken over all collections $\{B^u_{n_j}(z_j,r)\}$ of 
$u$-Bowen balls with $z_j\in \Vl^u(x) \cap \Lambda$, $n_j\ge N$ that cover $E_i^T$.  Observe that to each such collection there is associated a cover of $f(E_i^T)$ by the $u$-Bowen balls $f(B^u_{n_j}(z_j,r)) = B^u_{n_j-1}(f(z_j),r)$, and vice versa, so we get
\[
\int_{f^{-1}(A)} g(y)\,dm_x^\C(y) \geq \lim_{T\to\infty} \sum_{i=1}^{T-1} m_{f(x)}^\C(f(E_i^T)) = m_{f(x)}^\C(A).
\]
The reverse inequality is proved similarly by replacing $a_i^T$ with $a_{i+1}^T$ in \eqref{eqn:scaling-1} and using the second equality in \eqref{eqn:limit-of-simple}.  This completes the proof of \eqref{eqn:scaling}, and hence of Theorem \ref{thm:Gibbs}.

\subsection{Proof of Corollary \ref{cor:Gibbs}}\label{sec:gibbs}
Iterating \eqref{eqn:scaling}, we obtain
\begin{equation}\label{eqn:iterated-scaling}
m_{f^n(x)}^\C(A) = \int_{f^{-n}(A)} e^{nP(\ph) - S_n\ph(y)}\,dm_x^\C(y)
\end{equation}
for all $A\subset \Vl^u(f^n(x))$.  Fix $\delta \in (0,\tau)$.  Putting $A=B^u(f^n(x),\delta)$ and observing that $f^{-n}(B^u(f^n(x),\delta)) = B_n^u(x,\delta)$, we can use Theorem \ref{thm:finite} and the $u$-Bowen property to get
\begin{multline*}
K \geq m_{f^n(x)}^\C(B^u(f^n(x),\delta)) = \int_{B_n^u(x,\delta)} e^{nP(\ph) - S_n\ph(y)} \,dm_x^\C(y) \\
\geq e^{-Q_u} e^{nP(\ph) - S_n\ph(x)} m_x^\C(B_n^u(x,\delta)).
\end{multline*}
This gives $m_x^\C(B_n^u(x,\delta)) \leq e^{Q_u} K e^{nP(\ph) + S_n\ph(x)}$, proving the upper bound in \eqref{eqn:u-gibbs}.  For the lower bound, let $k\in\NN$ be such that $\delta \lambda^{-k} > \tau$; then $B^u(y,\delta) \supset f^{-k}(\Vl^u(f^k(y)))$ for all $y\in \L$, and again Theorem \ref{thm:finite} and the $u$-Bowen property give
$$
\begin{aligned}
K^{-1}\leq m_{f^{n+k}(x)}^\C(\Vl^u(f^{n+k}(x)))&\le\int_{B_n^u(x,\delta)}e^{(n+k)P(\ph) - S_{n+k} \ph(y)} \,dm_x^\C(y) \\
&\le e^{Q_u} e^{k(P(\ph) + \|\ph\|)}e^{nP(\ph) - S_n\ph(x)} m_x^\C(B_n^u(x,\delta)).
\end{aligned}
$$
This proves Corollary \ref{cor:Gibbs}.

\subsection{Proof of Theorem \ref{thm:holonomy}}\label{sec:holonomy}

Let $k\in \NN$ be such that for every $y\in \L$ there are $z_1,\dots, z_{k} \in B_\L^u(y,2r)$ for which $B_\L^u(y,2r) \subset \bigcup_{i=1}^k B_\L^u(z_i,r)$.
We can choose $\theta>0$ satisfying: if $x,y\in \L$ are such that $d(x,y)<\theta$, and $a,b\in \Vl^u(x) \cap \L$, $p,q\in \Vl^u(y) \cap \L$ are such that $p\in \Vl^\cs(a)$, $q\in \Vl^\cs(b)$, then $d(p,q) \leq d(a,b) + r$. 

Now let $\delta = \delta(\theta)>0$ be given by \ref{C1}, and let $R\subset \L$ be any rectangle with $\diam(R)<\delta$.  Given $y,z\in R$ and $E\subset V_R^u(y)$, we must compare $m_y^\C(E)$ and $m_z^\C(\pi_{yz}(E))$.  
In this case, for any cover $\{B^u_{n_i}(x_i,r) \}$ of $E$ with $x_i \in V_R^u(y)$ and $n_i\in \NN$, we have $d(f^{n_i}x_i, f^{n_i}(\pi_{yz} x_i)) < \eps$ by Condition \ref{C1},
and our choice of $\eps$ gives
\begin{align*}
\pi_{yz}(B^u_{n_i}(x_i,r)\cap \L) &= \pi_{yz}(f^{-n_i}(B_\L^u(f^{n_i}(x_i), r)))
= f^{-n_i}(
\pi_{f^{n_i} x_i, f^{n_i}(\pi_{yz}x_i)}(B_\L^u(f^{n_i}(x_i), r))) \\
&\subset f^{-n_i} (B_\L^u(f^{n_i}(\pi_{yz}(x_i)), 2r)).
\end{align*}
By our choice of $k$, for each $i$ there are points $x_i^1,\dots, x_i^k \subset B_\L^u(f^{n_i} (\pi_{yz}(x_i)), 2r)$ with $B_\L^u(f^{n_i}(\pi_{yz}(x_i)), 2r) \subset \bigcup_{j=1}^k B_\L^u(x_i^j,r)$.  Thus $\{B_{n_i}^u(f^{-n_i}(x_i^j),r)\}_{i,j}$ is a cover of 
$\pi_{yz}(E)$, and moreover for each $i,j$, the $u$- and $\cs$-Bowen properties give
\begin{multline*}
|S_{n_i} \ph(x_i) - S_{n_i} \ph(x_i^j)| \\
\leq |S_{n_i} \ph(x_i) - S_{n_i} \ph(\pi_{yz}(x_i))| + |S_{n_i} \ph(\pi_{yz}(x_i)) - S_{n_i} \ph(x_i^j)|\leq Q_u + Q_\cs.
\end{multline*}
Now we have
\[
\sum_{i,j} e^{-n_i P(\ph)} e^{S_{n_i} \ph(x_i^j)}
\leq k e^{Q_u + Q_\cs}  \sum_i e^{-n_i P(\ph)} e^{S_{n_i} \ph(x_i)},
\]
and taking an infimum over all such covers $\{B_{n_i}^u(x_i,r)\}$ of $E$ gives
\[
m_z^\C(\pi_{yz}(E))\le k e^{Q_u + Q_\cs}m_y^\C(E).
\]
By symmetry, we also get the reverse inequality, which completes the proof (we put $\sigma=\delta$).

\section{Proof of Theorem \ref{thm:main}}\label{sec:proof-of-main}

To prove items \ref{m1}--\ref{m4} from Theorem \ref{thm:main}, first observe that each measure $\mu_n$ has $\mu_n(\L) = 1$, and thus by weak*-compactness, there  is a  subsequence $\mu_{n_k}$ that converges to an $f$-invariant limiting probability measure. 

The first step in the proof is to show that every limit measure $\mu$ has conditional measures satisfying Statement \ref{m3}, which we do in \S\ref{cond-mes}.  By Theorem \ref{thm:holonomy}, this implies that $\mu$ has local product structure, so it satisfies Statement \ref{m4} as well.

The second step is to use \eqref{eqn:mu-m} to show that every limit measure $\mu$ satisfies the Gibbs property and gives positive weight to every open set; this is relatively straightforward and is done in \S\ref{sec:Gibbs-supported}.

The third step is to use the local product structure together with a variant of the Hopf argument to show that every limit measure $\mu$ is ergodic; see \S\ref{sec:hopf}.

For the fourth step, we recall that in the setting of an expansive homeomorphism, an ergodic Gibbs measure was shown by Bowen to be the unique equilibrium measure; see \cite[Lemma 8]{rB745}.  In our setting, $f$ may not be expansive, but we can adapt Bowen's argument (as presented in \cite[Theorem 20.3.7]{Kat}) so that it only requires expansivity along the unstable direction, which still holds; see \S\ref{sec:unique}.
Once this is done, it follows that $(\L,f,\ph)$ has a unique equilibrium measure $\mu_\ph$, and that every limit measure of $\{\mu_n\}$ is equal to $\mu_\ph$.  In particular, $\mu_n$ converges to this measure as well, which establishes Statement \ref{m1} and completes the proof of Theorem \ref{thm:main}.

\subsection{Conditional measures of limit measures}\label{cond-mes}

To produce the equilibrium measure $\mu_\ph$ using the reference measure $m_x^\C$, we start by writing the measures $f_*^n m_x^\C$ in terms of \emph{standard pairs}
$(\Vl^u(y),\rho)$, where $y\in f^n(\Vl^u(x))$ and $\rho\colon \Vl^u(y) \to [0,\infty)$ is a $m_y^\C$-integrable density function; each such pair determines a measure $\rho \,dm_y^\C$ on $\Vl^u(y)$.  By controlling the density functions that appear in the standard pairs representing $f_*^n m_x^\C$, we can guarantee that every limit measure $\mu$ of the sequence of measures $\mu_n = \frac 1n \sum_{k=0}^{n-1} f_*^k m_x^\C$ has conditional measures that satisfy part \ref{3} of Theorem \ref{thm:main}.  

Fix $x\in\Lambda$ and $n\in \NN$; let $W = \Vl^u(x)$ and $W_n = f^n(W\cap \Lambda)$.
Then the iterate $f_*^n m_x^\C$ is supported on $W_n$, and $W_n$ can be covered by finitely many local leaves $\Vl^u(y)$.  Iterating the formula for the Radon--Nikodym derivative in Theorem 
\ref{thm:Gibbs}, we obtain for every $y\in W_n$ and $z\in W_n\cap \Vl^u(y)$ that 
\begin{equation}\label{eqn:RN}
\frac{d(f_*^n m_x^\C)}{d m_y^\C}(z) = e^{-nP(\ph) + S_n\ph(f^{-n}z)} =: g_n(z).
\end{equation}
Write $\rho_n^y(z) := g_n(z) / g_n(y)$; then the $u$-Bowen property gives
\begin{equation}\label{eqn:eQu}
\rho_n^y(z) = e^{S_n\ph(f^{-n} z) - S_n\ph(f^{-n} y)} \in [e^{-Q_u},e^{Q_u}].
\end{equation}
Now suppose $y_1,\dots, y_s\in W_n$ are such that the local leaves $\Vl^u(y_i)$ are disjoint.  Then for every Borel set $E\subset \bigcup_{i=1}^s \Vl^u(y_i)$, we have
\begin{equation}\label{eqn:convex}
f^n_*m_x^\C(E)=\sum_{i=1}^{s}\int_E\,g_n(z)\,d m_{y_i}^\C(z)  
=\sum_{i=1}^{s} g_n(y_i)\int_E\rho_n^{y_i}(z)\,d m_{y_i}^\C(z).
\end{equation}
In other words, one can write $f_*^n m_x^\C$ on $\bigcup_{i=1}^n \Vl^u(y_i)$ as a linear combination of the measures  $\rho_n^{y_i}\,dm_{y_i}^\C$ associated to the standard pairs $(\Vl^u(y_i), \rho_n^{y_i})$, with coefficients given by $g_n(y_i)$.  The crucial properties that we will use are the following.
\begin{enumerate}
\item The uniform bounds given by \eqref{eqn:eQu} on the density functions $\rho_n^{y_i}$ allow us to control the limiting behavior of $f_*^n m_x^\C$.
\item When $\bigcup_{i=1}^s \Vl^u(y_i)$ covers ``enough'' of $W_n$, the sum of the weights $\sum_{i=1}^s g_n(y_i)$ can be bounded away from $0$ and $\infty$.
\end{enumerate}

Given a rectangle $R$, the intersection $W_n \cap R$ is contained in a disjoint union of local leaves, so that $f_*^n m_x^\C|_R$ is given by \eqref{eqn:convex}.  We will use this to prove the following result.

\begin{lemma}\label{lem:mu-m}
If $\mu$ is any limit point of the sequence $\mu_{n}$ from \eqref{seq-meas}, then 
$\mu$ satisfies Statement \ref{3} of Theorem \ref{thm:main}: given any rectangle $R$ with $\mu(R)>0$, the conditional measures of $\mu$ are equivalent to the reference measures $m_x^\C$ and satisfy the bound 
\begin{equation}\label{eqn:mu-m2}
C_0^{-1}\leq \frac{d\mu_y^u}{dm_y^\C}(z) m_y^\C(R) \leq C_0 \text{ for $\mu_y^u$-a.e.}\ z\in V_R^u(y).
\end{equation}
\end{lemma}

Before starting the proof, we observe that although \eqref{eqn:convex} gives good control of the conditional measures of $\mu_n$, it is not in general true that the conditionals of a limit are the limits of the conditionals; that is, one does not automatically have $(\mu_n)_y^\xi \to \mu_y^\xi$ whenever $\mu_n\to \mu$.  In order to establish the desired properties for the conditional measures of $\mu$, we will need to use the fact that the conditionals of $\mu_n$ are represented by density functions for which we have uniform bounds as in \eqref{eqn:eQu}.  We will also need the following characterization of the conditional measures, which is an immediate consequence of \cite[Corollary 5.21]{EW11}.

\begin{proposition}\label{prop:conditionals}
Let $\mu$ be a finite Borel measure on $\L$ and let $R\subset \L$ be a rectangle with $\mu(R)>0$.  Let $\{\xi_\ell\}_{\ell\in \NN}$ be a refining sequence of finite partitions of $R$ that converge to the partition $\xi$ into local unstable sets $V_R^u(x) = \Vl^u(x)\cap R$.
Then there is a set $R' \subset R$ with $\mu(R')=\mu(R)$ such that for every $y\in R'$ and every continuous $\psi\colon R\to \RR$, we have
\begin{equation}\label{eqn:conditional}
\int_{V_R^u(y)} \psi(z) \,d \mu_{V_R^u(y)}^\xi(z)
= \lim_{\ell\to\infty} \frac 1{\mu(\xi_\ell(y))} \int_{\xi_\ell(y)} \psi(z)\,d\mu(z),
\end{equation}
where $\xi_n(y)$ denotes the element of the partition $\xi_\ell$ that contains $y$.
\end{proposition}

\begin{proof}[Proof of Lemma \ref{lem:mu-m}]
Note that it suffices to prove the lemma when $\diam(R)<\sigma$, where $\sigma>0$ is as in Theorem \ref{thm:holonomy}, because any rectangle $R$ can be covered by a finite number of such rectangles, and Lemma \ref{lem:compare-cond} gives the relationship between the conditional measures associated to two different rectangles.

Given a rectangle $R\subset \L$ with $\mu(R)>0$ and $\diam(R)<\sigma$, let $\xi_\ell$ be a refining sequence of finite partitions of $R$ such that for every $y\in R$ and $\ell\in \NN$, the set $\xi_\ell(y)$ is a rectangle, and $\bigcap_{\ell\in \NN} \xi_\ell(y) = V_R^u(y)$.

Let $R'\subset R$ be the set given by Proposition \ref{prop:conditionals}.  We prove that \eqref{eqn:mu-m2} holds for each $y\in R'$.  Recall that $W_k = f^k(\Vl^u(x)\cap\Lambda)$, so that $f_*^k m_x^\C$ is supported on $W_k$. Given $y\in R'$ and $\ell,k\in \NN$, the set $W_k \cap \xi_\ell(y)$ is contained in $\bigcup_{i=1}^s V_R^u(z_{k,\ell}^{(i)})$ for some $s\in \NN$ and $z_{k,\ell}^{(1)},\dots,z_{k,\ell}^{(s)} \in W_k \cap \xi_\ell(y)$.  Without loss of generality we assume that the sets $V_R^u(z_{k,\ell}^{(i)})$ are disjoint.
Following \eqref{eqn:convex}, we want to write $f_*^k m_x^\C$ as a linear combination of measures supported on these sets; the only problem is that some of these sets may not be completely contained in $W_k$.


To address this, 
let $I(k) = \{i \in \{1,\dots, s\} : V_R^u(z_{k,\ell}^{(i)}) \subset W_k\}$, and let  $\nu_{k,\ell}$ be the restriction of $m_x^\C$ to the set $\bigcup_{i\in I(k)} f^{-k}V_R^u(z_{k,\ell}^{(i)}) \subset \Vl^u(x)$.  
Then we have 
\begin{equation}\label{eqn:fk-eta}
(f_*^k m_x^\C - f_*^k \nu_{k,\ell})|_{\xi_\ell(y)}= (f_*^k m_x^\C)|_{Z_k}
\quad\text{for}\quad
Z_k := \bigcup_{i\in I(k)^c} W_k\cap\L\cap \Vl^u(z_{k,\ell}^{(i)}),
\end{equation}
and since $V_R^u(z_{k,\ell}^{(i)}) \subset B_\L^u(z_{k,\ell}^{(i)},\tau)$, we obtain that 
\[
Z_k \subset \{ z\in W_k : B_\L^u(z,2\tau) \not\subset W_k\}.
\]
Taking the preimage gives
\[
Y_k := f^{-k} Z_k \subset \{y\in \Vl^u(x) \cap \L : B_\L^u(y,2\tau \lambda^k) \not\subset \Vl^u(x)\},
\]
so $\bigcap_{k=1}^\infty Y_k = \emptyset$; we conclude that
$f_*^km_x^\C(Z_k) = m_x^\C(Y_k) \to 0$ as $k\to\infty$, so \eqref{eqn:fk-eta} gives
\[
\lim_{k\to\infty} \| (f_*^k m_x^\C - f_*^k \nu_{k,\ell})|_{\xi_\ell(y)} \| = 0.
\]
It follows that $\frac 1{n_j} \sum_{k=0}^{n_j-1} f_*^k \nu_{k,\ell}$ converges to $\mu|_{\xi_\ell(y)}$ in the weak* topology, and thus for every continuous $\psi\colon \xi_\ell(y)\to \RR$, \eqref{eqn:convex} gives
\begin{equation}\label{eqn:int-psi}
\int_{\xi_\ell(y)} \psi\,d\mu = \lim_{j\to\infty} \frac 1{n_j}
\sum_{k=0}^{n_j-1} \sum_{i\in I(k)} g_k\big(z_{k,\ell}^{(i)}\big)
\int_{V_R^u(z_{k,\ell}^{(i)})} \psi(z) \rho_k^{z_{k,\ell}^{(i)}}(z)\,d(m_{z_{k,\ell}^{(i)}}^\C)(z).
\end{equation}
Given $p,q\in R$ and a continuous function $\psi\colon R\to \RR$, 
\eqref{eqn:eQu} and Theorem \ref{thm:holonomy} give
\[
\int_{V_R^u(p)} \psi(z) \rho_k^p(z) \,dm_p^\C(z)
= e^{\pm Q_u} \int_{V_R^u(p)} \psi(z) \,dm_p^\C(z)
= e^{\pm Q_u}C^{\pm1} \int_{V_R^u(q)} \psi(\pi_{pq} z') \,dm_q^\C(z').
\]
Now assume that $\psi>0$; then when the leaves $V_R^u(p)$ and $V_R^u(q)$ are sufficiently close, we have $\psi(\pi_{pq} z') = 2^{\pm1} \psi(z')$, and thus for all sufficiently large $\ell$, \eqref{eqn:int-psi} gives
\begin{equation}\label{eqn:int-psi-2}
\int_{\xi_\ell(y)} \psi\,d\mu = (2e^{Q_u}C)^{\pm 1} \Big(\ulim_{j\to\infty} \frac 1{n_j} \sum_{k=0}^{n_j-1} \sum_{i\in I(k)} g_k(z_{k,\ell}^{(i)})\Big) \int_{V_R^u(y)} \psi\,dm_y^\C.
\end{equation}
When $\psi\equiv 1$ this gives
\[
\mu(\xi_\ell(y)) = (2e^{Q_u}C)^{\pm 1} 
\Big(\ulim_{j\to\infty} \frac 1{n_j} \sum_{k=0}^{n_j-1} \sum_{i\in I(k)}
g_k(z_{k,\ell}^{(i)})\Big) m_y^\C(V_R^u(y)),
\]
and so \eqref{eqn:conditional} yields 
\[
\int_{V_R^u(y)} \psi\,d\mu_{V_R^u(y)}^\xi = (2e^{Q_u}C)^{\pm 2} \frac 1{m_y^\C(V_R^u(y))} \int_{V_R^u(y)} \psi\,dm_y^\C.
\]
Since $\psi>0$ was arbitrary, this proves \eqref{eqn:mu-m2} and completes the proof of Lemma \ref{lem:mu-m}.
\end{proof}

\subsection{Local product structure, Gibbs property, and full support}\label{sec:Gibbs-supported}

The fact that $\mu$ has local product structure, is fully supported, and has the Gibbs property follows by the same argument as in \cite[\S6.3.2]{CPZ}; here we outline the argument and prove two Lemmas that were stated in \cite{CPZ} without proof.  We point out that although \cite{CPZ} considers uniformly hyperbolic systems, the proofs in \cite[\S6.3.2]{CPZ} work in our setting of a partially hyperbolic set satisfying \ref{C1}--\ref{C3}, with one exception:
that section also includes a proof of ergodicity using the standard Hopf argument, which requires uniform contraction in the stable direction, a strictly stronger condition than our Condition \ref{C1}.  Since we only assume \ref{C1} here, we prove ergodicity using a modified Hopf argument in \S\ref{sec:hopf}. 

Now we give the arguments.
Local product structure follows from Theorem \ref{thm:holonomy} and \eqref{eqn:mu-m}: given a rectangle $R$ with $\mu(R)>0$ and $\diam(R)<\sigma$, for $\mu$-a.e.\ $y,z\in R$ and every $A\subset V_R^u(z)$, we have
\begin{equation}\label{eqn:lps}
\mu_y^u(\pi_{zy}A) = C_0^{\pm 1} m_y^\C(\pi_{zy} A)/m_y^\C(R) = C_0^{\pm 1} C^{\pm 2} m_z^\C(A) / m_z^\C(R) = (C_0 C)^{\pm 2} \mu_z^u(A).
\end{equation}
Thus the properties in Lemma \ref{lem:lps} hold, establishing local product structure.  The proofs of full support and the Gibbs property 
use the following rectangles:
\[
R_n(x,\delta) := [\overline{B_n^u(x,\delta) \cap \L}, \overline{B_\L^\cs(x,\delta)}]
= \{[y,z] : y\in \overline{B_n^u(x,\delta)\cap \L}, z\in \overline{B_\L^\cs(x,\delta)} \}.
\]
Note that when $n=0$ we get
\[
R_0(x,\delta) = R(x,\delta) = [\overline{B_\L^u(x,\delta)},\overline{B_\L^\cs(x,\delta)}] = \{[y,z] : y\in \overline{B_\L^u(x,\delta)}, z\in \overline{B_\L^\cs(x,\delta)}\}
\]
as in \eqref{rectangle}.  The rectangles $R_n(x,\delta)$ are related to the Bowen balls $B_n(x,\delta)$ as follows.

\begin{lemma}\label{lem:BnRn}
For every sufficiently small $\delta>0$, there are $\delta_1,\delta_2>0$ such that
\begin{equation}\label{eqn:BnRn}
R_n(x,\delta_1) \subset B_n(x,\delta) \cap \L \subset R_n(x,\delta_2)
\end{equation}
for every $x\in \L$ and $n\in \NN$.  Moreover, $\delta_1,\delta_2\to 0$ as $\delta\to 0$.
\end{lemma}
\begin{proof}
By Condition \ref{C1}, there is $\delta_1>0$ such that if $\gamma$ is a curve with length $\leq \delta_1$ and $n\geq 0$ is such that $f^n\gamma$ is a $cs$-curve, then $f^n\gamma$ has length $\leq \delta/3$.  Then given $p\in R_n(x,\delta_1)$, we have $p=[y,z]$ for some $y\in \overline{B_n^u(x,\delta_1) \cap \L}$ and $z\in \overline{B_\L^\cs(x,\delta_1)}$, and thus for every $0\leq k< n$, Condition \ref{C1} gives
\[
f^k(p) = [f^k(y), f^k (z)] \in [\overline{B_\L^u(f^k(x),\delta_1)},
\overline{B_\L^\cs(f^k(x),\delta/3)}]
\subset R(f^k(x),\delta/3) \subset B(f^k(x),\delta).
\]
This proves the first inclusion in \eqref{eqn:BnRn}.  For the second inclusion, observe that since $\Vl^u(x)$ and $\Vl^\cs(x)$ depend continuously on $x$, for every sufficiently small $\delta>0$ there is $\delta_2>0$ such that $B_\L(x,\delta) \subset R(x,\delta_2)$ for all $x\in \L$, and $\delta_2 \to 0$ as $\delta \to 0$.  Then given $p\in B_n(x,\delta) \cap \L$, for each $0\leq k < n$ we have
\[
f^k(p) \in B_\L(f^k(x),\delta) \subset R(f^k(x),\delta_2),
\]
so $f^k(p) = [y_k,z_k]$ for some $y_k \in \overline{B_\L^u(f^k x,\delta_2)}$ and $z_k \in \overline{B_\L^\cs(f^k x, \delta_2)}$.  We must have $y_k = [f^k(p), f^k(x)] = f^k([p,x]) = f^k(y_0)$ for each $k$, and thus $y_0 \in \overline{B_n^u(x,\delta_2)}$, so $p\in R_n(x,\delta_2)$.
\end{proof}

\begin{lemma}[{\cite[Lemma 6.8]{CPZ}}]\label{lem:pre-gibbs}
Given $\delta>0$, there is $Q_8>0$ such that for every $x,\delta,n$ as above, we have
\begin{equation}\label{eqn:muRn}
\mu(R_n(x,\delta)) = Q_8^{\pm 1} e^{-nP(\ph) + S_n\ph(x)} \mu(R(x,\delta)).
\end{equation}
\end{lemma}
\begin{proof}
Writing $\mu_y^u$ for the conditional measures of $\mu$ on unstable leaves in 
$R(z,\delta)$, we have 
\begin{align*}
\mu(R_n&(x,\delta)) = \int_{R(x,\delta)} \mu_y^u(R_n(x,\delta)) \,d\mu(y)
=C_0^{\pm 1} \int_{R(x,\delta)} \frac{m_y^\C(R_n(x,\delta))}{m_y^\C(R(x,\delta))} \,d\mu(y) \\
&= (KC_0)^{\pm 1} \int_{R(x,\delta)} m_y^\C (\pi_{xy} \overline{B_n^u(x,\delta)}) \,d\mu(y)
= (KC_0C)^{\pm 1} m_x^\C(\overline{B_n^u(x,\delta)}) \mu(R(x,\delta)),
\end{align*}
where the first equality uses the definition of conditional measures, the second uses \eqref{eqn:mu-m}, the third uses Theorem \ref{thm:finite}, and the fourth uses Theorem \ref{thm:holonomy}.
Since $B_n^u(x,\delta) \subset \overline{B_n^u(x,\delta)} \subset B_n^u(x,2\delta)$, the result follows from the $u$-Gibbs property of $m_x^\C$.
\end{proof}

To prove full support and the Gibbs property, it is enough to show that $\inf_{x\in \L} R(x,\delta)>0$ for every $\delta>0$.  For this we need the following.

\begin{lemma}\label{lem:Rzx}
For every sufficiently small $\delta>0$, there is $\delta'>0$ such that for every $z\in \L$ and $x\in R(z,\delta')$, we have $R(z,\delta') \subset R(x,\delta)$.
\end{lemma}
\begin{proof}
As in Lemma \ref{lem:BnRn}, given $\delta>0$ small, there is $\delta'>0$ such that $B_\L(x,6\delta') \subset R(x,\delta)$ for all $x\in \L$.  Then for all $z\in \L$ and $x\in R(z,\delta')$, we have $x\in B(z,3\delta')$ and thus
\[
R(z,\delta') \subset B_\L(z,3\delta') \subset B_\L(x,6\delta') \subset R(x,\delta).\qedhere
\]
\end{proof}

\begin{lemma}\label{lem:dense-orbit}
If $y\in \L$ has a backwards orbit that is dense in $\L$, then $\mu(R(y,\theta)) > 0$ for all $\theta>0$.
\end{lemma}
\begin{proof}
Let $\delta = \delta(\theta)>0$ be given by Condition \ref{C1}, and let $\delta' = \delta'(\delta)>0$
as in Lemma \ref{lem:Rzx}.  Since $\L$ is compact, there is a finite set $E\subset \L$ such that $\bigcup_{z\in E} R(z,\delta') = \L$, and thus there is $z\in E$ with $\mu(R(z,\delta'))>0$.  Since the backwards orbit of $y$ is dense, there is $n\geq 0$ such that $x := f^{-n}(y) \in R(z,\delta')$.  By Lemma \ref{lem:Rzx} and our choice of $x$, we have
\[
\mu(R(x,\delta)) \geq \mu(R(z,\delta')) > 0.
\]
By Lemma \ref{lem:pre-gibbs}, we conclude that $\mu(R_n(x,\delta))>0$.  Moreover, we have
\[
f^n R_n(x,\delta) \subset f^n [\overline{B_n^u(x,\delta) \cap \L}, \overline{B_\L^\cs(x,\theta)}]
\subset [\overline{B_\L^u(y,\delta)}, \overline{B_\L^\cs(y,\theta)}] = R(y,\theta),
\]
where the first inclusion uses Condition \ref{C1}.  Since $\mu$ is $f$-invariant, this gives $\mu(R(y,\theta)) \geq \mu(R_n(x,\delta))>0$, and since $\delta>0$ was arbitrary, this completes the proof.
\end{proof}


Lemma \ref{lem:dense-orbit} proves that $\mu$ has full support.  The Gibbs property follows from Lemmas \ref{lem:pre-gibbs} and \ref{lem:dense-orbit}.

\subsection{Ergodicity via a modified Hopf argument}\label{sec:hopf}

In this section, we prove that if $f|\Lambda$ is topologically transitive and if $\mu$ is an $f$-invariant probability measure on $\L$ with local product structure, then $\mu$ is ergodic.

\begin{definition}\label{def:regular}
A point $z\in \L$ is \emph{Birkhoff regular} if the Birkhoff averages
\[ 
\psi^-(z) = \lim_{n\to \infty} \frac{1}{n} \sum_{k=0}^{n-1} \psi(f^{-k}z) \text{ and }\psi^+(z) = \lim_{n\to  \infty} \frac{1}{n} \sum_{k=0}^{n-1} \psi(f^kz) 
\]
are defined and equal to each other for every continuous function $\psi$ on $\L$.  In this case we write $\bpsi(z) = \psi^-(z) = \psi^+(z)$ for their common value.  The set of Birkhoff regular points is denoted $\mathcal{B}$.
\end{definition}

\begin{lemma}\label{lem:hopf}
Let $\psi\colon \L\to\RR$ be continuous.  Then
\begin{enumerate}
\item for every $x,y\in \mathcal{B}$ with $y\in \Vl^u(x)$, we have 
$\bpsi(x) = \bpsi(y)$; 
and
\item for every $\zeta>0$, there is $\epsilon>0$ such that for every $x,y\in \mathcal{B}$ with $y\in B^\cs(x,\epsilon)$, we have 
$|\bpsi(x) - \bpsi(y)| < \zeta$.
\end{enumerate}
\end{lemma}
\begin{proof}
Both statements rely on the following consequence of uniform continuity:
for every continuous $\psi\colon \L\to\RR$ and every $\zeta>0$, there is $\theta>0$ such that if $x_k,y_k\in \L$ are sequences with $\limsup_{k\to\infty} d(x_k,y_k) \leq \theta$, then $\limsup_{n\to\infty} |\frac 1n \sum_{k=1}^n f(x_k) - \frac 1n\sum_{k=1}^n f(y_k)| < \zeta$.  For the first claim in the lemma, put $x_k = f^{-k}(x)$ and $y_k=f^{-k}(y)$ so that $d(x_k,y_k)\to 0$ and $\theta>0$ can be taken arbitrarily small.
For the second claim in the lemma, use Condition \ref{C1} to get $\eps>0$ such that $y\in B^{cs}(x,\eps)$ implies $d(f^k x, f^k y) \leq \theta$ for all $k \geq 0$, so that in particular $\limsup_{k\to\infty} d(f^k(x),f^k(y)) \leq \theta$.
\end{proof}

Now let $\mu$ have local product structure, and consider the set
\begin{equation}\label{eqn:A}
A_\mu := \{x\in \mathcal{B} : \mu_x^u(\Vl^u(x) \setminus \mathcal{B} )=0 \}
\end{equation}
of all points $x$ for which $\mu_x^u$-a.e.\ point in $\Vl^u(x)$ is Birkhoff regular for 
$\mu$. By the Birkhoff ergodic theorem, we have $\mu(\mathcal{B})=1$, so 
$\mu(M\setminus \mathcal{B})=0$, and thus the disintegration into conditional measures in \eqref{gibbs-split} gives $\mu_x^u(M\setminus \mathcal{B})=0$ for $\mu$-a.e.\ $x$; in other words, $\mu(A_\mu)=1$. Thus to prove that $\mu$ is ergodic, it suffices to prove that $\bpsi$ is constant on $A_\mu$, which we do in the next lemma.

\begin{figure}[htbp]
\includegraphics[width=.6\textwidth]{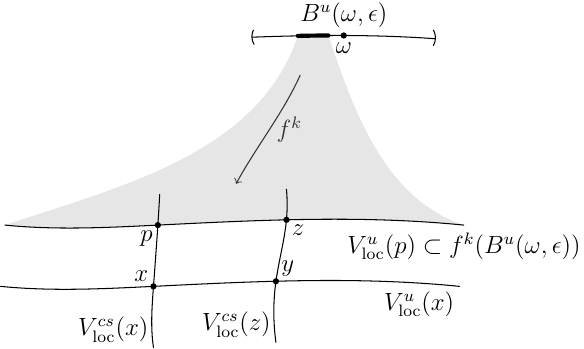}
\caption{Birkhoff averages are essentially constant.}
\label{fig:hopf}
\end{figure}

\begin{lemma}\label{lem:ae-const}
Given any $x,\omega \in A_\mu$, we have $\bpsi(\omega) = \bpsi(x)$.
\end{lemma}
\begin{proof}
Fix $\epsilon>0$ small enough that $B(x,\epsilon)\cap \L$ is contained in a rectangle $R$ on which $\mu$ has local product structure.  By Lemma \ref{lem:leaf-iterates},  there exists $k\in\mathbb{N}$ such that
\[
\tau \lambda^{k} < \epsilon/2 \text{ and }f^k(B^u(\omega,\epsilon/2))\cap B^\cs(x,\epsilon)\ne\emptyset.
\]
Let $p$ denote a point in the intersection; then $p\in B^\cs(x,\epsilon)$ and 
$f^{-k}(p)\in B^u(\omega,\epsilon)$, so by our choice of $k$ we have 
$f^{-k}(V_R^u(p)) \subset B^u(f^{-k} p,\epsilon/2) \subset B^u(\omega,\epsilon)$;  see Figure \ref{fig:hopf}. Since 
$\mu_\omega^u(B_\L^u(\omega,\epsilon)\setminus \mathcal{B})=0$, we conclude that 
$\mu_p^u(V_R^u(p)\setminus\mathcal{B})=0$ by Lemma \ref{lem:cond-inv}. Similarly, 
$\mu_x^u(V_R^u(x)\setminus\mathcal{B})=0$; since $\mu$ has local product structure, this implies that $\mu_p^u(\pi_{xp}(V_R^u(x)\setminus\mathcal{B}))=0$, and thus 
$\mu_p(V_R^u(p)\setminus (\mathcal{B}\cap\pi_{xp}\mathcal{B}))=0$. In particular, there exists $z\in V_R^u(p)\cap\mathcal{B}\cap\pi_{xp}\mathcal{B}$, so that 
$y=\pi_{px}(z)\in V_R^u(x)\cap\mathcal{B}$. Then we have
\[
\bpsi(x) = \bpsi(y) \text{ and } \bpsi(\omega) = \bpsi(f^{-k}z) = \bpsi(z)
\]
where the first two equalities use the first part of Lemma \ref{lem:hopf}. This gives 
$|\bpsi(\omega)-\bpsi(x)|=|\bpsi(z)-\bpsi(y)|$. Moreover, given any $\zeta>0$, we can use the second part of Lemma \ref{lem:hopf} to choose $\epsilon>0$ so small that 
$|\bpsi(z)-\bpsi(y)|<\zeta$. Letting $\zeta\to 0$ we conclude that 
$\bpsi(\omega)=\bpsi(x)$.
\end{proof}

\subsection{Uniqueness via Bowen's argument}\label{sec:unique}

In this section we prove the following result.

\begin{proposition}\label{prop:unique}
Let $\L,f,\ph$ be as in \S\ref{assumptions}, and let $\mu$ be an ergodic $f$-invariant probability measure on $\L$ such that the conditional measures $\mu_x^u$ are equivalent to the Carath\'eodory measures $m_x^\C$ for $\mu$-a.e.\ $x$. Then $\mu$ is the unique equilibrium measure for $\ph$.
\end{proposition}

We start by recalling some definitions and facts from \cite{HH} regarding entropy along the unstable foliation.

For a partition $\alpha$ of $\L$, let $\alpha(x)$ denote the element of $\alpha$ containing $x$. If $\alpha$ and $\beta$ are two partitions such that $\alpha(x)\subset\beta(x)$ for all $x\in\L$, we then write $\alpha\geq \beta$. For a measurable partition $\beta$, we denote $\beta_m^n=\bigvee_{i=m}^{n}f^{-i}\beta$. Take $\epsilon_0>0$ small.  Let $\mathcal{Q}=\mathcal{Q}_{\epsilon_0}$ denote the set of finite measurable partitions of $\L$ whose elements have diameters not exceeding $\epsilon_0$. For each $\beta\in\mathcal{Q}$ we define a finer partition $\eta=\mathcal{Q}^u(\beta)$ such that 
$\eta(x)=\beta(x)\cap\Vl^u(x)$ for each $x\in\L$. Let 
$\mathcal{Q}^u=\mathcal{Q}^u_{\epsilon_0}$ denote the set of all partitions obtained this way.

Given a measure $\nu$ and measurable partitions $\alpha$ and $\eta$, let
\begin{equation}\label{eqn:Hnu}
H_{\nu}(\alpha | \eta) := - \int_{\L} \log \nu_{\eta(x)}(\alpha(x))\,d\nu(x) 
\end{equation}
denote the \emph{conditional entropy of $\alpha$ given $\eta$ with respect to $\nu$}, where $\{ \nu_{\eta(x)}   \}$ is a family of (normalized) conditional measures of $\nu$
relative to $\eta$.

The \emph{conditional entropy of $f$ with respect to a measurable partition $\alpha$ given $\eta\in\mathcal{Q}^u$} is defined as  
\begin{equation}\label{eqn:cond-ent}
h_{\nu}(f, \alpha | \eta) = \limsup_{n\to \infty}  \frac{1}{n} H_{\nu}(\alpha_0^{n-1} | \eta).
\end{equation}
The following is a direct consequence of \cite[Theorem A and Corollary A.1]{HH} stated in our setting.

\begin{proposition}\label{prop:condentropy}
Suppose $\nu$ is an ergodic measure. Then for any $\alpha\in \mathcal{Q}$ and $\gamma \in \mathcal{Q}^u$ one has that
\begin{equation}\label{eqn:entropymain}
 h_\nu(f) = h_{\nu}(f, \alpha |\gamma ).
\end{equation}
\end{proposition}

We recall the following technical result.

\begin{lemma}\label{lem:condentropy}
Let $\alpha$, $\beta$, and $\gamma$ be measurable partitions with 
$H_{\nu}(\alpha|\gamma), H_{\nu}(\beta|\gamma)<\infty$.
\begin{enumerate}[label=\textup{(\arabic{*})}]
\item If $\gamma\ge\beta,$ then $H_{\nu}(\alpha|\gamma)\le H_{\nu}(\alpha|\beta)$;
\item $H_{\nu}(\alpha^{n-1}_0|\gamma)=H_{\nu}(\alpha|\gamma)+\sum_{i=1}^{n-1} H_{\nu} (\alpha |f^i (\alpha_0^{i-1}\vee\gamma)).$
\end{enumerate}
\end{lemma}
Statement $(1)$ of Lemma \ref{lem:condentropy} is well known and can be found for example in \cite{Roh67}. Statement $(2)$ is proved in \cite{HH} as Lemma 2.6(i).  We use these to prove the following lemma.
\begin{lemma}\label{lem:entropy}
For any $\alpha\in \mathcal{Q}$ and $\gamma \in \mathcal{Q}^u$ one has that
\begin{equation}\label{eqn:entropy}
h_{\mu}(f, \alpha |\gamma ) \leq H_{\mu}(\alpha | \mathcal{Q}^u(f\alpha)).
\end{equation}
\end{lemma}
\begin{proof}
We start with an observation regarding the partition $f^i(\alpha_0^{i-1}\vee\gamma)$ from Lemma \ref{lem:condentropy}(2): because $\alpha \in \mathcal{Q}$, every element of 
$\alpha_0^{i-1}$ has diameter less than $\epsilon_0$ in the dynamical metric $d_i$, and is thus contained in $B_i(x,\epsilon_0)$ for some $x$. Since $\gamma\in\mathcal{Q}^u$, we have $\gamma(x)\subset \Vl^u(x)$, and thus every element of 
$\alpha_0^{i-1}\vee \gamma$ is contained in 
$B_i(x,\epsilon_0)\cap\Vl^u(x)=B_i^u(x,\epsilon_0)$ for some $x$. It follows that every element of $f^i(\alpha_0^{i-1}\vee\gamma)$ is contained in 
$f^i(B_i^u(x,\epsilon_0))\subset B^u(f^{i-1} x,\epsilon_0)$ for some $x$.  This gives
\[
(f^i(\alpha_0^{i-1}\vee\gamma))(y) \subset (f\alpha)(y) \cap \Vl^u(y) = (\mathcal{Q}^u(f\alpha))(y),
\]
and thus $f^i(\alpha_0^{i-1}\vee\gamma) \geq \mathcal{Q}^u(f\alpha)$.
We deduce that
\begin{align*}
H_\mu(\alpha_0^{n-1}|\gamma) &= H_\mu(\alpha|\gamma) + \sum_{i=1}^{n-1} H_\mu(\alpha | f^i(\alpha_0^{i-1}\vee\gamma) ) \\
&\leq H_\mu(\alpha|\gamma) + (n-1) H_\mu(\alpha | \mathcal{Q}^u(f\alpha)),
\end{align*}
where the first line uses Statement (2) of Lemma \ref{lem:condentropy}, and the second line uses Statement (1) of that lemma.
Dividing both sides by $n$ and sending $n\to\infty$ concludes the proof of Lemma \ref{lem:entropy}.
\end{proof}
Now we prove Proposition \ref{prop:unique}.  Since every ergodic component of an equilibrium measure is itself an equilibrium measure, it suffices to prove that $\mu$ is the only \emph{ergodic} equilibrium measure for $\ph$.  Moreover, if $\nu\ll \mu$ is ergodic, then ergodicity of $\mu$ implies that $\nu=\mu$; thus it suffices to consider an ergodic measure $\nu$ such that $\nu\perp\mu$, and prove that
$P_\nu:=h_{\nu}(\ph)+\int_M\ph \,d\nu< P(\varphi)$.

\begin{figure}[htbp]
\includegraphics[width=.7\textwidth]{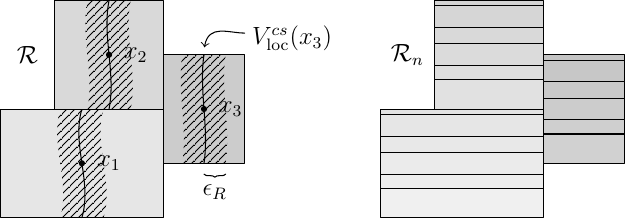}
\caption{The partitions $\mathcal{R}$ and $\mathcal{R}_n$, and the choice of $\epsilon_R$.}
\label{fig:RRn}
\end{figure}

With $\epsilon_0$ as above, let $\mathcal{R}\in \mathcal{Q}_{\epsilon_0}$ be a finite partition by $m$ rectangles as in Lemma \ref{lem:rectangle-partition}, so that each rectangle has diameter $<\epsilon_0$, is the closure of its interior, and moreover has $(\mu+\nu)$-null boundary, where we use the relative topology on $\L$.  Choose one point from the interior of each rectangle and enumerate them as $x_1,\dots, x_m\in \L$.
Fix $\epsilon_R>0$ such that
\begin{equation}\label{eqn:eps-R}
B_\L^u(y,\epsilon_R)\subset \mathcal{R}(x_i)\text{ for all }i=1,\dots,m \text{ and } y\in V_{\mathcal{R}(x_i)}^\cs(x_i),
\end{equation}
as in  Figure \ref{fig:RRn}.
As shown in that figure, for $n\in\mathbb{N}$ let $\mathcal{R}_n$ be a partition by rectangles obtained by dividing each rectangle in $\mathcal{R}$ further into rectangles in such a way that
\begin{itemize}
\item for every $R\in \mathcal{R}_n$ and $R_0\in \mathcal{R}$ with $R \subset R_0$, and every $x\in R$, we have $V_R^u(x) = V_{R_0}^u(x)$;
\item for every $R\in \mathcal{R}_n$, every $x\in R$, and every $k=0,\dots, n$, the diameter of $f^{-k}(V_R^\cs(x))$ does not exceed $\epsilon_R$.
\end{itemize} 
The first of these items guarantees that $\mathcal{Q}^u(\mathcal{R}) = \mathcal{Q}^u(\mathcal{R}_n)$ for every $n$, and the second guarantees that
the partition by rectangles $\alpha_n:= f^{-n}\mathcal{R}_n$ is also contained in $\mathcal{Q}$.  Observe that
\begin{equation}\label{eqn:hnu}
\begin{aligned}  
n h_{\nu}(f)=h_{\nu}(f^n)&=h_\nu(f^n, \alpha_n | \mathcal{Q}^u(\alpha_n))\\
&\leq H_{\nu}(\alpha_n | \mathcal{Q}^u(\mathcal{R}_n))
= H_{\nu}(\alpha_n | \mathcal{Q}^u(\mathcal{R})),
\end{aligned}
\end{equation}
where the first equality is a standard fact about entropy, the second equality uses Proposition \ref{prop:condentropy}, and the inequality uses Lemma \ref{lem:entropy}.

For every element $\eta\in\mathcal{Q}^u(\mathcal{R})$, let $\nu_{\eta}$  denote the conditional measure of $\nu$ on $\eta$ (relative to the measurable partition $\mathcal{Q}^u(\mathcal{R})$ of $\L$); we also write this as $\nu_x^u$ when $\eta$ is the partition element containing $x$. Let also $m^{\mathcal{C}}_{\eta}$ denote the Carath\'eodory measure on $\eta$. 

As in Definition \ref{def:regular}, let $\mathcal{B}$ denote the set of Birkhoff regular points, and consider as in \eqref{eqn:A} the sets
\[
A_\nu = \{x\in \mathcal{B} : \nu_x^u(\Vl^u(x) \setminus \mathcal{B}) = 0\},
\quad
A_\mu = \{x\in \mathcal{B} : \mu_x^u(\Vl^u(x) \setminus \mathcal{B}) = 0\}.
\]
The same argument as given there shows that $\nu(A_\nu) = 1$ and $\mu(A_\mu)=1$.  Let $\psi\colon \L\to\RR$ be a continuous function such that $\int\psi\,d\nu \neq \int\psi\,d\mu$, and consider the following sets of generic points:
\[
G^\psi_\nu = \bigg\{x\in \mathcal{B} : \bpsi(x) = \int \psi\,d\nu \bigg\},
\quad
G^\psi_\mu = \bigg\{x\in \mathcal{B} : \bpsi(x) = \int \psi\,d\mu \bigg\}.
\]
By the Birkhoff ergodic theorem and the fact that $\nu,\mu$ are ergodic, we have $\nu(G^\psi_\nu) = \mu(G^\psi_\mu)=1$. 

\begin{lemma}\label{lem:setB}
For every $x\in A_\nu \cap G_\nu^\psi$, writing $\eta = \eta(x)$, the measures $\nu_{\eta} = \nu_x^u$ and $m_\eta^\C$ are mutually singular, and in particular, satisfy $\nu_\eta(\eta \setminus \mathcal{B})=0$ and $m_x^\C(\eta \cap \mathcal{B})=0$.
\end{lemma}
\begin{proof}
Suppose to the contrary that $\nu_x^u$ and $m_x^\C$ are not mutually singular; then since $\nu_x^u(\eta(x) \setminus \mathcal{B}) = 0$, we must have $m_x^\C(\eta(x) \cap \mathcal{B}) > 0$.
Fix $0<\zeta < |\int\psi\,d\mu-\int\psi\,d\nu|$ and let $\epsilon>0$ be given by Lemma \ref{lem:hopf}. Since $\mu$ gives positive measure to every open set in $\L$, there exist
$p\in B(x,\epsilon) \cap A_\mu \cap G_\mu^\psi$. In particular, $\eta(p)\cap\mathcal{B}$ has full $\mu_p^u$-measure, and thus also has full $m_p^\C$-measure by \eqref{eqn:mu-m}. As in the proof of Lemma \ref{lem:ae-const} (see also Figure \ref{fig:hopf}), we conclude that  $m^{\mathcal{C}}_{\eta(p)}(\pi_{xp}(\eta(x)\cap\mathcal{B})\cap \mathcal{B})>0$, and thus there exists $y\in \mathcal{B}\cap\eta$ such that $z := \Vl^\cs(y) \cap \eta(p) \in \mathcal{B}$, and Lemma \ref{lem:hopf} gives
\[
\bpsi(x) = \bpsi(y) = \bpsi(z)\pm \zeta = \bpsi(p)\pm \zeta.
\]
Since $x\in G_\nu^\psi$ and $p\in G_\mu^\psi$, we conclude that $\int\psi\,d\nu = \int\psi\,d\mu \pm \zeta$, contradicting our choice of $\zeta$, and we conclude that $\nu_x^u$ and $m_x^\C$ are mutually singular, as claimed.
\end{proof}
We continue with the proof of uniqueness. 
By Lemma \ref{lem:setB}, for every $x\in A_\nu \cap G_\nu^\psi$, we have $\nu_{\eta(x)}(\eta(x)\cap \mathcal{B})=1$ and $m^{\mathcal{C}}_{\eta(x)}(\eta(x)\cap \mathcal{B})=0$.
Thus for every $k\in \NN$ there are sets $G_k \subset U_k \subset \eta(x)$ such that $G_k$ is compact, $U_k$ is (relatively) open, $\nu_{\eta(x)}(G_k) \to 1$, and $m_{\eta(x)}^\C(U_k) \to 0$.  Since the diameter of every element in $\alpha_n \cap \eta(x)$ goes to $0$ uniformly as $n\to\infty$, we can choose for each $n$ a 
set $C_n^{\eta(x)} \subset \eta(x)$ and a 
value of $k=k(n)$ such that
\begin{enumerate}
\item $C_n^{\eta(x)}$ is a union of elements of $\alpha_n \cap \eta(x)$, 
\item $G_k \subset C_n^{\eta(x)} \subset U_k$, and
\item $k\to \infty$ as $n\to\infty$.
\end{enumerate}
In particular, the sequence of sets $C_n^{\eta(x)}$ satisfies
\begin{equation}\label{eqn:nu+m}  
\lim_{n\to\infty} \nu_{\eta(x)}(C_n^{\eta(x)}) = 1
\quad\text{and}\quad
\lim_{n\to\infty} m_{\eta(x)}^\C(C_n^{\eta(x)}) = 0.
\end{equation}
By our construction of $\alpha_n$ and the definition of $\epsilon_R$ in \eqref{eqn:eps-R}, for every $A\in \alpha_n$ there is a point $x_A\in A$ with the property that
\begin{equation}\label{eqn:Bnuy}
B_n^u(y,\epsilon_R) \cap \L \subset A \text{ for all } y \in V_A^\cs(x_A).
\end{equation}
Given $x\in \L$ and $A\in \alpha_n$ such that the intersection
$\Vl^\cs(x_A) \cap \Vl^u(x)$ is nonempty, denote the unique point in this intersection by $x_A^{\eta(x)}$.  Note that in particular, $x_A^\eta$ is defined whenever $A\cap \eta \neq\emptyset$.  With the convention that $0\log 0=0$, \eqref{eqn:hnu} gives
\begin{align*}
nP_\nu &= n\bigg( h_\nu(f) + \int\ph\,d\nu\bigg) \\
&\leq \int_\L \Big(-\log \nu_x^u\big(\alpha_n(x)\cap\eta(x)\big)+S_n\varphi(x) \Big)\, d\nu(x) \\
&= \int_{\L/\mathcal{Q}^u(\mathcal{R})} \sum_{A\in \alpha_n} \bigg( -\nu_\eta(A\cap \eta) \log \nu_\eta(A\cap \eta) + \int_{A\cap \eta} S_n\ph \,d\nu_\eta \bigg) \,d\tilde\nu(\eta).
\end{align*}
Whenever $A\cap \eta \neq\emptyset$, the point $x_A^\eta$ exists, lies on $V_A^\cs(x_A)$, and the $u$-Bowen property gives $S_n\ph(y) \leq S_n\ph(x_A^\eta) + Q_u$ for all $y\in A\cap \eta$.  Thus we obtain
\begin{equation}\label{eqn:nPnu}
nP_\nu \leq \int_{\L/\mathcal{Q}^u(\mathcal{R})} \sum_{A\in \alpha_n} \nu_\eta(A\cap \eta) \Big( S_n\ph(x_A^\eta) -\log \nu_\eta(A\cap\eta)  \Big) \,d\tilde\nu(\eta) + Q_u.
\end{equation}
The next step is to separate the sum into two pieces, one corresponding to $A\subset C_n^\eta$ and the other to $A\not\subset C_n^\eta$, and then bound each piece from above by the following general inequality from  \cite[(20.3.5)]{Kat}, which holds for all $x_1,\dots, x_m\geq 0$ and $b_1,\dots, b_m \in \RR$:
\begin{equation}\label{eqn:xibi}
\sum_{i=1}^m x_i (b_i - \log x_i) \leq \Big( \sum_{i=1}^m x_i \Big) \log \Big( \sum_{j=1}^m e^{b_j} \Big)  + \frac 1e.
\end{equation}
Fix a choice of $\eta$ and $\alpha_n^\eta = \{A \in \alpha_n : A\cap \eta \subset C_n^\eta \}$.  Applying \eqref{eqn:xibi} to the sum over $\alpha_n^\eta$, with $\nu_\eta(A\cap \eta)$ in place of the $x_i$ terms and $S_n\ph(x_A^\eta)$ in place of the $b_i$ terms, we obtain 
\begin{equation}\label{eqn:alpha-n-eta}
\sum_{A\in \alpha_n^\eta} \nu_\eta(A\cap\eta) \Big( S_n\ph(x_A^\eta) -\log \nu_\eta(A\cap\eta)  \Big)
\leq
\nu_\eta(C_n^\eta) \log  \sum_{A\in \alpha_n^\eta} e^{S_n\ph(x_A^\eta)} + \frac 1e.
\end{equation}
Since the measures $m_x^\C$ satisfy the $u$-Gibbs property, there is a constant $Q_0$ such that $e^{S_n\ph(x)} \leq Q_0 e^{nP(\ph)} m_x^\C(B_n^u(x,\epsilon_R))$ for all $x\in \L$ and $n\in \NN$, and thus \eqref{eqn:Bnuy} gives
\[
\sum_{A\in \alpha_n^\eta} e^{S_n\ph(x_A^\eta)} \leq \sum_{A\in \alpha_n^\eta} Q_0 e^{nP(\ph)} m_\eta^\C(B_n^u(x_A^\eta,\epsilon_R))
\leq Q_0 e^{nP(\ph)} m_\eta^\C(C_n^\eta).
\]
Together with \eqref{eqn:alpha-n-eta}, we obtain
\begin{multline*}
\int_{\L / \mathcal{Q}^u(\mathcal{R})}
\sum_{A\in \alpha_n^\eta} \nu_\eta(A\cap\eta) \Big( S_n\ph(x_A^\eta) -\log \nu_\eta(A\cap\eta)  \Big) \,d\tilde\nu(\eta)\\
\leq \int_{\L / \mathcal{Q}^u(\mathcal{R})}
\Big(\nu_\eta(C_n^\eta) \big(\log m_\eta^\C(C_n^\eta) + nP(\ph) + \log Q_0\big) + \frac 1e\Big) \,d\tilde\nu(\eta).
\end{multline*}
A similar estimate holds if we replace $\alpha_n^\eta$ by $\alpha_n \setminus \alpha_n^\eta$, and thus \eqref{eqn:nPnu} gives
\begin{multline*}
nP_\nu \leq \int_{\L / \mathcal{Q}^u(\mathcal{R})} 
\Big( \nu_\eta(C_n^\eta) \log m_\eta^\C(C_n^\eta)
+ \nu_\eta(\eta \setminus C_n^\eta) \log m_\eta^C ( \eta \setminus C_n^\eta) \Big)\,d\tilde\nu(\eta) 
\\
+ nP(\ph) + \log Q_0 + \frac 2e.
\end{multline*}
Recalling \eqref{eqn:nu+m}, we see that the second term inside the integral is uniformly bounded independent of $n$, while the first term goes to $-\infty$ as $n\to\infty$ for every $\eta$; we conclude that the integral goes to $-\infty$ as $n\to\infty$, and thus
\[
\lim_{n\to\infty} \Big(n(P_\nu - P(\ph)) - \log Q_0 - \frac 2e\Big) = -\infty.
\]
This implies that $P_\nu < P(\ph)$, and thus $\nu$ is not an equilibrium measure for $\ph$.  We conclude that $\mu$ is the unique equilibrium measure for $\ph$, as claimed.

\appendix

\section{Lemmas on rectangles and conditional measures}\label{appendix}

\begin{proof}[Proof of Lemma \ref{lem:rectangle-partition}]
Let $\epsilon>0$ be small enough that $R(x,\delta)$ as in \eqref{rectangle} is defined for all $x\in \L$ and $\delta \in (0,\epsilon)$.  Then given $x\in \L$, the function $\delta \mapsto \mu(R(x,\delta))$ is monotonic on $(0,\epsilon)$, and hence continuous at all but countably many values of $\delta$.  Let $\delta(x)$ be a point of continuity; then $\mu(\partial R(x,\delta))=0$.  Note that $\{\Int R(x,\delta(x)) : x\in \L\}$ is an open cover for the compact set $\L$, so there are $x_1,\dots, x_n\in \L$ such that writing $\delta_i = \delta(x_i)$, the rectangles $\{R(x_i,\delta_i)\}_{i=1}^n$ cover $\L$.  Given $I \subset \{1,\dots, n\}$, the set $R_I = \bigcap_{i\in I} R(x_i,\delta_i)$ is either empty or is itself a rectangle that is the closure of its interior and has $\mu$-null (relative) boundary.  Taking $\mathcal{I}$ to be the collection of all subsets $I$ of $\{1,\dots, n\}$ for which $R_I$ is nonempty, the set of rectangles $\{R_I : I\in \mathcal{I}\}$ satisfies the properties required by the lemma.
\end{proof}

\begin{proof}[Proof of Lemma \ref{lem:compare-cond}]
First we prove the lemma when $R_1 \subset R_2$.  In this case every $\psi \in L^1(R_1,\mu)$ extends to a function in $L^1(R_2,\mu)$ by setting it to $0$ on $R_2\setminus R_1$; then fixing $x\in R_2$ and defining measures $\tilde\mu^i$ on $V_{R_i}^\cs(x)$ by
\[
\tilde\mu^i(A) = \mu \Bigg(\bigcup_{y\in A} V_{R_i}^u(y)\Bigg),
\]
we see from \eqref{eqn:R-split} that
\begin{equation}\label{eqn:R1R2}
\begin{aligned}
\int_{R_1} \psi\,d\mu &= \int_{R_2} \psi\,d\mu = \int_{V_{R_2}^\cs(x)} \int_{V_{R_2}^u} \psi\,d\mu_y^2 \,d\tilde\mu^2(y) \\
&= \int_{V_{R_1}^\cs(x)} \int_{V_{R_1}^u} \psi \,d\mu_y^2 \,d\tilde\mu^2(y),
\end{aligned}
\end{equation}
where the last equality uses the fact that $\psi$ vanishes on $R_2\setminus R_1$.  Let $N = \{y\in R_1 : \mu_y^2(R_1) = 0\}$; then $N$ is a union of unstable sets in $R_1$ and thus
\[
\tilde \mu^1(N) = \mu(N) = \int_{N\cap V_{R_1}^\cs(x)} \mu_y^2(R_1) \,d\tilde\mu^2(y) = 0,
\]
so the function $c(y) = \mu_y^2(R_1)$ is positive $\mu$-a.e.\ on $R_1$, and
$\tilde\mu^1$-a.e.\ on $V_{R_1}^\cs(x)$.  Given $A\subset V_{R_1}^\cs(x)$, we have
\[
\tilde\mu^1(A) = \mu\Bigg(\bigcup_{y\in A} V_{R_1}^u(y)\Bigg)
= \int_A \mu_y^2(V_{R_1}^u(y)) \,d\tilde\mu^2(y)
= \int_A c(y) \,d\tilde \mu^2(y);
\]
we conclude that $\tilde\mu^1 \ll \tilde\mu^2$ with Radon--Nikodym derivative given by $c$.  Since $c$ is positive $\tilde\mu^1$-a.e., we conclude that $\tilde\mu^2 \ll \tilde\mu^1$ with derivative $1/c$, so \eqref{eqn:R1R2} gives
\[
\int_{R_1} \psi\,d\mu = \int_{V_{R_1}^\cs(x)} \int_{V_{R_1}^u} \frac{\psi}{c(y)} \,d\mu_y^2 \,d\tilde\mu^1(y),
\]
By a.e.-uniqueness of the system of conditional measures, this shows that $\mu_y^1 = \frac 1{c(y)}\mu_y^2|_{V_{R_1}^u(y)}$ for $\mu$-a.e.\ $y\in R_1$.

Now consider two arbitrary rectangles $R_1,R_2$.  The set $R_3 = R_1 \cap R_2$ is also a rectangle, and by the argument given above, given $x\in R_3$, the functions $c_i\colon V_{R_3}^\cs(x) \to [0,1]$ defined by $c_i(y) = \mu_y^2(R_3)$ are positive $\tilde\mu^i$-a.e., and $\mu_y^3 = \frac 1{c_i(y)} \mu_y^i|_{V_{R_3}^u(y)}$ for $\mu$-a.e.\ $y\in R_3$, from which we conclude that
\[
\mu_y^1|_{V_{R_3}^u(y)} = \frac{c_1(y)}{c_2(y)} \mu_y^2|_{V_{R_3}^u(y)},
\]
completing the proof of Lemma \ref{lem:compare-cond}.
\end{proof}

\begin{proof}[Proof of Lemma \ref{lem:cond-inv}]
From Lemma \ref{lem:compare-cond}, it suffices to prove the result for a single choice of rectangles.  Thus we let $R$ be any rectangle containing $x$, small enough that $f(R)$ is also a rectangle.  Then since $\mu$ is $f$-invariant, for every $\psi\in L^1(f(R),\mu)$, we have $\psi\circ f \in L^1(R,\mu)$, and \eqref{eqn:conditionals} gives 
\begin{align*}
\int_{f(R)} \psi \,d\mu &= \int_R \psi\circ f \,d\mu
= \int_R \int_{V_R^u(x)} \psi\circ f \,d\mu_x^u \,d\mu(x) \\
&= \int_R \int_{f(V_R^u(x))} \psi \,d(f_* \mu_x^u) \,d\mu(x)
= \int_{f(R)} \int_{V_R^u(y)} \psi\,d(f_* \mu_x^u) \,d\mu(y).
\end{align*}
By uniqueness of the system of conditional measures, this shows that for the rectangles $R$ and $f(R)$, we have $f_* \mu_x^u = \mu_{f(x)}^u$ for $\mu$-a.e.\ $x\in R$, and by Lemma \ref{lem:compare-cond}, this completes the proof of Lemma \ref{lem:cond-inv}.
\end{proof}

\begin{proof}[Proof of Lemma \ref{lem:lps}]
Since $y,z$ play symmetric roles in \ref{cs-ac}, we conclude that \ref{cs-ac} and \ref{cs-sim} are equivalent, and similarly for \ref{u-ac} and \ref{u-sim}.  We prove that \ref{cs-sim} is equivalent to \ref{prod-sim-0} and \ref{prod-sim}; the proof for \ref{u-sim} is similar.
Clearly \ref{prod-sim} implies \ref{prod-sim-0}.  We prove that \ref{cs-sim} implies \ref{prod-sim-0}, and then that \ref{prod-sim-0} implies both \ref{cs-sim} and \ref{prod-sim}.

First suppose that \ref{cs-sim} holds, fix $p\in R$ such that $(\pi_{pz}^\cs)_* \mu_p^u \sim \mu_z^u$ for $\mu$-a.e.\ $z\in R$, and define $h\colon R\to (0,\infty)$ by
\begin{equation}\label{eqn:h}
h(z) = \frac{d\mu_z^u}{d(\pi_{pz}^\cs)_*\mu_p^u}(z).
\end{equation}
Define $\tilde\mu_p$ on $V_R^\cs(p)$ by $\tilde\mu_p(E) = \mu(\bigcup_{y\in E} V_R^u(y))$, so that \eqref{eqn:R-split} and \eqref{eqn:h} give
\begin{multline}\label{eqn:intpsi}
\int\psi\,d\mu = \int_{V_R^\cs(p)} \int_{V_R^u(y)} \psi(z) \,d\mu_y^u(z) \,d\tilde\mu_p(y) \\
= \int_{V_R^\cs(p)} \int_{V_R^u(y)} \psi(z) h(z) \,d(\pi_{zp}^* \mu_p^u)(z) \,d\tilde\mu_p(y)
= \int_R \psi(z) h(z) \,d(\mu_p^u\otimes \tilde\mu_p)(z).
\end{multline}
This proves \ref{prod-sim-0}.  Now we suppose that \ref{prod-sim-0} holds, so that writing
\begin{equation}\label{eqn:h2}
h(z) = \frac{d\mu}{d(\tilde\mu_p^u \otimes \tilde\mu_p^\cs)}(z),
\end{equation}
we have
\begin{equation}\label{eqn:intpsi-2}
\int\psi\,d\mu = \int_{V_R^\cs(p)} \int_{V_R^u(y)} \psi(z) h(z) \,d(\pi_{yp}^* \tilde\mu_p^u)(z) \,d\tilde\mu_p^\cs(y).
\end{equation}
Recall that $\mu_y^u$ is uniquely determined (up to a scalar) for $\mu$-a.e.\ $y$ by the condition that there is a measure $\nu$ on $V_R^u(p)$ with
\[
\int\psi\,d\mu = \int_{V_R^\cs(p)} \int_{V_R^u(y)} \psi(z) \,d\mu_y^u(z) \,d\nu(y)
\]
for every integrable $\psi$.  Comparing this to \eqref{eqn:intpsi-2} we conclude that $\mu_y^u \sim \pi_{yp}^* \tilde\mu_p^u$ for $\mu$-a.e.\ $y\in R$, which establishes both \ref{cs-sim} and \ref{prod-sim}.
\end{proof}

\bibliographystyle{amsplain}
\bibliography{ph-car-ref}
\end{document}